\documentclass[12pt]{article}
\usepackage{amsthm,amscd}
 \usepackage{txfonts}

\usepackage{amsmath}
\usepackage{mathrsfs}
\usepackage{enumerate}
\newenvironment{tightenumerate}{
\begin{enumerate}[(1)]
  \setlength{\itemsep}{3pt}
  \setlength{\parskip}{0pt}
}{\end{enumerate}}
\usepackage[english]{babel}
\usepackage{color}
\usepackage{hyperref}
\usepackage{lastpage}
\usepackage{fancyhdr}
\usepackage{geometry}
\usepackage{graphicx}
\hypersetup{
    colorlinks,
    citecolor=black,
    filecolor=black,
    linkcolor=black,
    urlcolor=black
}
\usepackage{fancyhdr}
\usepackage{tocloft}

\usepackage{mathtools}
\mathtoolsset{showonlyrefs}
\allowdisplaybreaks[1]

\numberwithin{equation}{section}
\theoremstyle{plain}
	\newtheorem{theorem}{Theorem}[section]           % Theorem environment
	\newtheorem{corollary}[theorem]{Corollary}             % Corollary environment
	\newtheorem{lemma}[theorem]{Lemma}                 % Lemma environment
	\newtheorem{proposition}[theorem]{Proposition}          % Proposition environment
	\newtheorem{assumption}{Assumption}[section]  
\theoremstyle{definition}
	\newtheorem{definition}[theorem]{Definition}                 % Definition environment
\theoremstyle{remark}
	\newtheorem{remark}[theorem]{Remark}                % Remark environment
\newcommand\bR{\mathbf{R}}
\newcommand\bE{\mathbf{E}}
\newcommand\bF{\mathbf{F}}
\newcommand\bP{\mathbf{P}}

\newcommand\bN{\mathbf{N}}

\newcommand\cB{\mathcal{B}}
\newcommand\cC{\mathcal{C}}

\newcommand\cE{\mathcal{E}}
\newcommand\cF{\mathcal{F}}

\newcommand\cI{\mathcal{I}}

\newcommand\cL{\mathcal{L}}

\newcommand\cN{\mathcal{N}}
\newcommand\cO{\mathcal{O}}
\newcommand\cP{\mathcal{P}}

\newcommand\cX{\mathcal{X}}
\newcommand\cZ{\mathcal{Z}}
\newcommand\cU{\mathcal{U}}

\title{\vspace{-1.2in}On Classical Solutions of Linear Stochastic Integro-Differential Equations\vspace{-0.5cm}}

\begin{document}

\maketitle

\begin{tabular}{l}
{\large\textbf{James-Michael Leahy}} \vspace{0.2cm}\\
 The University of Edinburgh, E-mail: \href{J.Leahy-2@sms.ed.ac.uk}{J.Leahy-2@sms.ed.ac.uk}\vspace{0.4cm} \\
{\large \textbf{Remigijus Mikulevi\v{c}ius}}\vspace{0.2cm}\\
 The University of Southern California, E-mail: \href{Mikulvcs@math.usc.edu }{mikulvcs@math.usc.edu}\vspace{0.4cm}\\
 {\large\textbf{Abstract}} \vspace{0.2cm} \\
 \begin{minipage}[t]{0.9\columnwidth}%
 We prove the existence of classical solutions to parabolic linear stochastic
 integro-differential equations with  adapted coefficients using
 Feynman-Kac transformations, conditioning, and the interlacing of
 space-inverses of stochastic flows associated with the equations. The
 equations are forward and the derivation of existence does not use the
 ``general theory" of SPDEs. Uniqueness is proved in the class of classical
 solutions with polynomial growth.
 \end{minipage}
\end{tabular}

\tableofcontents
\thispagestyle{empty}

\section{Introduction}

Let $\left( \Omega ,\cF,\mathbf{P}\right) $ be a complete filtered
probability space and $\tilde{\cF}_0$ be a sub-sigma-algebra of $\cF$.  We assume that this probability space supports a sequence $(w_{t}^{1;\varrho})_{\varrho\ge 1}$, $t \ge 0$, $\varrho\in \bN$,  of  independent  one-dimensional Wiener processes and a  Poisson random measure $p^{1}(dt,dz )$ on $ (\bR_+
\times Z^{1},$ $ \mathcal{B}(\bR_+ \otimes \mathcal{Z}^{1} )$ with intensity measure $\pi^{1}(dz)dt$, where  $(Z^{1}, \mathcal{Z}^{1},\pi^{1})$ is a
sigma-finite measure space.  We also assume that $(w_{t}^{1;\varrho})_{\varrho\ge 1}$ and $p^{1}(dt,dz )$ are independent of $\cF_0$. 
Let  $\mathbf{F}=(\cF_{t})_{t\ge 0}$ be
the standard augmentation of the filtration $( \bar{\cF}
_t)_{t\ge 0}$, where for each $t\ge 0$,
\begin{equation*}
\bar{\cF}_t=\sigma \left(\tilde{\cF}_0,(w^{1}_{s})_{\varrho\ge 1},\;p^{1}( [0,s],\Gamma)
:s\le t, \;\Gamma \in \mathcal{Z}^{1}\right).
\end{equation*}%
For each real number $T>0$, we let  $\mathcal{R}_T$, $\mathcal{O}_T$,  and  $\mathcal{P}%
_T $ be the $\mathbf{F}$-progressive, $\mathbf{F}$-optional,  and $\mathbf{F}$-predictable
sigma-algebra on $\Omega\times [0,T] $, respectively. Denote by
$
q^{1}(dt,dz )=p^{1}(dt,dz) -\pi^{1}(dz)dt
$
the compensated Poisson random measure.  Let $D^{1},E^{1},V^{1}\in Z$ be disjoint  $\cZ^{1}$-measurable subsets such that $D^{1}\cup E^{1}\cup V^{1}=Z^{1}$ and  $%
\pi(V^{1})<\infty$.  Let $( Z^{2},\mathcal{Z}^{2},\pi^{2})$ be a sigma-finite measure space and $ D^{2}, E^{2}\in  Z^{2}$ be disjoint $ \cZ^{2}$-measurable subsets such that $ D^{2}\cup E^{2}= Z^{2}$.\\ 
\indent Fix an arbitrary positive real number $T>0$ and integers $d_1,d_2\geq 1$. Let $%
\alpha \in ( 0,2]$ and let $\tau\le T$ be a stopping time. Let $\cF%
_{\tau}$ be the stopping time sigma-algebra associated with $\tau$ and let $%
\varphi:\Omega \times \bR^{d_1}\rightarrow \bR^{d_2}$ be $\cF%
_{\tau }\otimes \mathcal{B}(\bR^{d_1})$-measurable. We consider the
system of stochastic integro-differential equations  on $[0,T]\times 
\bR^{d_1}$ given by 
\begin{align}\label{eq:FullSIDE} 
du^l_t &=\left(( \mathcal{L}^{1;l}_{t}+\mathcal{L}^{2;l}_t)u_t+\mathbf{1}_{[1,2]}(\alpha)b^i_t\partial_iu_t^{l}+c^{l\bar l}_tu_t^{\bar l}+f^l_t\right)dt+\left(\mathcal{N}^{1;l\varrho}_tu_t+g^{l\varrho }_t\right)dw^{1;\varrho}_t\notag \\
&\quad +\int_{Z^{1}}\left(\mathcal{I}^{1;l}_{t,z }u_{t-}+h^l_t(z)\right)[\mathbf{1}_{D^{1}}(z)q^{1}(dt,dz )+\mathbf{1}_{E^{1}\cup V^{1}}(z)p^{1}(dt,dz)],\;\;\tau
\leq t\leq T, \notag  \\
u^l_t&=\varphi^l ,\;\;t\leq \tau , \;\;l\in\{1,\ldots,d_2\}, 
\end{align}%
where for $\phi \in C_{c}^{\infty }( \bR^{d_1};\mathbf{R}
^{d_2})$, $k\in\{1,2\},$ and $l\in\{1,\ldots,d_2\}$,
\begin{align*}
\mathcal{L}^{k;l}_{t}\phi(x) :&=\mathbf{1}_{\{2\}}(\alpha)\frac{1}{2}\sigma_t
^{k;i\varrho}(x)\sigma_t
^{k;j\varrho}(x)\partial_{ij}\phi^l(x)+\mathbf{1}_{\{2\}}(\alpha)\sigma^{k;i\varrho}_t(x)\upsilon^{k;l\bar l\varrho}_t(x)\partial_i\phi^{\bar l}(x)\\
&\quad  +\int _{D^{k}}\rho^{k;l\bar l}_t(x,z )\left(\phi^{\bar l} (x+H^{k}_t(x,z ))-\phi^{\bar l} (x)\right)\pi^{k}(dz ) \\
&\quad +\int _{D^{k}}\left(\phi^{ l} (x+H^{k}_t(x,z ))-\phi^{ l} (x)-%
\mathbf{1}_{(1,2]}(\alpha)H^{k;i}_t(x,z )\partial_i\phi^l(x)\right)\pi^{k}(dz ) \\
&\quad +\mathbf{1}_{\{2\}}(k)\int _{E^{2}}\left((I_{d_2}^{l\bar l}+\rho^{2;l\bar l}_{t}(x,z ))\phi^{\bar l} (x+H^{2}_t(x,z
))-\phi^l (x)\right)\pi^{2}(dz), \\
\mathcal{N}^{1;l\varrho}_t\phi(x):&=\mathbf{1}_{\{2\}}(\alpha)\sigma_t^{1;i\varrho}(x)\partial_i\phi^l(x)+\upsilon^{1;l\bar l\varrho}_t(x)\phi^{\bar l}(x), \;\;\varrho\ge 1, \\
\mathcal{I}^{1;l}_{t,z }\phi(x) :&=(I_{d_2}^{l\bar l}+\rho^{1;l\bar l}_{t}(x,z ))\phi^{\bar l} (x+H^{1}_t(x,z
))-\phi^l (x),
\end{align*}
and 
$$
\int_{D^{k}}\left(|H^{k}_t(x,z)|^{\alpha}+|\rho^{k}_t(x,z)|^2\right)\pi^{k}(dz)+\int_{E^{k}}\left(|H^{k}_t(x,z)|^{1\wedge \alpha}+|\rho^{k}_t(x,z)|\right)\pi^{k}(dz)<\infty.
$$
 The summation convention with respect to repeated indices $i,j\in \{1,\ldots,d_1\}$, $\bar l\in \{1,\ldots,d_2\}$, and $\varrho\in \bN$ is  used here and below. 
The $d_2\times d_2$ dimensional identity matrix is denoted by $I_{d_2}$.
For a subset $A$ of a larger set $X$, $\mathbf{1}_{A}$ denotes the  $\{0,1\}$%
-valued function taking the value $1$ on the set $A$ and $0$ on the
complement of $A$. We assume that for each $k\in\{1,2\}$,
$$
\sigma_{t}^{k}(x)=(\sigma^{k;i\varrho}_{t}(\omega,x))_{1\le i\le d_1,\;\varrho\ge 1},\;\; b_t(x)=(b^i_t(\omega,x))_{1\le i\le d_1},\;\; c_t(x) =( c^{l\bar l}_t(\omega,x))_{1\leq l,\bar l\le d_2},
$$
$$
\upsilon^{k}_t(x)=(\upsilon_t^{k;l\bar l\varrho}(\omega,x))_{1\le l,\bar l\le d_2,\;\varrho\ge 1},\;\; f_t(x) =(f^{i}_t(\omega,x))_{1\leq i\le d_2},\;\;   g_t(x)=(g^{i\varrho}_t(\omega,x))_{1\le i\le d_2,\;\varrho\ge 1},
 $$
 are   random fields  on $\Omega\times [0,T]\times \bR^{d_1}$  that are $\mathcal{R}_T\otimes 
\mathcal{B}(\bR^{d_1})$-measurable. Moreover, for each $k\in\{1,2\}$, we assume  that
$$
H^{k}_{t}(x,z)=(H^{k;i}_{t}(\omega,x,z))_{1\le i\le d_1}, \;
\rho^{k}_{t}(x,z )=( \rho ^{k;l\bar l}_{t}(\omega,x,z )) _{1\leq l,\bar l\le d_2},$$
are random fields on $\Omega\times [0,T]\times \bR^{d_1}\times Z^{k}$ that are $\mathcal{P}%
_{T}\otimes \mathcal{B}(\bR^{d_1})\otimes \mathcal{Z}^{k}$-measurable. Moreover,  we assume that $$h_t(x,z) =(
h^{i}_t(\omega,x,z))_{1\leq i\le d_2},
$$
 is a  random field  on $\Omega\times [0,T]\times \bR^{d_1}$  that is $\mathcal{P}_T\otimes 
\mathcal{B}(\bR^{d_1})$-measurable \\
\indent Systems of linear stochastic integro-differential equations appear in many contexts.  They may be considered as extensions of both first-order symmetric hyperbolic  systems and  linear fractional  advection-diffusion equations.  The equation \eqref{eq:FullSIDE}   also arises in 
non-linear filtering of semimartingales as the equation for the unormalized
filter of the signal (see, e.g.,  \cite{Gr76} and \cite{GrMi11}).  Moreover,  \eqref{eq:FullSIDE}  is intimately related to linear transformations of inverse flows of jump SDEs and it is precisely this connection that we will exploit to obtain solutions.\\
\indent There are various  techniques available to derive the existence and uniqueness of classical solutions of linear  parabolic SPDEs and SIDEs. One approach is to  develop a theory of weak solutions for the equations (e.g. variational, mild solution, or etc...) and then study further regularity in classical function spaces  via an embedding theorem. We refer the reader to  \cite{Pa72, Pa75,  MePi75, KrRo77, Ti77b, Gy82, Wa86,  DaZa92, Kr99, ChKi10,PeZa07, Ha05, RoZh07,  BrNeVeWe08, HoOkUbZh10, LeMi14a} for more information about weak solutions of SPDEs driven by continuous and discontinuous martingales and martingale measures.  This approach is especially important in the non-degenerate setting where some smoothing occurs and has the obvious advantage that it is broader in scope. Another approach is to regard the solution as a function with values in a probability space and use the method deterministic PDEs (i.e. Schauder estimates, see, e.g. \cite{Mi00, MiPr09a}). A third approach is a direct one that uses solutions of stochastic differential equations. The direct method allows  to obtain classical solutions in the entire H{\"o}lder scale while not restricting to integer derivative assumptions for the coefficients and data.  \\
\indent  In this paper, we derive the existence  of a classical solutions of \eqref{eq:FullSIDE} with regular coefficients using a Feynman-Kac-type
transformation and the interlacing of the space-inverse (first integrals 
\cite{KrRo81}) of a stochastic flow associated with the equation. 
The construction of the solution gives an insight into the structure of the
solution as well. We prove that the solution of \eqref{eq:FullSIDE} is unique in the class
of classical solutions  with polynomial growth (i.e. weighted H{\"o}lder spaces). As an immediate corollary of
our main result, we obtain the existence and uniqueness of classical
solutions of linear  integro-differential equations with random coefficients, since the
coefficients $\sigma^1$, $H^1$, $a^1$, $\rho^1$, and free terms $g$ and $h$ can be zero.  Our work here directly extends  the method of characteristics for deterministic first-order partial differential equations and  the well-known Feynman-Kac  formula for deterministic second-order PDEs. \\
\indent In the continuous case (i.e. $H^1\equiv 0, H^{2}\equiv 0, h\equiv 0$), the classical solution of \eqref{eq:FullSIDE} 
was constructed in \cite{KrRo81,Ku81,Ku86,Ro90} (see references therein as well) using the first
integrals of the associated backward SDE. This method was also  used to
obtain classical solutions of (\ref{eq:FullSIDE}) in \cite{DaMeTu07}. In the
references above, the forward Liouville equation for the first integrals of
associated stochastic flow was derived directly. However, since the backward
equation involves a time reversal, the coefficients and input functions are
assumed to be non-random. The generalized solutions of (\ref{eq:FullSIDE})
with $d_2=1$, non-random coefficients, non-degenerate diffusion, and finite
measures $\pi ^{1}=\pi ^{2}$ were discussed in \cite{Me07}.   
In this paper, we give a direct derivation of (\ref{eq:FullSIDE})
and all the equations considered are forward, possibly degenerate, and the
coefficients and input functions are  adapted.  For other interesting and related developments, we refer the reader to \cite{Pr12,Zh13b,Pr14}.\\
\indent This paper is organized as follows. In Section 2, our notation is set forth
and the main results are stated.  In Section 3, the main theorem is proved
and is divided into a proof of uniqueness and existence. In Section 4, the appendix,
auxiliary facts that are used throughout the paper are discussed.

\section{Outline of main results}

For each integer $n\ge 1$, let $\mathbf{R}^{n}$  be the space of $d$-dimensional Euclidean points $x=(x^1,\ldots,\allowbreak x^{n})$. For each $x$, denote by $|x|$ the Euclidean norm of $x$. Let $\bR_+$ denote the set of non-negative real-numbers. Let $\bN$ be the set of natural numbers.  
Elements of $\bR^{d_1}$ and $\bR^{d_2}$ are  understood as column vectors  and elements of $\mathbf{R}^{2d_1}$  and $\mathbf{R}^{2d_2}$ are understood as matrices of dimension $d_1\times d_1$ and $d_2\times d_2$, respectively.  For each integer $n\ge 1$, the norm of an element $x$ of  $\ell_2(\mathbf{R}^{n})$, the space of square-summable $\mathbf{R}^n$-valued  sequences,  is  denoted by $|x|$. For a topological space $(X,\cX)$ we denote the Borel sigma-field on $X$ by $\cB(X)$.

For each $i\in \{1,\ldots,d_1\}$, let $\partial_i=\frac{\partial}{\partial x_i}$ be
the spatial derivative operator with respect to $x_i$ and write  $\partial_{ij}=\partial_i\partial_j$ for each $i,j\in \{1,\ldots,d_1\}$.  For a once differentiable function $%
f=(f^1\ldots,f^{d_1}):\bR^{d_1}\rightarrow\bR^{d_1}$, we denote the gradient  of $f$ by $\nabla
f=(\partial_jf^i)_{1\le i,j\le d_1}$.  Similarly, for a once differentiable function $f=(f^{1\varrho},\ldots,f^{d\varrho})_{\varrho\ge 1} : \bR^{d_1}\rightarrow \ell_2(\bR^{d_1})$, we denote the gradient of $f$ by $\nabla f=(\partial_jf^{i\varrho})_{1\le i,j\le d_1,\varrho\ge 1} $ and understand it as a function from $\bR^{d_1}$ to $\ell_2(\mathbf{R}^{2d_1})$.
For a multi-index $%
\gamma=(\gamma_1,\ldots,\gamma_d)\in\{0,1,2,\ldots,\}^{d_1}$ of length $%
|\gamma|:=\gamma_1+\cdots+\gamma_d$, denote by $\partial^{\gamma}$ the
operator $\partial^\gamma=\partial_1^{\gamma_1}\cdots \partial_d^{\gamma_d}$, where $\partial_i^0$ is the identity operator for all $i\in\{1,\ldots,d_1\}$. For each integer $d\ge 1$, we denote by $C_c^{\infty}(\bR^{d_1}; \bR^{d})$ the space of infinitely differentiable functions with compact support in $\bR^{d}$. 

\indent For a Banach space $V$ with norm $|\cdot |_{V}$, domain $Q$ of $%
\mathbf{R}^{d}$, and continuous function $f:Q\rightarrow V$,  we define 
$$
|f|_{0;Q;V}=\sup_{x\in Q}|f(x)|
$$
and
$$
[f]_{\beta;Q;V}=\sup_{x,y\in Q,x\neq y}\frac{|
f(x)-f(y)|_{V}}{|x-y|_{V}^{\beta }},\;\;\beta \in (0,1].$$
For each real number $\beta\in \mathbf{R}$, we write  $\beta =[\beta]^-+\{\beta\}^+$, and $\{\beta\}^+\in (0,1]$.   For a Banach space $V$ with norm $|\cdot |_{V}$,  real number $\beta>0$, and domain $Q$ of $%
\mathbf{R}^{d}$, we denote by $\cC^{\beta }(Q;V)$ the Banach space of
all bounded continuous functions $f:Q\rightarrow V$ having finite norm 
\begin{equation*}
|f|_{\beta ;Q;V}:=\sum_{| \gamma |\leq [\beta ]^-
}|\partial^{\gamma }f|_{0;Q;V}+\sum_{|\gamma|=[\beta]^-}[\partial^{\gamma}f]_{\{\beta\}^+ ;Q;V}.
\end{equation*}%
When $Q=\mathbf{R}^{d_1}$ and $V=\mathbf{R}^n$ or $V=\ell_2(\bR^n)$ for any integer $n\ge 1$, we drop the subscripts $Q$  and $V$ from the norm  $| \cdot |_{\beta;Q;V}$ and write $|\cdot |_{\beta}
$.   For a Banach space $V$ and for each $\beta>0$, denote by  $%
\cC_{loc}^{\beta}(\bR^d;V)$   the Fr\'echet space   of   continuous functions   $f:\mathbf{R}^d\rightarrow V$ satisfying $f\in \cC^{\beta}(Q;V)$ for all
bounded domains $Q\subset \mathbf{R}^{d}$. We call a function $f:\mathbf{R}%
^{d}\rightarrow \mathbf{R}^{d} $ a $\cC_{loc}^{\beta}(\bR^d;\bR^d)$-diffeomorphism  if $f$
is a homeomorphism and both $f$ and its inverse $f^{-1}$ are in $\cC_{loc}^{\beta}(\bR^d;\bR^d)$. \\
\indent For a Fr\'echet space $\chi$, we denote by $D([0,T];\chi)$ the space of $\chi$%
-valued c\`{a}dl\`{a}g functions on $[0,T]$.  Unless otherwise specified,  we endow   $D([0,T];\chi)$  with the supremum semi-norms.  

The notation $N=N(\cdot ,\cdots,\cdot )$ is used to denote a positive
constant depending only on the quantities appearing in the parentheses. In a
given context, the same letter is often used to denote different constants
depending on the same parameter. If we do not specify to which space  the parameters $%
\omega ,t,x,y,z$ and $n$ belong, then we  mean $\omega \in \Omega $, $%
t\in [ 0,T]$, $x,y\in \bR^{d_1}$, $z\in Z^{k}$, and $n\in\mathbf{N}$. 

Let $r_{1 }(x):=\sqrt{ 1+|
x|^{2}},\;x\in \bR^{d_1}$. Let us introduce  some regularity conditions on the coefficients and free terms.  We consider these assumptions for $\bar{\beta}>1\vee \alpha$ and $\tilde{\beta}>\alpha$.

\begin{assumption}[$\bar\beta$] \label{asm:regcoeffclassicalwcorrec}
\begin{tightenumerate}
\item There is a constant  $N_0>0$ such that for each $k\in \{1,2\}$ and all $\omega,t\in \Omega\times [0,T]$, 
$$
	|r_{1}^{-1}b_t|_{0}+|\nabla b_t|_{\bar{\beta}-1}+|r_1^{-1}\sigma ^{k}_t|_{0}+|\nabla \sigma ^{k}_t|_{\bar{%
\beta}-1} \le N_{0}.
$$
Moreover, for each $k\in \{1,2\}$ and all $(\omega,t,z)\in\Omega\times [0,T]\times ( D^{k}\cup E^{k})$,
$$
|r_{1}^{-1}H^{k}_t(z)| _{0}\le K^k_t(z) \quad \textrm{and} \quad |\nabla
H^{k}_t(z)| _{\bar{\beta} -1} \leq \bar K^{k}_t(z)
$$
where $K^{k},\bar K^k: \Omega \times[ 0,T]\times (D^{k}\cup E^{k})\rightarrow \mathbf{R}_+$ are $\mathcal{P}_{T}\otimes \mathcal{Z}^{k}$-measurable functions satisfying
$$
K^{k}_t(z)+\bar K^{k}_t(z)+\int_{D^{k}}\left(K^{k}_t(z)^{\alpha }+\bar K^{k}_t(z)^{2 }\right)\pi^{k}
(dz)+\int_{E^{k}}\left(K^{k}_t(z)^{1\wedge\alpha}+\bar K^{k}_t(z)\right)\pi^{k}(dz)\leq
N_{0},
$$
for all
$(\omega,t,z)\in\Omega\times [0,T]\times ( D^{k}\cup E^{k})$.
\item 
For each $k\in\{1,2\},$ there is a constant $\eta^{k}\in (0,1)$ such that for all $(\omega ,t,x,z)\in \{(\omega ,t,x,z)\in \Omega \times
[ 0,T]\times\bR^{d_1}\times (D^{k}\cup E^{k}):|\nabla H_t ^{k}(\omega,x,z)|>\eta^{k} \},$ 
$$
\left|\left( I_{d_1}+\nabla
H^{k}_t(x,z)\right) ^{-1}\right|\leq N_{0}.
$$
\end{tightenumerate}
\end{assumption}

\begin{assumption}[$\tilde{\beta}$] \label{asm:regzerofreeclassicalwcorrec}  There is a constant  $N_0>0$ such that for each $k\in \{1,2\}$  and all $(\omega,t)\in \Omega\times [0,T]$,
$$
|c_t|_{\tilde{\beta}}+|\upsilon^{k}_t|_{\tilde{\beta}}+|r_{1}^{-\theta }f_t|_{\tilde{\beta}}+|r_{1}^{-\theta }g_t|_{%
\tilde{\beta}}\leq N_{0}.
$$
Moreover, for each $k\in \{1,2\}$ and all  $(\omega,t,z)\in \Omega\times [0,T]\times ( D^{k}\cup E^{k})$,
$$
|\rho^{k}
_{t}(z)|_{\tilde{\beta}} \le l^{k}_t(z),\quad |r_{1}^{-\theta }h_t(z)|_{\tilde{\beta}}\leq l^{k}_t(z),
$$
where $l^{k}: \Omega \times[ 0,T]\times Z^{k}\rightarrow \mathbf{R}_+$ are $\mathcal{P}_{T}\otimes \mathcal{Z}^{k}$-measurable function satisfying
$$
l^{k}_t(z)+\int_{D^{k}}l^{k}_t(z)^{2}\pi^{k}(dz)+\int_{E^{k}}l^{k}_t(z)\pi^{k}(dz)\leq N_{0},
$$
for all
$(\omega,t,z)\in \Omega\times [0,T]\times ( D^{k}\cup E^{k})$.
\end{assumption}

\begin{remark}
It follows from Lemma \ref{lem:compositediffeo} and Remark \ref{rem:GradientSmallInverse} that if Assumption \ref{asm:regcoeffclassicalwcorrec}$(\bar{\beta})$ holds for some $\bar{\beta}>1\vee \alpha$, then for all $\omega,t,$ and $z\in D^{k}\cup E^{k}$,
$
x\mapsto \tilde{H}^{k}_t(x,z):=x+H^{k}_t(x,z) 
$
is a diffeomorphism.
\end{remark}
Let  Assumptions \ref{asm:regcoeffclassicalwcorrec}$(\bar{\beta})$ and \ref{asm:regzerofreeclassicalwcorrec}$(\tilde{\beta}$)  hold for some   $\bar{\beta}>1\vee \alpha$ and  $\tilde{
\beta}>\alpha $.  In our derivation  of a solutions of \eqref{eq:FullSIDE}, we first obtain solutions of  equations of a special form. Specifically, consider the
system of SIDEs  on $[0,T]\times 
\bR^{d_1}$ given by 
\begin{align}\label{eq:FullSIDEwcorrec} 
d\hat{u}^l_t &=\left((\mathcal{L}^{1;l}_t+\mathcal{L}^{2;l}_t)\hat{u}_t+\hat{b}^i_t\partial_iu_t^{l}+\hat{c}^{l\bar l}_tu_t^{\bar l}+\hat{f}^l_t\right)dt +\left(\mathcal{N}^{1;l\varrho}_t\hat{u}_t+g^{l\varrho}_t\right)dw^{1;\varrho}_t\notag  \\
&\quad +\int_{Z^{1}}\left(\mathcal{I}^{1;l}_{t,z }\hat{u}_{t-}+h^l_t(z)\right)[\mathbf{1}_{D^{1}}(z)q^{1}(dt,dz )+\mathbf{1}_{E^{1}}(z)p^{1}(dt,dz)],\;\;\tau
< t\leq T,  \notag \\
\hat{u}^l_t&=\varphi^l,\;\;t\leq \tau , \;\;l\in\{1,\ldots,d_2\}, 
\end{align}%
where 
\begin{align*}
\hat{b}^i_t(x):&=\mathbf{1}_{[1,2]}(\alpha)b^i_t(x)+\sum_{k=1}^{2}\mathbf{1}_{\{2\}}(\alpha)\sigma_{t}^{k;j\varrho}(x) \partial_{j}\sigma^{k;i\varrho}_{t}(x)\\
&\quad +\sum_{k=1}^{2}\mathbf{1}_{(1,2]}(
\alpha)\int_{D^{k}} \left(H^{k;i}_{t}(x,z )-H^{k;i}_t(\tilde{H}_t^{k;-1}(x,z ),z)\right)\pi^{k} (dz ),\\
\hat{c}^{l\bar l}_t(x):&=c^{l\bar l}_t(x)+ \sum_{k=1}^2\mathbf{1}_{\{2\}}(\alpha)\sigma
^{k;j\varrho}_t(x)\partial_j\upsilon^{k;l\bar l \varrho}_t(x)\\
&\quad +\sum_{k=1}^2\int _{D^{k}}(\rho^{k;l\bar l}_t (x,z)-\rho^{k;l\bar l}_t (\tilde{H}^{k;-1}_t(x,z),z))
\pi^{k} (dz),\\
\hat{f}^l_t(x):&=f^l_t(x)+\sigma ^{1;j\varrho}_t(x)\partial_jg^{l\varrho}_t(x) +\int_{D^{1}} \left(h^l_t(x,z)- h^l_t(\tilde{H}^{1;-1}_t(x,z),z)\right)\pi^{1}(dz).
\end{align*}
\indent Let  $(w_{t}^{2;\varrho})_{\varrho\ge 1}$, $t\ge 0$, $\varrho\in \bN$, be a sequence of  independent  one-dimensional  Wiener processes.  Let $p^{2}(dt,dz )$ be a Poisson random measure  on $( [0,\infty)
\times Z^{2},\mathcal{B}([0,\infty)\otimes \mathcal{Z}^{2})$ with intensity measure $\pi^{2}(dz)dt$. Extending the probability space if necessary, we take  $w^{2}$ and $p^{2}(dt,dz)$ to be independent of $w^{1}$ and $p^{1}(dt,dz)$.
 Let 
\begin{equation*}
\hat{\cF}_t=\sigma \left((w^{2}_{s})_{\varrho\ge 1},\;p^{2}( [0,s],\Gamma)
:s\le t, \;\Gamma \in \mathcal{Z}^{2}\right)
\end{equation*}%
and $\tilde{\mathbf{F}}=\left(\mathcal{\tilde{F}}_{t}\right)_{t\le T}$ be
the standard augmentation of $\left( \cF_{t}\vee \hat{\cF}%
_t\right)_{t\le T}$.  Denote by
$
q^{2}(dt,dz )=p^{2}(dt,dz) -\pi^{2}(dz)dt
$
the compensated Poisson random measure. We
associate with the SIDE \eqref{eq:FullSIDEwcorrec}, the  $\tilde{\mathbf{F}}$-adapted stochastic flow $X_{t}=X_{t}(x)=X_{t}( \tau ,x),$ $(t,x)\in [0,T]\times\bR^{d_1}$, generated by the SDE
\begin{align} \label{eq:repflow}
dX_t &=-\mathbf{1}_{[1,2]}(\alpha)b_t(X_t)dt+\sum_{k=1}^2\mathbf{1}_{\{2\}}(\alpha)%
\sigma^{k;\varrho}_t (X_{t})dw^{k;\varrho}_{t}\notag \\
&\quad-\sum_{k=1}^2\int_{D ^{k}} H^{k}_t(\tilde{H}_t^{k;-1}(X_{t-},z ),z) [p^{k}( dt,dz)-\mathbf{1}_{(1,2]}(\alpha)\pi^{k}(dz)dt]\notag\\
&\quad -\sum_{k=1}^2\int_{E^{k}} H^{k}_t(\tilde{H}_t^{k;-1}(X_{t-},z ),z) p^{k}( dt,dz),\;\;\tau<t\le T, \notag  \\
X_{t} &=x,\;\;t\leq \tau,
\end{align}
and the $\tilde{\mathbf{F}}$-adapted random field $\Phi_t(x)=\Phi_t(\tau,x)$, $(t,x)\in [0,T]\times\bR^{d_1}$, solving the linear SDE given by
\begin{align}\label{eq:FeynmannKacTransformationMain}
d\Phi_{t}(x)& =\left(c_t(X_t(x))\Phi_{t}(x)+f_t(X_{t}(x))\right)dt+\sum_{k=1}^2\upsilon^{k;\varrho}_t(X_{t}(x))\Phi_{t}(x)dw^{k;\varrho}_{t}+g^{\varrho}_t(X_{t}(x))dw^{1;\varrho}_{t}\notag\\
&\quad +\sum_{k=1}^2\int_{Z^{k}}\rho^{k}_t (\tilde{H}%
^{k;-1}_t(X_{t-}(x),z),z)\Phi_{t-}(x)[\mathbf{1}_{D^{k}}(z)q^{k}(dt,dz )+\mathbf{1}_{E^{k}}(z)p^{k}(dt,dz)]  \notag\\
& \quad +\int_{Z^{1}}h_t(\tilde{H}%
^{1;-1}_t(X_{t-}(x),z),z)[\mathbf{1}_{D^{1}}(z)q^{1}(dt,dz )+\mathbf{1}_{E^{1}}(z)p^{1}(dt,dz)]  ,\;\;\tau<t\le  T, \notag \\
\Phi_{t}(x)& =\varphi(x),\;\;t\leq \tau.
\end{align}
\indent The coming theorem  is our existence, uniqueness, and representation theorem for \eqref{eq:FullSIDEwcorrec}. Let us describe our solution class. For each $\beta'\in(0,\infty)$, denote by $\mathfrak{C}^{\beta'}( \bR^{d_1};\bR^{d_2}) $ the 
 linear space of  all $\mathbf{F}$-adapted random fields $v=v_t(x)$  such that $\bP$-a.s.\ $$\mathbf{1} _{[ \tau _{n},\tau _{n+1})}r_{1}^{-\lambda_{n}
}v\in D([0,T];\cC^{\beta'}(\bR^{d_1},\bR^{d_2})),$$
where $(\tau_n)_{n\ge 0}$ is an increasing sequence of 
$\mathbf{F}$-stopping times with $\tau_0=0$ and $%
\tau _{n}=T$ for sufficiently large $n,$ and where for each $n$, $\lambda_{n}$ is a 
positive $\cF_{\tau _{n}}$-measurable random variable.  
\begin{theorem}\label{thm:repwcorrec}
Let  Assumptions \ref{asm:regcoeffclassicalwcorrec}$(\bar{\beta})$ and \ref{asm:regzerofreeclassicalwcorrec}$(\tilde{\beta})$ hold for some  $\bar{\beta}>1\vee \alpha$ and   $\tilde{
\beta}>\alpha $.  For each stopping time $\tau\le T$ and $\cF%
_{\tau }\otimes \mathcal{B}(\bR^{d_1})$-measurable random field $\varphi$ such that   for some $\beta '\in (\alpha ,\bar{\beta}\wedge \tilde{\beta} )$  and $\theta'\geq 0$, $\bP$-a.s.\
$ r_{1}^{-\theta'}\varphi\in\cC^{\beta'}(\mathbf{R}
^{d_1};\bR^{d_2})$, there exists a unique solution $\hat{u}=\hat{u}(\tau)$ of \eqref{eq:FullSIDEwcorrec}
in $\mathfrak{C}^{\beta'}( \bR^{d_1};\mathbf{R}%
^{d_2})$ and for all $(t,x)\in [0,T]\times \bR^{d_1}$, $\bP$-a.s.\
\begin{equation}\label{eq:conditionexpectationrepofsolution}
\hat{u}_t(\tau,x)=\mathbf{E}\left[\Phi_t (\tau,X_{t}^{-1}(\tau ,x))|\cF_{t}\right].
\end{equation}%
Moreover,  for each $\epsilon>0$ and $p\geq 2$, 
\begin{equation}\label{eq:estofoptionalprojection}
\mathbf{E}\left[\sup_{t\le T}| r_{1}^{-\theta \vee \theta'-\epsilon}\hat{u}_t(\tau)| _{\beta'}^{p}\big|\cF_{\tau }%
\right] \leq N(|  r_{1}^{-\theta'}\varphi%
|_{\beta'}^{p}+1),
\end{equation}
for a constant  $N=N(d_1,d_2,p,N_{0},T,\beta ',\eta^{1} ,\eta^{2},\epsilon,\theta,\theta')$.
\end{theorem}
Using It\^{o}'s formula it is easy to check that if $m=1$  and
\begin{equation*}
g_t(x)=0, \quad h_t(x)=0, \quad \textrm{and} \quad \rho^{k}_t(x,z)\ge -1,  
\end{equation*}
for all $(\omega,t,x,z)\in \Omega\times[{[\tau,T]}]\times \mathbf{R}%
^{d_1}\times (D^{k}\cup E^{k}), \;k\in\{1,2\},$
then
\begin{align*}
\Phi_t(x)&=\Psi_t(x)\phi(x)+\Psi_t(x)\int_{]\tau,\tau\vee t]} \Psi_{s}^{-1}(x)f_s(X_s(x))ds,
\end{align*}
where $\bP$-a.s.\ for all $t$ and $x$,
\begin{align}\label{eq:solutionofFK}
\Psi_t(x)&=e^{\int_{[\tau,\tau\vee t]} \left(c_s(X_s(x))-\sum_{k=1}^2\frac{1}{2}\upsilon^{k;\varrho}_s(X_{s}(x))\upsilon^{k;\varrho}_s(X_{s}(x))\right)ds +\sum_{k=1}^2 \int_{]\tau,\tau\vee t]} \upsilon^{k;\varrho}_s(X_{s}(x))dw^{k;\varrho}_{s}}\notag\\
&\quad \cdot e^{-\sum_{k=1}^2\int_{]\tau,\tau \vee t]}\int_{D^{k}}\left(\ln\left(1+\rho^{k}_s(\tilde{H}%
^{k;-1}_s(X_{s-}(x),z),z)\right)-\rho^{k}_s(\tilde{H}%
^{k;-1}_s(X_{s-}(x),z),z)\right)\pi^{k}(dz)ds }\notag\\
&\quad \cdot e^{\sum_{k=1}^2\int_{]\tau,\tau \vee t]}\int_{Z^{k}}\ln\left(1+\rho^{k}_s(\tilde{H}%
^{k;-1}_s(X_{s-}(x),z),z)\right)[\mathbf{1}_{D^{k}}(z)q^{k}(ds,dz )+\mathbf{1}_{E^{k}}(z)p^{k}(ds,dz)] }.
\end{align}
The following corollary  then follows  directly  from  \eqref{eq:conditionexpectationrepofsolution} and the  \eqref{eq:solutionofFK}.

\begin{corollary}\label{cor:Positivity}
Let   $m=1$ and assume that
 \begin{equation*}
g_t(x)=0, \;\; h_t(x,z)=0,\;\;  \rho^{k}_t(x,z)\ge -1,\;\;  \forall (\omega,t,x,z)\in [[\tau,T]]\times \mathbf{R}%
 ^{d_1}\times ( D^{k}\cup E^{k}), \;k\in\{1,2\}.
 \end{equation*}     Moreover, let  Assumptions \ref{asm:regcoeffclassicalwcorrec}$(\bar{\beta})$ and  \ref{asm:regzerofreeclassicalwcorrec}$(\tilde{\beta})$ hold for some  $\bar{\beta}>1\vee \alpha$ and  $\tilde{
 \beta}>\alpha $.     Let $\tau\le T$  be  stopping time and  $\varphi$ be a $\cF%
_{\tau }\otimes \mathcal{B}(\bR^{d_1})$-measurable random field such that   for some $\beta '\in (\alpha ,\bar{\beta}\wedge \tilde{\beta} )$  and $\theta'\geq 0$, $\bP$-a.s.\
$ r_{1}^{-\theta'}\varphi\in\cC^{\beta'}(\mathbf{R}
^{d_1};\bR^{d_2})$.
\begin{tightenumerate}
\item If for all $(\omega,t,x)\in[[\tau,T]] \times\bR^{d_1}$, $f_t(x)\ge 0$ and $\varphi(x)\ge 0$, then  the solution $\hat{u}$ of (\ref{eq:FullSIDE}) satisfies $
\hat{u}_t(x)\ge 0,$ $\bP$-a.s.\ for all $(t,x)\in [0,T]\times%
\bR^{d_1}$.
\item If for all $(\omega,t,x,z)\in [[\tau,T]]\times \mathbf{R}%
^{d_1}\times ( D^{k}\cup E^{k}),$ $k\in\{1,2\},$ $\upsilon^{k}_t(x)= 0,$ $f_t(x)\le 0$, $c_t(x)\le 0$%
, $\varphi(x)\le 1$, and $\rho_t^{k}(x,z)\le 0, $
 then the solution $\hat{u}$ of (\ref{eq:FullSIDE}) satisfies $
\hat{u}_t(x)\le 1$, $\bP$-a.s.\ for all $(t,x)\in [0,T]\times%
\bR^{d_1}$.
\end{tightenumerate}
\end{corollary}
\begin{remark}
Since $%
\mathcal{L}^{2}$ can be the zero operator, both Theorem \ref{thm:repwcorrec} and Corollary \ref{cor:Positivity} apply to fully degenerate equations and partial differential equations with random coefficients.
\end{remark}

Now, let us discuss our existence and uniqueness theorem for \eqref{eq:FullSIDE}. We construct the solution of $u=u(\tau)$ of \eqref{eq:FullSIDE} by interlacing the solutions of \eqref{eq:FullSIDEwcorrec} along a sequence of large jump moments (see Section \ref{s:Large
Jumps}). By using an interlacing procedure we are also able to drop the condition of boundedness of $(I+\nabla H^{1}_t(x,z))^{-1}$ on the set $(\omega ,t,x,z)\in $ $\{(\omega ,t,x,z)\in \Omega \times
[ 0,T]\times\bR^{d_1}\times (D^{1}\cup E^{1})$ $:|\nabla H_t ^{1}(\omega,x,z)|>\eta^{k} \}.$  Also, in order to remove the  terms in $\hat{b}$, $\hat{c},$ and $\hat{f}$  that appear in \eqref{eq:FullSIDEwcorrec}, but not in  \eqref{eq:FullSIDE}, we  subtract  terms from the relevant coefficients in the flow and the transformation. However, in order to do this, we need to impose  stronger  regularity assumptions on some of the coefficients and free terms.  We will introduce the  parameters $\mu^{1},\mu^2,\delta^1,\delta^2 \in[0,\frac{\alpha}{2}]$, which essentially allows one to trade-off integrability in $z$ and regularity  in $x$ of the coefficients $H^k_t(x,z),\rho^{k}_t(x,z),h^k_t(x,z)$.  It is worth mentioning that the removal of terms and the interlacing procedure are independent of each other and that it is due only to the weak assumptions on $H^1$ and $\rho^1$ on the set $V^1$ that we do not have moment estimates and a simple representation property like \eqref{eq:estofoptionalprojection} for the solution of \eqref{eq:FullSIDE}. Nevertheless, there is a representation of sorts and we refer the reader to the proof of the coming theorem for an explicit construction of the solution.

We introduce the following assumption for  $\bar{\beta}>1\vee \alpha$, $\tilde{\beta}>\alpha$, and  $\delta^1,\delta^2,\mu^{1},\mu^2\in[0,\frac{\alpha}{2}]$.  

\begin{assumption}[$\bar{\beta},\mu^1,\mu^2,\delta^1,\delta^2$] \label{asm:regclassicalwocorrec}
\begin{tightenumerate}
\item There is a constant  $N_0>0$ such that  for each $k\in \{1,2\}$ and    all $(\omega,t)\in \Omega\times [0,T]$,
\begin{gather*}
	|r_{1}^{-1}b_t|_{0}+|\nabla b_t|_{\bar{\beta}-1}+| \sigma ^{k}_t|_{\bar{%
\beta}+1}\leq N_{0}.
\end{gather*}
\item For each $k\in \{1,2\}$ and all $(\omega,t)\in \Omega\times [0,T]$,
\begin{gather*}
|H^{k}_t(z)|_{0}\le K^k_t(z), \quad |\nabla
H^{k}_t(z)| _{\bar{\beta} -1}, \;\forall z\in D^k,  \\
|r_{1}^{-1}H^{k}_t(z)|_{0}\le K^k_t(z), \quad |\nabla
H^{k}_t(z)| _{\bar{\beta} -1} \leq \bar K^{k}_t(z),\;\forall  z\in E^{k},\\
|\rho ^{k}(t,z)|_{\bar{\beta}}\leq l^{k}_t(z),\;\;\forall  z\in D^k,\quad |r_{1}^{-\theta }h_t(z)|_{\bar{\beta}} \leq l^{1}_t(z),\;\forall z\in D^{1},
\end{gather*}
where $K^{k},\bar K^k,l^k: \Omega \times[ 0,T]\times (D^{k}\cup E^{k})\rightarrow \mathbf{R}_+$ are $\mathcal{P}_{T}\otimes \mathcal{Z}^{k}$-measurable functions satisfying for all $(\omega,t,z)\in \Omega\times [0,T]\times  (D^{k}\cup E^{k})$,
$$
K^{k}_t(z)+\bar K^{k}_t(z)+l^{k}_t(z)\le N_0
$$
and
$$
\int_{D^{k}}\left(K^{k}_t(z)^{\alpha }+\bar K^{k}_t(z)^{2}+l^k_t(z)^{2}\right)\pi^{k}
(dz)+\int_{E^{k}}\left(K^{k}_t(z)^{1\wedge\alpha}+\bar K^k_t(z)\right)\pi^{k}(dz)\leq
N_{0}.
$$
\item For each $k\in \{1,2\}$ and all $(\omega,t)\in \Omega\times [0,T]$,
\begin{gather*}
|\upsilon^{k}_t|_{\bar{\beta}+1}\leq N_{0},\textit{ if } \sigma^k_t\ne 0,\quad |g_t|_{\bar{\beta}+1}\leq N_{0},\textit{ if } \sigma^1_t\ne 0,\\
\sum_{|\gamma|=[\bar{\beta}]^-}| \partial^{\gamma }H^k_t(z))|_{\{\bar{\beta}\}^++\delta^{k}}\le \tilde{K}^k_t(z), \;\forall  z\in D^{k}, \textrm{ if } \{\bar\beta\}^++\delta^k\le 1,\\
\sum_{|\gamma|=[\bar \beta]^-}|\nabla \partial^{\gamma} H^k_t(z)|_0\leq \bar K^{k}_t(z),\; \sum_{|\gamma|=[\bar{\beta}]^-}|\nabla \partial^{\gamma }H^k_t(z))|_{\{\bar{\beta}\}^++\delta^{k}-1}\le \tilde{K}^k_t(z), \;\forall  z\in D^{k}, \textrm{ if } \{\bar \beta\}^++\delta^k>1,\\
\sum_{|\gamma|=[\bar{\beta}]^-}| \partial^{\gamma }\rho^k_t(z))|_{\{\bar{\beta}\}^++\mu^{k}}\le \tilde{l}^k_t(z), \;\forall  z\in D^{k}, \textrm{ if } \{\bar\beta\}^++\mu^k\le 1,\\
\sum_{|\gamma|=[\bar \beta]^-}|\nabla \partial^{\gamma} \rho^k_t(z)|_0\leq l^{k}_t(z),\;\; \sum_{|\gamma|=[\bar{\beta}]^-}|\nabla \partial^{\gamma }\rho^k_t(z))|_{\{\bar{\beta}\}^++\mu^{k}-1}\le \tilde{l}^k_t(z), \;\forall  z\in D^{k}, \textrm{ if } \{\bar\beta\}^++\mu^k>1,\\
\sum_{|\gamma|=[\bar{\beta}]^-}| \partial^{\gamma }h^1_t(z))|_{\{\bar{\beta}\}^++\mu^{1}}\le \tilde{l}^1_t(z), \;\forall  z\in D^{1}, \textrm{ if } \{\bar\beta\}^++\mu^1\le 1,\\
\sum_{|\gamma|=[\bar \beta]^-}|\nabla \partial^{\gamma} h^1_t(z)|_0\leq l^{1}_t(z),\;\; \sum_{|\gamma|=[\bar{\beta}]^-}|\nabla \partial^{\gamma }h_t(z))|_{\{\bar{\beta}\}^++\mu^{1}-1}\le \tilde{l}^1_t(z), \;\forall  z\in D^{1}, \textrm{ if } \{\bar\beta\}^++\mu^1>1,
\end{gather*}
where $\tilde{K}^k,\tilde{l}^k: \Omega \times[ 0,T]\times D^k\rightarrow \mathbf{R}_+$ are $\mathcal{P}_{T}\otimes \mathcal{Z}^{k}$-measurable functions satisfying for all $(\omega,t,z)\in \Omega\times [0,T]\times  D^{k}$,
$$
\tilde{K}^{k}_t(z)+\tilde{l}^{k}_t(z)+\int_{D^{k}}\left(\tilde{K}^k_t(z)^{\frac{\alpha}{\alpha-\delta^{k}}}\mathbf{1}_{[0,\frac{\alpha}{2}]}(\delta^{k})+\tilde{K}^k_t(z)^2+\tilde{l}^k_t(z)^{\frac{\alpha}{\alpha-\mu^{k}}}\mathbf{1}_{[0,\frac{\alpha}{2}]}(\mu^{k})+\tilde{l}^k_t(z)^2\right)\pi^{k}
(dz)\leq
N_{0}.
$$
\item 
There is a constant $\eta^{2}\in (0,1)$ such that for all $(\omega ,t,x,z)\in \{(\omega ,t,x,z)\in \Omega \times
[ 0,T]\times\bR^{d_1}\times Z^{2}:|\nabla H_t ^{2}(\omega,x,z)|>\eta^{2} \},$
$$
\left|\left(I_{d_1}+\nabla
H^{2}_t(x,z)\right) ^{-1}\right|\leq N_{0}.
$$
\end{tightenumerate}
\end{assumption}

\begin{assumption}[$\tilde{\beta}$]\label{asm:regclassicalwocorrectild}
\begin{tightenumerate}
\item There is a constant  $N_0>0$ such that  for each $k\in \{1,2\}$ and    all $(\omega,t)\in \Omega\times [0,T]$,
\begin{gather*}
|c_t|_{\tilde{\beta}}+|r_{1}^{-\theta }f_t|_{\tilde{\beta}} \le N_{0},\\
 |\upsilon^{k}_t|_{\tilde{\beta}}\leq N_{0}, \textit{ if }\sigma^k_t=0,\quad|g_t|_{\tilde{\beta}}\leq N_{0}, \textit{ if }\sigma^1_t=0, \\
|\rho ^{k}(t,z)|_{\tilde{\beta}}\leq l^{k}_t(z),\;\;\forall  z\in E^{k},\quad |r_{1}^{-\theta }h_t(z)|_{\tilde{\beta}} \leq l^{1}_t(z),\;\forall z\in E^1,
\end{gather*}
where for all $(\omega,t)\in \Omega\times [0,T]$,
$
\int_{E^{k}}l^k_t(z)\pi^{k}(dz)\leq
N_{0}.
$
\item There exist processes $\xi,\zeta:\Omega \times [ 0,T]\times V^{1}\rightarrow \mathbf{R}_{+}$ that are $\mathcal{P}_{T}\otimes 
\cZ^{1}$measurable  satisfying
$$
|r_1^{-\xi_t(z)}H^{1}_t(z)|_{\tilde{\beta}\vee 1}+|r_1^{-\xi_t(z)}\rho^{1}_t(z)|_{\tilde{\beta}}+|r_1^{-\xi_t(z)}h_t(z)|_{\tilde{\beta}}\leq \zeta_t (z),
$$
for all $(\omega ,t,z)\in \Omega\times [0,T]\times V^{1}$. 
\end{tightenumerate}
\end{assumption}

We now state our existence and uniqueness theorem for \eqref{eq:FullSIDE}.
\begin{theorem}\label{thm:Existwocorrec} Let Assumptions \ref{asm:regclassicalwocorrec}$(\bar \beta,\delta^1,\delta^2,\mu^1,\mu^2)$  and \ref{asm:regclassicalwocorrectild}$(\tilde{\beta})$ hold for some  $\bar{\beta}>1\vee \alpha$,  $\tilde{
\beta}>\alpha $, and $\delta^1,\delta^2,\mu^{1},\mu^2\in[0,\frac{\alpha}{2}]$.   For each stopping time $\tau\le T$ and $\cF%
_{\tau }\otimes \mathcal{B}(\bR^{d_1})$-measurable random field $\varphi$ such that   for some $\beta '\in (\alpha ,\bar{\beta}\wedge \tilde{\beta} )$  and $\theta'\geq 0$, $\bP$-a.s.\
$ r_{1}^{-\theta'}\varphi\in\cC^{\beta'}(\mathbf{R}
^{d_1};\bR^{d_2})$, there
exists a unique solution $u=u(\tau)$ of (\ref{eq:FullSIDE}) in $\mathfrak{C}^{\beta'}( \bR^{d_1};\bR^{d_2})$.  
\end{theorem}

\section{Proof of main theorems}

We will first  prove  uniqueness of  the solution of  \eqref{eq:FullSIDEwcorrec} in the class $\mathfrak{C}%
^{\beta'}( \bR^{d_1};\bR^{d_2})$. The existence part of the proof of  Theorem \ref{thm:repwcorrec} is divided into a series of steps.  In the first step, by appealing to  the representation theorem we  derived for  solutions of continuous SPDEs in Theorem 2.4 in \cite{LeMi14b}, we use an interlacing
procedure and  the strong limit theorem given in Theorem 2.3 in \cite{LeMi14b} to show that the space inverse of the flow  generated by  a jump SDE (i.e. the SDE \eqref{eq:repflow} without the uncorrelated noise) solves a degenerate linear SIDE.   Then we linearly transform the inverse flow of a jump SDE to obtain solutions of degenerate linear SIDEs with free and zero-order terms and an initial condition. In the last step of the proof of  Theorem \ref{thm:repwcorrec}, we  introduce an independent
Wiener process and Poisson random measure as explained above, apply the results we know for fully degenerate equations, and  then take the optional projection of the equation. In the last section, Section \ref{sec:uncorrelated}, we prove Theorem  \ref{thm:Existwocorrec}  using an interlacing procedure and removing the extra terms in $\hat{b}, \hat{c}$ and $\hat{f}$.  The uniqueness  of the solution $u$ of   \eqref{eq:FullSIDE} follows directly from our construction. 
\subsection{Proof of uniqueness for Theorem \ref{thm:repwcorrec} }
\begin{proof}[Proof of Uniqueness for Theorem \ref{thm:repwcorrec}]
Fix a stopping time $\tau\le T$ and $\cF%
_{\tau }\otimes \mathcal{B}(\bR^{d_1})$\allowbreak-measurable random field $\varphi$ such that   for some $\beta '\in (\alpha,\bar{\beta}\wedge \tilde{\beta} )$  and $\theta'\geq 0$, $\bP$-a.s.\
$ r_{1}^{-\theta'}\varphi\in\cC^{\beta'}(\mathbf{R}
^{d_1};\bR^{d_2})$.  In this section we will drop the dependence of processes $t,x, $ and $z$ when we feel it will not obscure the  argument.
Let $\hat{u}^{1}(\tau)$ and $\hat{u}^{2}(\tau)$ be solutions of \eqref{eq:FullSIDEwcorrec} in $\mathfrak{C}^{\beta'}$. It follows that  $v:=\hat{u}^{1}(\tau)-\hat{u}^{2}(\tau)$ solves 
\begin{align}\label{eq:FullSIDEwcorrecuniqueness} 
dv^l_t &=[(\mathcal{L}^{1;l}_t+\mathcal{L}^{2;l}_t)v_t+\hat{b}^i_t\partial_iv_t^{l}+\hat{c}^{l\bar l}_tv_t^{\bar l}]dt +\mathcal{N}^{1;\varrho}_tv^l_tdw^{1;\varrho}_t\notag  \\
&\quad +\int_{Z^{1}}\mathcal{I}^{1;l}_{t,z }v_{t-}[\mathbf{1}_{D^{1}}(z)q^{1}(dt,dz )+\mathbf{1}_{E^{1}}(z)p^{1}(dt,dz)],\;\;\tau
< t\leq T,   \\
v^l_t&=0,\;\;t\leq \tau , \;\;l\in\{1,\ldots,d_2\},\notag
\end{align}
and $\bP$-a.s.\ 
$$\mathbf{1} _{[ \tau _{n},\tau _{n+1})}r_{1}^{-\lambda_{n}
}v\in D([0,T];\cC^{\beta'}(\bR^{d_1},\bR^{d_2})),$$
where $(\tau_n)_{n\ge 0}$ is an increasing sequence of 
$\mathbf{F}$-stopping times with $\tau_0=0$ and $%
\tau _{n}=T$ for sufficiently large $n,$ and where for each $n$, $\lambda_{n}$ is a 
positive $\cF_{\tau _{n}}$-measurable random variable.  Clearly it suffices to take $\tau_1=\tau$ and $\lambda_0=0$.  Thus, $v_{t}(x)=0$ for all $(\omega,t)\in [[\tau_0,\tau_1))$.   Assume that for some $n$, $\bP$-a.s.\ for all $t$ and $x$, $v_{t\wedge \tau_{n}}(x)=0$.  We will show that $\bP$-a.s.\ for all $t$ and $x$, $\tilde{v}_t(x):=v_{(\tau_{n}\vee t)\wedge \tau_{n+1} }(x)=0$. Applying It\^o's formula, for each $x$, $\bP$-a.s.\ for all $t$, we find
\begin{align}\label{eq:FullSIDEIto} 
d|\tilde{v}_t|^2 &=\left(2\tilde{v}^l_t\mathfrak{L}^{1;l}_t\tilde{v}_t+|\mathcal{N}^{1}_t\tilde{v}_t|^2+2\tilde{v}^l_tb^i_t\partial_i\tilde{v}_t^{l}+2\tilde{v}^l_tc_t^{l\bar l}\tilde{v}^{\bar l}_t\right)dt\notag\\
&\quad+\left(2\tilde{v}^l_t\mathfrak{I}^{1;l}_{t,z}\tilde{v}_t+\int_{D^{1}\cup E^{1}}|\mathcal{I}^{1;l}_{t,z }\tilde{v}_{t}|^2\pi^1(dz)\right)dt \notag  \\
&\quad +\left(2v^l_t\mathfrak{L}^{2;l}_t\tilde{v}_t+2\tilde{v}^l_t\mathfrak{I}^{2;l}_{t,z}\tilde{v}_t\right)dt+2v^l_t\mathcal{N}^{1;\varrho}_t\tilde{v}^l_tdw^{1;\varrho}_t\notag\\
&\quad +\int_{Z^{1}}\left(2\tilde{v}^l_{t-}\mathcal{I}^{1;l}_{t,z }\tilde{v}_{t-}+|\mathcal{I}^{1;l}_{t,z }\tilde{v}_{t-}|^2\right)q^{1}(dt,dz ),\;\;\tau_n
< t\leq \tau_{n+1}, \notag  \\
|\tilde{v}_t|^2&=0,\;\;t\leq \tau_n , \;\;l\in\{1,\ldots,d_2\},
\end{align}
where for $\phi\in C_c^{\infty}(\bR^{d_1};\bR^{d_2})$, $k\in\{1,2\}$, and $l\in\{1,\ldots,d_2\}$,
$$
\mathfrak{L}^{k;l}\phi:=\frac{1}{2}\sigma
^{k;i\varrho}\sigma
^{k;j\varrho}\partial_{ij}\phi^l+\sigma^{k;j\varrho} \partial_{j}\sigma^{k;i\varrho}\partial_i\phi^l+\sigma^{k;i\varrho}\upsilon^{k;l\bar l\varrho}\partial_i\phi^{\bar l}+\sigma
^{k;j\varrho}\partial_ja^{k;l\bar l \varrho}\phi^{\bar l}
$$
and
\begin{align*}
\mathfrak{I}^{k;l}\phi:&=\int _{D^{k}}\left(\rho^{k;l\bar l}\phi^{\bar l} (\tilde{H}^{k})-\rho^{k;l\bar l}(\tilde{H}^{k;-1})\phi^{\bar l}\right)\pi^{k}(dz ) \\
&\quad +\int _{D^{k}}\left(\phi^{ l} (\tilde{H}^{k})-\phi^{ l}+\mathbf{1}_{(1,2]}(\alpha)F^{k;i}\partial_i\phi^l\right)\pi^{k}(dz )\\
&\quad + \int_{E^{k}}\left((I_{d_2}^{l\bar l}+\rho^{k;l\bar l})\phi^{\bar l} (\tilde{H}^{k})-\phi^l\right)\pi^{k}(dz).
\end{align*}
For each $\omega$ and $t$, let 
$$
Q_t=\int_{\bR^{d^1}}|\tilde{v}_t(x)|^2 r_1^{-\lambda}(x)dx,
$$
where  $\lambda=\lambda_n+(d'+2)/2$ and $d'>d_1$. Note that 
$$
\bE Q_t\le \int_{\bR^{d^1}}r_1^{-d'}(x)dx\bE|r_1^{-\lambda_n}\tilde{v}_t|_{0}<\infty.
$$ 
It suffices to show that $\sup_{t\le T} \bE Q_t=0$. 
To this end, we will multiply the equation  \eqref{eq:FullSIDEIto} by the weight $r_1^{-2\lambda}=r_1^{-2\lambda_n+1}r_1^{-d'}$, integrate in $x$, and change the order of  the integrals in time and space. Thus, we must verify the assumptions of  stochastic Fubini theorem hold (see Corollary \ref{cor:StochasticFubini} and Remark \ref{rem:stochasticfubini} as well) with the finite measure $\mu(dx)=r_1^{-d'}(x)dx$ on $\bR^{d_1}$.  
Since  $b$ and $\sigma^k$ have linear growth an $\upsilon^k$ and $c$ are bounded, owing to Lemma \ref{lem:equivnorm}, we easily obtain that  there is a constant $N=N(d_1,d_2,N_0,\lambda_n)$ such that $\bP$-a.s for all $t$,
$$
\int_{\bR^{d_1}}\left(\sum_{k=1}^22|r_1^{-\lambda_n}\tilde{v}| |r_1^{-\lambda_n-2}\mathfrak{L}^{k}\tilde{v}| + |r_1^{\lambda_n-1}\cN^{1}\tilde{v}|^2\right)r_1^{-d'}dx\le N\sup_{t\le T}|r_1^{-\lambda_n}\tilde{v}|^2_{\beta'},
$$
$$
\int_{\bR^{d_1}}4|r_1^{-\lambda_n}\tilde{v}|^2 |r_1^{-\lambda_n-1}\cN^{1}\tilde{v}|^2r_1^{-d'}dx\le N\sup_{t\le T}|r_1^{-\lambda_n}\tilde{v}_t|^4_{\beta'},
$$
and
$$
\int_{\bR^{d_1}}\left(2|r_1^{-\lambda_n}\tilde{v}|r_1^{-\lambda_n-1}b\partial_i\tilde{v}|+2|r_1^{-\lambda_n}\tilde{v}||r_1^{-\lambda_n}c\tilde{v}|\right)r_1^{-d'}dx\le N\sup_{t\le T}|r_1^{-\lambda_n}\tilde{v}_t|^2_{\beta'}.
$$
For all $\phi\in C^{\alpha}_{loc}(\bR^{d_1};\bR^{d_2})$ and all  $k,\omega,t,x,p,$ and $z$,  
\begin{gather}
r_1^{-p}( \phi(\tilde{H}^k)-\phi+\mathbf{1}_{(1,2]}(\alpha)F^{k;i}\partial_i\phi)\\
=\bar{\phi}(\tilde{H}^k)-\bar{\phi}-\mathbf{1}_{(1,2]}(\alpha)H^{k;i}\partial_i\bar{\phi}+ \mathbf{1}_{(1,2]}(\alpha)(H^{k;i}+F^{k;i})\partial_i\bar{\phi}\notag\\
 +p\mathbf{1}_{(1,2]}(\alpha)(H^{k;i}+F^{k;i})r_1^{-2}x^i\bar{\phi} +\left(\frac{r_1^p(\tilde{H}^{k})}{r_1^p}-1\right)(\bar{\phi}(\tilde{H}^k)-\mathbf{1}_{(1,2]}(\alpha)\bar{\phi})\notag\\\label{eq:weightIntegralOperator} +\mathbf{1}_{(1,2]}(\alpha)\left(\frac{r_1^p(\tilde{H}^{k})}{r_1^p}-1+pH^{k;i}r_1^{-2}x^i\right)\bar{\phi},
\end{gather}
where $\bar \phi:= r^{-p}\phi$. By  Taylor's formula, for all $\phi\in C^{\alpha}(\bR^{d_1};\bR^{d_2})$ and all  $k,\omega,t,x,$ and $z$,  we have
\begin{equation}\label{ineq:estimateofIntegralalpha}
|\phi(\tilde{H}^k)-\phi-\mathbf{1}_{(1,2]}(\alpha)H^{k;i}\partial_i\phi|\le r_1^{\alpha} |\phi|_{\alpha}|r_1^{- 1}H|_0^{\alpha}.
\end{equation}
Combining \eqref{eq:weightIntegralOperator}, \eqref{ineq:estimateofIntegralalpha}, and the estimates given in Lemma \ref{lem:compositediffeo} (1), for all $k,\omega,t,x$ and $z$,  we obtain
$$
r_1^{-\alpha}| \rho^k(\tilde{H}^{k;-1})-\rho^k|\le N| \rho|_{\alpha\wedge 1}|r_1^{-1}H^k|_0^{\alpha\wedge 1}
$$
and
\begin{gather}
r_1^{-\lambda_n-\alpha}| \tilde{v}(\tilde{H}^k)-\tilde{v}+\mathbf{1}_{(1,2]}(\alpha)F^{k;i}\partial_i\tilde{v}|\notag \\
\label{ineq:estimateofIntegralalpha2} \le N | r_1^{-\lambda_n}\tilde{v}|_{\alpha}\left(|r_1^{-1}H^k|_0^{\alpha}+|r_1^{-1}H|_0[H^k]_1+|r_1^{-1}H|_0^{[\alpha]^-+1}+[H]_1^{[\alpha]^-+1}\right),
\end{gather}
for some constant  $N=N(d_1,\lambda_n,N_{0},\eta^{1},\eta^2 ).$
Therefore, $\bP$-a.s for all $t$,
$$
\int_{\bR^{d_1}}\left(\sum_{k=1}^22|r_1^{-\lambda_n}\tilde{v}| |r_1^{-\lambda_n-2}\mathfrak{I}^{k}\tilde{v}| +\int_{D^1\cup E^1}|r_1^{-\lambda-1}\cI_{z}\tilde{v}|^2\pi^1(dz)\right)r_1^{-d'}dx\le N\sup_{t\le T}|r_1^{-\lambda_n}\tilde{v}|^2_{\beta'},
$$
and
$$
\int_{\bR^{d_1}}\left(2|r_1^{-\lambda_n}\tilde{v}| |r_1^{-\lambda_n-2}\cI^{k}_z\tilde{v}|+|r_1^{-\lambda_n-1}\cI_{z}\tilde{v}|^2\right)^2r_1^{-d'}dx\le N\sup_{t\le T}|r_1^{-\lambda_n}\tilde{v}_t|^4_{\beta'},
$$
for some constant $N=N(d_1,d_2,\lambda_n,N_{0},\eta^{1} ,\eta^{2}).$ 

Let $L^2(\bR^{d_1};\bR^{d_2})$  be the space of square-integrable functions $f:\bR^{d_1}\rightarrow \bR^{d_2}$  with norm $\|\cdot\|_{0}$ and inner product $(\cdot,\cdot)_{0}$.  Moreover, let  $L^2(\bR^{d_1};\ell_2(\bR^{d_2}))$ be the space of  square-integrable functions $f:\bR^{d_1}\rightarrow \ell_2(\bR^{d_2})$  with norm $\|\cdot\|_{0}$.  With the help of the above estimates and Corollary \ref{cor:StochasticFubini}, denoting $\bar{v}=r^{-\lambda}\tilde{v}$, $\bP$-a.s.\ for all $t$, we have 
\begin{align}\label{eq:ItoSquareL2norm}
d\|\bar{v}_t\|_{0}^2 &=\left(2(\bar{v}^l_t,\bar{\mathfrak{L}}^{1}_t\bar{v}_t)_{0}+\|\bar{\mathcal{N}}^{1}_t\bar{v}_t\|^2_{0}+2(\bar{v}_t,\bar{\mathfrak{I}}^{1}_{t,z}\bar{v}_t)_{0}+\int_{D^{1}\cup E^{1}}\|\bar{\mathcal{I}}^{1}_{t,z }\bar{v}_{t}\|^2_{0}\pi^1(dz)\right)dt \notag  \\
&\quad +\left(2(\tilde{v}_t,b^i_t\partial_i\tilde{v}_t+\bar{c}_t^{\bar l}\tilde{v}^{\bar l}_t)_{0} +2(\tilde{v}_t,\bar{\mathfrak{L}}^{2}_t\tilde{v}_t)_{0}+2(\tilde{v}_t,\bar{\mathfrak{I}}^{2}_{t,z}\tilde{v}_t)_{0}\right)dt+2(v_t,\bar{\mathcal{N}}^{1;\varrho}_t\tilde{v}_t)_{0}dw^{1;\varrho}_t\notag\\
&\quad +\int_{Z^{1}}\left(2(\tilde{v}_{t-},\bar{\mathcal{I}}^{1}_{t,z }\tilde{v}_{t-})_{0}+\|\bar{\mathcal{I}}^{1}_{t,z }\tilde{v}_{t-}\|^2_{0}\right)q^{1}(dt,dz ),\;\;\tau_n
< t\leq \tau_{n+1},\notag   \\
\|\bar{v}_t\|^2_{0}&=0,\;\;t\leq \tau_n , \;\;l\in\{1,\ldots,d_2\},
\end{align}
where all coefficients and operators  are defined  as in \eqref{eq:FullSIDEwcorrec} with the following changes: 
\begin{tightenumerate}
\item   for each $k\in\{1,2\}$, $\upsilon^{k}$ is replaced with
$$
\bar{\upsilon}^{k;l\bar l}:=\upsilon^{k;l\bar l}+\mathbf{1}_{\{2\}}(\alpha)\lambda\sigma^{k;i\varrho}r_1^{-2}x^i\delta_{l\bar l};
$$
\item for each $k\in\{1,2\}$, $\rho^{k}$ replaced with 
$$
\bar{\rho}^{k;l\bar l}:=\rho^{k;l\bar l}+\left(\frac{r_1^{\lambda}(\tilde{H}^{k})}{r_1^{\lambda}}-1\right)(I_{d_2}^{l\bar l}+\rho^{k;l\bar l});
$$
\item  $c$ is replaced with 
\begin{gather*}
\bar{c}^{l\bar l}=c^{l\bar l}+\lambda b^ir^{-2}x^i\delta_{l\bar l}+\sum_{k=1}^2\lambda^2\sigma^{k;i\varrho}\sigma^{k;j\varrho}r_1^{-4}x^ix^j\\ +\sum_{k=1}^2\int_{D^{k}}\left(\left(\frac{r_1^{\lambda}}{r_1^{\lambda}(\tilde{H}^{k;-1})}-1\right)(I^{l\bar l}_m+\rho^{k}(\tilde{H}^{k;-1}))-\mathbf{1}_{(1,2]}(\alpha)\lambda r_1^{-2}x_iH^{k;i}(\tilde{H}^{k;-1})\right)\pi^k(dz).
\end{gather*}
\end{tightenumerate}
Since for all $k,\omega$ and $t$, $|r_1^{-1}\sigma^k|_{0}+|r_1^{-1}\nabla \sigma^k|_{\bar{\beta}-1}+|\upsilon^k|_{\tilde{\beta}}\le N_0,$ for $\bar{\beta}>1\vee \alpha$ and $\tilde{\beta}>\alpha$, it is clear that
$
|\bar \upsilon^k|_{\alpha}\le N.
$
Moreover, since for all  $k,\omega$ and $t$, $|r_1^{-1} H^k|_0+|H^k|_{\bar \beta}\le K^k$ and $|\rho|_{\tilde{\beta}'}\le l^k$, applying the estimates in Lemma \eqref{lem:compositediffeo} (1), we get
$$
|\bar \rho^k|_{\alpha}\le l^k+K^k(1+l^k) \quad \textrm{and} \quad |c|_{\alpha}\le N_0.
$$
We will now  estimate the drift terms of \eqref{eq:ItoSquareL2norm} in terms of $\|\bar v_t\|_{0}^2$. 
We write $f\sim g$ if  $\int_{\bR^{d_1}}|f(x)|dx $ $=\int_{\mathbf{{R}^{d_1}}}|g(x)|dx$ and $f \ll g$ if$\int_{\bR^{d_1}}|f(x)|dx \le  \int_{\mathbf{{R}^{d_1}}}|g(x)|dx$. 
Using the divergence theorem, for any $v:\bR^{d_1}\rightarrow\bR^{d_2}$, $\sigma :\bR^{d_1}\rightarrow\bR^{d_1}$ and $\upsilon:\bR^{d_1}\rightarrow\mathbf{R}^{2d_2}$ and all $x$,   we get 
$$
\sigma^i\sigma^jv^lv^l_{ij}\sim \frac{1}{2}(\sigma^i\sigma^j)_{ij}v-\sigma^i\sigma^jv^l_iv^l_j =  (\sigma^i_{ij}\sigma^j+\sigma^i_j\sigma ^j_i)|v|^2-\sigma^i\sigma^jv^l_iv^l_j ,
$$
$$
2\sigma^i_j\sigma ^jv^lv^l_i\sim- (\sigma^i_j\sigma^j)_i|v|^2=(\sigma^i_{ij}\sigma^j+\sigma^i_j\sigma^{j}_i)|v|^2, 
$$
and
$$ \quad \sigma^iv^l\upsilon^{l\bar l}v_i^{\bar l}+\sigma^iv^{\bar l}\upsilon^{l\bar l}v_i^{ l}=\sigma^iv^l\upsilon^{l\bar l}_{sym}v_i^{\bar l}\sim -(\sigma^i\upsilon^{l\bar l}_{sym})_i|v|^2=-(\sigma^i_i\upsilon_{sym}^{l\bar l}+\sigma^i\upsilon_{sym}^{l\bar l})|v|^2,
$$
where $\upsilon_{sym}^{l\bar l}=(\upsilon^{l\bar l}+\upsilon^{\bar l l})/2$. Consequently,  for all $\omega,t,$ and $x$, we have 
\begin{align*}
2\bar{v}^{l}\bar{\mathfrak{L}}^{1;l}\bar{v}+|\bar{\mathcal{N}}^{1}\bar{v}|^2& \sim \frac{1}{2}\left(|\operatorname{div} \sigma^1|^2-\partial_i\sigma^{1;j\varrho}\partial_j\sigma^{1;i\varrho}\right)|\bar{v}|^2-\bar{\upsilon}^{1;l\bar l\varrho}_{sym}\bar{v}^l\bar{v}^{\bar l}\operatorname{div}\sigma^{1;\varrho}+|\bar{\upsilon}^{1}\bar{v}|^2\ll N |\bar{v}|^2
\end{align*}
and
$$
2\bar{v}^{l}\bar{\mathfrak{L}}^{(2);l}\bar{v}\ll -(1+\epsilon )|\sigma^{2;i}\partial_i\bar v|^2+N|\bar v|^2,
$$
for any $\epsilon>0$, where in the last estimate we have also used Young's inequality.  By Lemma \ref{lem:compositediffeo} (2) and basic properties of the determinant, there is a constant $N=N(d,N_0,\eta^{1} ,\eta^{2})$ such that for all $k,\omega,t,x,$ and $z$, 
$$
\det \tilde{H}^{k;-1}-1=\det(I_d+ F^{k})-1\le |\nabla F^k|\le N |\nabla H^k|$$
and
$$
 \det \tilde{H}^{k;-1}-1-\operatorname{div}F^k\le |\nabla F^k|^2\le N|\nabla H^k|^2.
$$
Thus, integrating by parts, 
 for all $\omega,t,$ and $x$, we get 
\begin{gather*}
2\bar{v}^{l}\bar{\mathfrak{I}}^{1;l}\bar{v}+\int_{D^1\cup E^1}|\bar{\cI}^{1}\bar{v}|^2\pi
^{1}(dz)\sim  2\int_{D^{1}}\bar{\rho}^{1;l\bar l}_{sym}(\tilde{H}^{1;-1})(\det\nabla\tilde{H}^{1;-1}-1)\pi^{1}(dz)\bar{v}^{\bar l}\bar{v}^{ l}\\
+\int_{D^{1}\cup E^{1}}\left(\det\nabla\tilde{H}^{1;-1}-1+\mathbf{1}_{(1,2]}(\alpha)\mathbf{1}_{D^{1}}\operatorname{div}F^{1}\right)\pi^{1}(dz)|\bar{v}|^2\\
+\int_{D^{1}\cup E^{1}}\left(\mathbf{1}_{E^{1}}2\bar{\rho}^{1;l\bar l}_{sym}(\tilde{H}^{1;-1})\bar v^{\bar l}\bar{v} ^l+|\bar{\rho}^{1}(\tilde{H}^{1;-1})\bar v|^2 \right)\det\nabla\tilde{H}^{1;-1}\pi^{1}(dz)\\
\ll N\left(\int_{D^{1}} \left(K^{1}(z)^2+l^{1}(z)K^{1}(z)+l^{1}(z)^2\right)\pi^{1}(dz)+\int_{E^{1}}\left(K^k(z)+l^{k}(z)\right)\pi^{1}(dz)\right)|\bar v|^2.
\end{gather*}
Analogously,  for all $\omega,t,$ and $x$, we obtain
$$
2\bar{v}^{l}\bar{\mathfrak{I}}^{2;l}\bar{v}\le -(1+\epsilon)\int_{D^2\cup E^2}|\bar{v}(\tilde{H}^2)-\bar{v}|^2\pi^2(dz)+N |\bar v|^2.
$$
Therefore, combining the above estimates, $\bP$-a.s.\ for all $t,$
\begin{equation}\label{ineq:estimateofsquarenorm}
Q_t\le N\int_0^t Q_sds+M_t,
\end{equation}
where $(M_t)_{t\le T}$ is a c\`{a}dl\`{a}g  square-integrable martingale. 
Taking the expectation of \eqref{ineq:estimateofsquarenorm} and applying Gronwall's lemma, we get
$
\sup_{t\le T}\bE Q_t=0,
$
which implies that $\bP$-a.s.\ for all $t$ and $x$, $\tilde{v}_t(x)=0$.  This completes the proof. 
\end{proof}

\subsection{Small jump case}
Set $(w^{\varrho})_{\varrho\ge 1}=(w^{1;\varrho})_{\varrho\ge 1}$, $(Z, \cZ,\pi )=(\cZ^{1},\cZ^{1},\pi ^{1}),$ $%
p(dt,dz )=p^{1}(dt,dz )$, and $q(dt,dz )=q^{1}(dt,dz )$.  Let $%
\sigma_t (x)=( \sigma
^{i\varrho }_t(x)) _{1\leq i\le d_1,\varrho \geq 1}$ be a $\ell_2(\bR^{d_1})$-valued $\mathcal{R}%
_T\otimes \mathcal{B}(\bR^{d_1})$-measurable function defined on $%
\Omega \times [ 0,T]\times \bR^{d_1}$ and $%
H_t(x,z)=(H^{i}_t(x,z))_{1\leq i\le d_1}$ be a $\mathcal{P}_{T}\otimes 
\mathcal{B}(\bR^{d_1})\otimes \mathcal{Z}$-measurable function defined
on $\Omega \times [ 0,T]\times \bR^{d_1}\times Z. $ 

We  introduce the following assumption for $\beta >1\vee \alpha $.

\begin{assumption}[$\beta$]\label{asm:smalljump}
\begin{tightenumerate}
\item There is a constant $N_0>0$ such that for all $(\omega,t)\in \Omega\times[0,T]$,
$$
|r_1^{-1}b_t|_0+|r_1^{-1}\sigma_t|_0+
|\nabla b_t|_{\beta -1}+| \nabla \sigma_t|_{\beta -1}\leq N_{0}.
$$
Moreover, for all $(\omega,t,z)\in \Omega\times [0,T]\times Z$,
$$
|r_{1}^{-1}H_t(z)|_{0}\le K_t(z) \quad \textrm{and} \quad  |\nabla
H_t(z)| _{\beta -1} \le \bar K_t(z),
$$
where $K: \Omega \times[ 0,T]\times Z\rightarrow \mathbf{R}_+$ is a $\mathcal{P}_{T}\otimes \mathcal{Z}$-measurable function satisfying
\begin{equation*}
K_t(z)+\bar K_t(z)+\int_{Z}\left(K_t(z)^{\alpha }+\bar K_t(z)^2\right)\pi
(dz)\leq
N_{0},
\end{equation*}
for all $(\omega,t,z)\in \Omega\times[0,T]\times Z$.
\item There is a constant $\eta\in (0,1)$ such that for all $(\omega ,t,x,z)\in \{(\omega ,t,x,z)\in \Omega \times
[ 0,T]\times\bR^{d_1}\times Z:|\nabla H_t (\omega,x,z)|>\eta \},$ 
\begin{equation*}
| \left(I_{d_1}+\nabla
H_t(x,z)\right)^{-1}|\leq N_{0}.
\end{equation*}
\end{tightenumerate}
\end{assumption}
Let Assumption \ref{asm:smalljump}$(\beta)$ hold for
some $\beta >1\vee \alpha$.  Let  $\tau\le T$ be a stopping time. Consider the
system of SIDEs on $[0,T]\times\bR^{d_1}$  given by
\begin{align}\label{eq:SmallJumpSIDE} 
dv_t(x) &=\left(\mathbf{1}%
_{\{2\}}(\alpha)\frac{1}{2}\sigma ^{i\varrho}_t(x) \sigma
^{j\varrho}_t(x)\partial_{ij}v_t(x)+\mathfrak{b}^{i}_t(x)\partial_iv_t(x)\right)dt   +\mathbf{1}_{\{2\}}(\alpha)\sigma ^{i\varrho}_t(x)\partial_iv_{t}(x)dw^{\varrho}_{t}  \notag \\
&+1_{(1,2]}(\alpha)\int_{Z}\left(v_t(x+H_t(x,z ))-v_t(x)+F_t(x,z)\partial_iv_{t}(x)\right)\pi (dz )dt  \notag \\
&+\int_{Z}\left(v_{t-}(x+H_t(x,z ))-v_{t-}(x)\right)[1_{(1,2]}(\alpha)q(dt,dz)+1_{[0,1]}(\alpha)p(dt,dz)]\notag,\;\; \tau < t\leq T,\\
v_t(x)&=x,\;\;t\leq \tau ,
\end{align}%
where 
$$
\mathfrak{b}^i_t(x):=\mathbf{1}_{[1,2]}(\alpha)b^i_t(x)+\mathbf{1}_{\{2\}}(\alpha)\sigma_{t}^{j\varrho}(x) \partial_{j}\sigma^{i\varrho}_{t}(x) 
$$
and
$$
F_t(x,z):=-H_t(\tilde{H}^{-1}_t(x,z),z).
$$
We associate with \eqref{eq:SmallJumpSIDE},  the stochastic flow  $Y_t=Y_t(\tau,x)$,
 $(t,x) \in [0,T]\times \bR^{d_1}$, generated by the SDE
\begin{align}\label{eq:SmallJumpSDE}
dY_{t} &=-\mathbf{1}_{[1,2]}(\alpha )b_t(Y_t)dt-\mathbf{1}_{\{2\}}(\alpha )\sigma^{\varrho}_t
(Y_{t})dw^{\varrho}_{t}\notag\\
&\quad +\int_{Z}F_t(Y_{t-},z) [1_{(1,2]}(z)q(dt,dz)+1_{[0,1]}(z)p(dt,dz)],\;\;\tau<t\le  T,
 \\
Y_{t} &=x,\;\;t\leq \tau.\notag
\end{align}%
Owing to parts (1) and (2) of Lemma \ref{lem:compositediffeo},  for each $\omega ,t,$ and $z,$ the inverse of the mapping $%
\tilde{F}_t(x,z):=x+F_t(x,z)=x-H_t(\tilde{H}^{-1}_t(x,z),z)$ is $\tilde{H}_t
(x,z):=x+H_t(x,z)$ and  there is a constant $N=N(d_1,N_0,\beta,\eta)$
such that for all $\omega,t,x,y,$ and $z$, 
$$
|r_1^{-1}F_t(z)|_0\le NK_t(z),\quad  |\nabla F_t(z)|_{\beta -1}\leq K_t(z),
\quad 
|(I_{d_1}+\nabla F_t(x,z))^{-1}|\leq N.
$$
Thus,  by Theorem 2.1 in \cite{LeMi14b}, there is a modification of the  solution of  \eqref{eq:SmallJumpSDE}, which we still denote by $Y_t=Y_{t}(\tau ,x)$, that is a $\cC_{loc}^{\beta'}$%
-diffeomorphism for any $\beta'\in [1,\beta)$. Moreover,  $\bP$-a.s.\ $Y_{\cdot}(\tau,\cdot),Y^{-1}_{\cdot}(\tau,\cdot)\in D([0,T];\cC_{loc}^{\beta'}(\bR^{d_1};\bR^{d_1}))$, and  $Y_{t-}^{-1}(\tau,\cdot )$ coincides
with the inverse of $Y_{t-}(\tau,\cdot)$ for all $t$.  The following proposition shows that 
 the inverse flow $Y_{t}^{-1}(\tau )$ solves \eqref{eq:SmallJumpSIDE}.

\begin{proposition}\label{Prop:SmallJump}
Let Assumption \ref{asm:smalljump}$(\beta)$ hold for
some $\beta >1\vee \alpha$.  For each stopping time $\tau\le T$ and $\beta'\in [1\vee \alpha,\beta)$, $v_t(x)=v_t(\tau,x)=Y_t^{-1}(\tau,x)$ solves \eqref{eq:SmallJumpSIDE} and for each $\epsilon >0$ and $p\ge 2$, there is a constant $N=N(d_1,p,N_{0},T,\beta ',\eta ,\epsilon)$ such that
\begin{equation}\label{ineq:SmallJumpEst}
\mathbf{E}\left[\sup_{t\leq T}|r_{1}^{-(1+\epsilon )}v_{t}(\tau )|_{0}^{p}\right]+\mathbf{E}\left[\sup_{t\leq T}|r_{1}^{-\epsilon }\nabla v_{t}(\tau )|_{\beta'-1}^{p}\right]\leq N.
\end{equation}
\end{proposition}
\begin{proof}
The estimate  \eqref{ineq:SmallJumpEst} is given in Theorem 2.1 in  \cite{LeMi14b} (see also Remark 2.1), so we 
only need to show that $Y^{-1}_t(\tau,x)$ solves \eqref{eq:SmallJumpSIDE}.  Let $(\delta _{n})_{n\geq 1}$ be a
sequence such that $\delta _{n}\in (0,\eta )$ for all $n$ and $\delta
_{n}\rightarrow 0$ as $n\rightarrow \infty $. It is clear that there is a constant $N=N(N_0)$ such that for all $\omega$ and $t$, 
\begin{equation}\label{ineq:SmallJumpFiniteofLevyMeasure}
\pi (\{z:K_t(z)>\delta _{n}\})\leq \frac{N}{\delta_n ^{\alpha }}.
\end{equation}
 For each $n$, consider the
system of SIDEs on $[0,T]\times\bR^{d_1}$ given by 
\begin{gather}
dv^{(n)}_t(x) =\left(\mathbf{1}_{\{2\}}(\alpha )\frac{1}{2}\sigma ^{i\varrho}_t(x) \sigma
^{j\varrho}_t(x)\partial _{ij}v^{(n)}_t(x)+\mathfrak{b}^i_t(x)\partial _{i}v^{(n)}_t(x)\right)dt \notag \\
+1_{(1,2]}(\alpha)\int_{Z}\mathbf{1}_{\{K_t>\delta
_{n}\}}(z)\left(v^{(n)}_t(x+H_t(x,z))-v^{(n)}_t(x)+F^{i}_t(x,z)\partial _{i}v^{(n)}_t(x)\right)\pi (dz)dt  \notag \\
 +\int_Z \mathbf{1}_{\{K_t>\delta
_{n}%
\}}(z)\left(v^{(n)}_{t-}(x+H_t(x,z))-v^{(n)}_{t-}(x)\right)[1_{(1,2]}(\alpha)q(dt,dz)+1_{[0,1]}(\alpha)p(dt,dz)],
\notag \\+\mathbf{1}_{\{2\}}(\alpha )\sigma_t ^{i\varrho}(x)\partial
_{i}v^{(n)}_t(x)dw_{t}^{\varrho},  \;\;\tau<t\le  T,\;\;
v^{(n)}_t(x) =x, \;\; t\le \tau,\label{eq:SmallJumpSIDEn1} 
\end{gather}%
and the stochastic flow  $Y^{(n)}_t=Y^{(n)}_t(\tau,x)$,
 $(t,x) \in [0,T]\times \bR^{d_1}$, generated by the SDE
\begin{align} \label{eq:SmallJumpSDEn1} 
dY^{(n)}_t &=-\mathbf{1}_{[1,2]}(\alpha )b_t(Y^{(n)}_t)dt-\mathbf{1}_{\{2\}}(\alpha )\sigma_t^{\varrho}
(Y^{(n)}_t)dw_{t}^{\varrho} \notag\\
&\quad +\int_Z \mathbf{1}_{\{K_t>\delta
_{n}\}}(z)F_t(Y^{(n)}_{t-},z)[1_{(1,2]}(\alpha)q(dt,dz)+1_{[0,1]}(\alpha)p(dt,dz)],\;\;\tau<t\le T, 
\notag \\
Y^{(n)}_t(x) &=x, \;\; t\le \tau. 
\end{align}%
Since \eqref{ineq:SmallJumpFiniteofLevyMeasure} holds, we can rewrite equation \eqref{eq:SmallJumpSDEn1} as 
\begin{align}\label{eq:SmallJumpSDEn2} 
dY^{(n)}_t &=-\left(\mathbf{1}_{[1,2]}(\alpha )b_t(Y^{(n)}_t)+1_{(1,2]}(\alpha)\int_{Z}\mathbf{1}_{\{K_t>
\delta _{n}\}}(z)F_t(Y_t^{(n)},z)\pi (dz)\right)dt\\
&\quad -\mathbf{1}_{\{2\}}(\alpha
)\sigma^{\varrho}_t(Y_{n}(t))dw^{\varrho}_{t} +\int_{Z}\mathbf{1}_{\{K_t>\delta
_{n}\}}(z)F_t(Y^{(n)}_{t-},z)p(dt,dz),\;\;\tau<t\le T,\notag  
\end{align}
and  \eqref{eq:SmallJumpSIDEn1} as 
\begin{align}\label{eq:SmallJumpSIDEn2} 
dv^{(n)}_t(x)& =\left(\mathbf{1}_{\{2\}}(\alpha )\frac{1}{2}\sigma ^{i\varrho}_t(x) \sigma
^{j\varrho}_t(x)\partial _{ij}v^{(n)}_t(x)+\mathfrak{b}^i_t(x)\partial_{j}\sigma^{i\varrho}_{t}(x) \right)dt\notag\\
& \quad +\mathbf{1}_{\{2\}}(\alpha )\sigma_t ^{i\varrho}(x)\partial
_{i}v^{(n)}_t(x)dw_{t}^{\varrho}+1_{(1,2]}(\alpha)\int_{Z}\mathbf{1}_{\{K_t>
\delta _{n}\}}(z)F^{i}_t(x,z)\pi (dz)\partial _{i}v_{t}^{(n)}(x)dt \notag \\
& \quad +\int_{Z}\mathbf{1}_{\{K_t>\delta
_{n}\}}(z)\left(v_{t-}^{(n)}(x+H_t(x,z))-v_{t-}^{(n)}(x)\right)p(dt,dz),\;\;\tau<t\le T.
\end{align}%
We claim that the  solution $Y^{(n)}_t=Y^{(n)}_t(x)$ of \eqref{eq:SmallJumpSDEn2} can be written as the
solution of continuous SDEs with a finite number of jumps interlaced.
Indeed, for each $n$ and stopping time $\tau'\le T$, consider the stochastic flow  $\tilde{Y}^{(n)}_t=\tilde{Y}^{(n)}_t(\tau',x)$,
 $(t,x) \in [0,T]\times \bR^{d_1}$, generated by the SDE
\begin{align*}
d\tilde{Y}^{(n)}_t &=-[\mathbf{1}_{[1,2]}(\alpha )b_t(\tilde{Y}^{(n)}_t)+1_{(1,2]}(\alpha)\int_{Z}\mathbf{1}_{\{K>
\delta _{n}\}}(t,z)F_t(\tilde{Y}_t^{(n)},z)\pi (dz)]dt\notag\\
&\quad -\mathbf{1}_{\{2\}}(\alpha
)\sigma_t^{\varrho} (\tilde{Y}^{(n)}_t)dw_{t}^{\varrho},\;\tau'<t\le  T,\notag \\
\tilde{Y}^{(n)}_t &=x,\;\;t\leq \tau'.
\end{align*}
By Theorems 2.1 and 2.4 and Remark 2.2 in \cite{LeMi14b}, there is a modification of $\tilde{Y}^{(n)}_t$ $=$   $\tilde{Y}^{(n)}_t$ $(\tau',x),$ still denoted $\tilde{Y}_{t}^{(n)}(\tau' ,x)$, that is a $\cC_{loc}^{\beta'}$%
-diffeomorphism. Furthermore, $\bP$-a.s.\   we have that $$\tilde{Y}_{\cdot}^{(n)}(\tau' ,\cdot),\tilde{Y}^{(n);-1}_{\cdot}(\tau',\cdot)\in C([0,T];\cC_{loc}^{\beta'})$$ and   $\tilde{v}^{(n)}_t=\tilde{v}^{(n)}_t(\tau',x)=\tilde{Y}^{(n);-1}_t(\tau',x)$ solves the SPDE given by
\begin{align}\label{eq:SmallJumpCtsSPDE}
d\tilde{v}^{(n)}_t(x)& =\left(\mathbf{1}_{\{2\}}(\alpha )\frac{1}{2}\sigma ^{i\varrho}_t(x) \sigma
^{j\varrho}_t(x)\partial _{ij}v^{(n)}_t(x)+\mathfrak{b}^i_t(x)\partial _{i}v^{(n)}_t(x)\right)dt\\
&\quad +\mathbf{1}_{\{2\}}(\alpha )\sigma_t ^{i\varrho}(x)\partial
_{i}v^{(n)}_t(x)dw_{t}^{\varrho} \notag\\
& \quad +1_{(1,2]}(\alpha)\int_{Z}\mathbf{1}_{\{K> \delta
_{n}\}}(t,z)F^{i}(t,z)\pi
(dz)dt\partial _{i}v^{(n)}_t(x),\;\;\tau'<t\le  T,  \notag \\
\tilde{v}^{(n)}_t(x)& =x,\;\;t\leq \tau'.  
\end{align}%
For each $n$, let 
\begin{equation*}
A^{(n)}_t=\int_{]0,t]}\int_{Z}\mathbf{1}_{\{K_s>\delta
_{n}\}}(z)p(ds,dz),\;\;t\geq 0,
\end{equation*}%
and define the sequence of stopping times $(\tau _{l}^{(n)})_{l=1}^{\infty}$ recursively by $\tau _{0}^{(n)}=\tau$ and 
\begin{equation*}
\tau _{l+1}^{(n)}=\inf \left\{t>\tau _{l}^{(n)}:\Delta A^{(n)}_t\neq 0\right\}
\wedge T.
\end{equation*}%
Fix some $n\geq 1$. It is clear that  $\bP$-a.s.\ for all $x$ and $t\in [0,\tau _{1}^{(n)})$, $$Y^{(n);-1}_t(\tau,x)=\tilde{Y}^{(n);-1}_t(\tau,x)=\tilde{v}^{(n)}_t(\tau,x)$$
satisfies \eqref{eq:SmallJumpSIDEn2} up to, but  not including time $\tau _{1}^{(n)}$.   Moreover, $\bP$-a.s.\ for all $x$,
$$
Y^{(n)}_{\tau _{1}^{(n)}}(\tau,x) =\tilde{Y}^{(n)}_{\tau _{1}^{n}-}(\tau,x)+\int_{Z}F_{\tau
_{1}^{(n)}}(\tilde{Y}^{(n)}_{\tau _{1}^{(n)}-}(\tau,x),z)p(\{\tau _{1}^{(n)}\},dz),
$$
and hence 
\begin{equation*}
Y^{(n);-1}_{\tau _{1}^{(n)}}(\tau,x)=\int_{Z}\tilde{v}^{(n)}_{\tau _{1}^{(n)}-}(\tau,x+H_{\tau
_{1}^{(n)}}(x,z))p(\{\tau _{1}^{(n)}\},dz).
\end{equation*}%
Consequently, $v^{(n)}_t(\tau,x)=Y^{(n);-1}_t(\tau,x)$ solves \eqref{eq:SmallJumpSIDEn2} up to and including  time $\tau _{1}^{(n)}$. Assume that for some $l\geq 1$, $v^{(n)}_t(\tau,x)=
Y^{(n);-1}_t(\tau,x)$ solves \eqref{eq:SmallJumpSIDEn2} up to and including time $\tau _{l}^{(n)}$.
Clearly, $\bP$-a.s.\ for all $x$ and $t\in[\tau _{l}^{(n)},\tau _{l+1}^{(n)})$, $Y^{(n)}_t(x)=\tilde{Y}^{(n)}_t(\tau _{l}^{(n)},Y^{(n)}_{\tau _{l}^{(n)}-}(x))$,
and thus $\bP$-a.s.\ for all $x$ and $t\in[\tau _{l}^{(n)},\tau _{l+1}^{(n)})$,
$$Y^{(n);-1}_t(x)=\tilde{Y}^{(n)}_t(\tau _{l}^{(n)},Y^{(n)}_{\tau _{l}^{(n)}-}(x))=\tilde{v}^{(n)}_t(\tau _{l}^{(n)},Y^{(n)}_{\tau _{l}^{(n)}-}(x)).$$
Moreover,  $\bP$-a.s.\ for all $x$,
\begin{equation*}
Y_{n}^{-1}(\tau _{l+1}^{n},x)=\int_{U}\tilde{v}_{n}(\tau _{l}^{n},\tau
_{l+1}^{n}-,x+H(\tau _{l+1}^{n},x,z))p(\{\tau _{l+1}^{n}\},dz),
\end{equation*}%
which implies that $v^{(n)}_t(\tau,x)=Y^{(n);-1}_t(\tau,x)$ solves \eqref{eq:SmallJumpSIDEn2} up to and including
time $\tau _{l+1}^{n}$. Therefore, by induction, for each $n$, $%
v^{(n)}_t(\tau,x)=Y^{(n);-1}_t(\tau,x)$ solves \eqref{eq:SmallJumpSIDEn2}.  It is easy to see that for all $\omega,t,$ and $z$, 
$$
|r_1^{-1}\mathbf{1}_{\{K_t>\delta
_{n}\}}(z)F_t(z)-r_1^{-1}F_t(z)|_{0}+| \mathbf{1}_{\{K_t>\delta
_{n}\}}(z)\nabla F_t(z)-\nabla F_t(z)|_{\beta-1}\le \mathbf{1}_{\{K_t\le \delta
_{n}\}}(z)K_t(z)
$$
and thus
$$
d\bP dt-\lim_{n\rightarrow\infty} \int_D \mathbf{1}_{\{K\le \delta
_{n}\}}(t,z)K_t(z)^2 \pi(dz)+d\bP dt-\lim_{n\rightarrow\infty} \int_E \mathbf{1}_{\{K\le \delta
_{n}\}}(t,z)K_t(z) \pi(dz)=0.
$$
By virtue of Theorem 2.3 in \cite{LeMi14b}, for each $\epsilon>0,$ and $p\ge2,$  we have
\begin{gather*}
\lim_{n\rightarrow \infty }\left(\mathbf{E}\left[\sup_{t\leq T} | r_1^{-(1+\epsilon)}(Y^{(n)}_t(\tau )-r_1^{-(1+\epsilon)}Y_t(\tau)|_{0}^{p}\right]+\mathbf{E}\left[\sup_{t\leq T} |
r_1^{-\epsilon}\nabla Y^{(n)}_t (\tau) -r_1^{-\epsilon}\nabla Y_{t}(\tau) |
_{\beta'-1}^{p}\right]\right)=0,\\
\lim_{n\rightarrow \infty }\mathbf{E}\left[\sup_{t\leq T} | r_1^{-(1+\epsilon)}Y^{(n);-1}_t(\tau )-r_1^{-(1+\epsilon)}Y^{-1}_t(\tau)|_{0}^{p}\right]=0
\end{gather*}
and
$$
\lim_{n\rightarrow \infty }\mathbf{E}\left[\sup_{t\leq T} |
r_1^{-\epsilon}\nabla Y^{(n);-1}_t (\tau) -r_1^{-\epsilon}\nabla Y^{-1}_{t}(\tau ) |
_{\beta'-1}^{p}\right]=0.
$$
Then passing to the limit in both sides of \eqref{eq:SmallJumpSIDEn1} and making use of Assumption \ref{asm:smalljump}$(\beta)$, the estimate \eqref{ineq:estimateofIntegralalpha2},  and basic convergence properties of stochastic integrals, we find that $v_t(\tau,x)=X^{-1}_t(\tau,x)$ solves \eqref{eq:SmallJumpSIDE} .
\end{proof}

\subsection{Adding free and zero-order terms}

Set $(w^{\varrho})_{\varrho\ge 1}=(w^{1;\varrho})_{\varrho\ge 1}$, $(Z, \cZ,\pi )=(\cZ^{1},\cZ^{1},\pi ^{1}),$ $%
p(dt,dz )=p^{1}(dt,dz )$, and $q(dt,dz )$ $=p^{1}(dt,dz )-\pi^1(dz)dt$.  Also, set $D=D^{1}$, $E=E^{1},$ and assume $Z=D\cup E$. Let $\upsilon_t(x)=(\upsilon_t^{l\bar l\varrho}(\omega,x))_{1\le l,\bar l\le d_2,\;\varrho\ge 1}$ be a $\ell_2(\mathbf{R}^{2d_2})$-valued $\mathcal{R}%
_T\otimes \mathcal{B}(\bR^{d_1})$-measurable function defined on $%
\Omega \times [ 0,T]\times \bR^{d_1}$ and $%
\rho_t(x,z)=(\rho^{l \bar l}_t(\omega,x,z))_{1\leq l,\bar l\le d_2}$ be a $\mathcal{P}_{T}\otimes 
\mathcal{B}(\bR^{d_1})\otimes \mathcal{Z}$-measurable function defined
on $\Omega \times [ 0,T]\times \bR^{d_1}\times Z. $ 

We introduce the following assumptions for $\beta >1\vee \alpha $ and $\tilde{\beta}%
>\alpha $.

\begin{assumption}[$\beta$]\label{asm:smalljumpinfree}
\begin{tightenumerate}
\item There is a constant $N_0>0$ such that for all $(\omega,t)\in \Omega\times[0,T]$,
$$
|r_1^{-1}b_t|_0+|r_1^{-1}\sigma_t|_0+
|\nabla b_t|_{\beta -1}+| \nabla \sigma_t|_{\beta -1}\leq N_{0}.
$$
Moreover, for all $(\omega,t,z)\in \Omega\times [0,T]\times Z$,
$$
|r_{1}^{-1}H_t(z)|_{0}\le K_t(z) \quad \textrm{and} \quad  |\nabla
H_t(z)| _{\beta -1} \le \bar K_t(z),
$$
where $K: \Omega \times[ 0,T]\times Z\rightarrow \mathbf{R}_+$ is a $\mathcal{P}_{T}\otimes \mathcal{Z}$-measurable function satisfying
\begin{equation*}
K_t(z)+\bar K_t(z)+\int_{D}\left(K_t(z)^{\alpha }+\bar K_t(z)^2\right)\pi
(dz)+\int_{E}\left(K_t(z)^{\alpha \wedge 1 }+\bar K_t(z)\right)\pi
(dz)\leq
N_{0},
\end{equation*}
for all $(\omega,t,z)\in \Omega\times[0,T]\times Z$.
\item There is a constant $\eta\in (0,1)$ such that for all $(\omega ,t,x,z)\in \{(\omega ,t,x,z)\in \Omega \times
[ 0,T]\times\bR^{d_1}\times Z:|\nabla H_t (\omega,x,z)|>\eta \},$ 
\begin{equation*}
| \left(I_{d_1}+\nabla
H_t(x,z)\right)^{-1}|\leq N_{0}.
\end{equation*}
\end{tightenumerate}
\end{assumption}

\begin{assumption}[$\tilde{\beta}$]\label{asm:freeandzero}
There is a constant $N_0>0$ such that for all $(\omega,t)\in\Omega\times [0,T]$,
$$
|c_t|_{\tilde{\beta}}+|\upsilon_t|_{\tilde{\beta}}+|r_{1}^{-\theta }f_t|_{\tilde{\beta}}+|r_{1}^{-\theta }g_t|_{%
\tilde{\beta}}\leq N_{0}.
$$
Moreover, for all $(\omega,t,z)\in \Omega\times [0,T]\times Z$,
\begin{align*}
|\rho
_{t}(z)|_{\tilde{\beta}} +|r_{1}^{-\theta }h_t(z)|_{\tilde{\beta}}\leq l_t(z),
\end{align*}%
where $l: \Omega \times[ 0,T]\times Z\rightarrow \mathbf{R}_+$ is a $\mathcal{P}_{T}\otimes \mathcal{Z}$-measurable function satisfying
\begin{equation*}
l_t(z)+\int_{D}l_t(z)^{2}\pi(dz)+\int_{E}l_t(z)\pi(dz)\leq N_{0}.
\end{equation*}
$(\omega,t,z)\in \Omega\times [0,T]\times Z$.
\end{assumption}

Let Assumptions \ref{asm:smalljumpinfree}$(\bar \beta)$ and  \ref{asm:freeandzero}$(\tilde{\beta})$ hold
for some  $\bar\beta >1\vee \alpha$ and $\tilde{\beta}>\alpha$.  Let  $\tau\le T$  be a stopping time and $\varphi:\Omega \times \bR^{d_1}\rightarrow \bR^{d_2}$ be a  $\cF_{\tau }\otimes \mathcal{B}(\bR^{d_1})$-measurable random field. Consider the system of SIDEs on $[0,T]\times\bR^{d_1}$ 
given by 
\begin{align}\label{eq:FKTransSIDE} 
dv^l_t&=\left(\cL^l_tv_t+\hat{\mathfrak{b}}^i_t\partial_i\phi^{l}+\hat{\mathfrak{c}}^{l\bar l}_t\phi^{\bar l}+\hat{\mathfrak{f}}^l_t\right)dt +\left(\mathcal{N}^{l\varrho}_tv_t+g^{l\varrho}_t\right)dw_t^{\varrho} \notag\\
&\quad +\int_{Z}\left(
\mathcal{I}^{l}_{t,z }v_{t-}+h^l_t(z )\right)[\mathbf{1}_{D}(z)q(dt,dz )+\mathbf{1}_E(z)p(dt,dz)],\;\;\tau<t\le  T,  \notag\\
v^l_t&=\varphi^l ,\;\;t\leq \tau, \;\;l\in\{1,\ldots,d_2\},
\end{align}
where for $\phi \in C_{c}^{\infty }( \bR^{d_1};\mathbf{R}
^{d_2})$ and $l\in\{1,\ldots,d_2\}$,
\begin{align*}
\cL^l_{t}\phi(x) &:=\mathbf{1}_{\{2\}}(\alpha)\frac{1}{2}\sigma_t
^{i\varrho}(x)\sigma_t
^{j\varrho}(x)\partial_{ij}\phi^l(x)+\mathbf{1}_{\{2\}}(\alpha)\sigma^{i\varrho}_t(x)a^{l\bar l\varrho}_t(x)\partial_i\phi^{\bar l}(x)\\
&\quad +\int _{D^{k}}\rho^{l\bar l}_t(x,z )\left(\phi^{\bar l} (x+H_t(x,z ))-\phi^{\bar l} (x)\right)\pi(dz ) \\
&\quad+\int _{D^{k}}\left(\phi^{ l} (x+H_t(x,z ))-\phi^{ l} (x)-%
\mathbf{1}_{(1,2]}(\alpha)\partial_i\phi^l(x)H^{i}_t(x,z )\right)\pi(dz ) \\
\mathcal{N}^{l\varrho}_t\phi^l(x)&:=\mathbf{1}_{\{2\}}(\alpha)\sigma^{i\varrho}_t(x)\partial_i\phi^l(x)+\upsilon^{l\bar l
\varrho}_t(x)\phi^{\bar l}(x), \\
\mathcal{I}^{l}_{t,z }\phi^l (x) &:=(I_{d_2}+\rho^{l\bar l}_t(x,z ))\phi^{\bar l} (x+H_t(x,z ))-\phi^{l}
(x),
\end{align*}
and where 
\begin{align*}
\hat{\mathfrak{b}}^i_t(x):&=\mathbf{1}_{[1,2]}(\alpha)b^i_t(x)+\mathbf{1}_{\{2\}}(\alpha)\sigma_{t}^{j\varrho}(x) \partial_{j}\sigma^{i\varrho}_{t}(x) \\
&\quad +\int_{D} \left(\mathbf{1}_{(1,2]}(
\alpha)H^i_{t}(x,z )-H^i_t(\tilde{H}_t^{-1}(x,z ),z)\right)\pi (dz ),\\
\hat{\mathfrak{c}}^{l\bar l}_t(x):&=c^{l\bar l}_t(x)+ \mathbf{1}_{\{2\}}(\alpha)\sigma
^{j\varrho}_t(x)\partial_j\upsilon^{l\bar l \varrho}_t(x)+\int _D\left( \rho^{l\bar l}_t (x,z)-\rho^{l\bar l}_t (\tilde{H}^{-1}_t(x,z),z)\right)
\pi (dz),\\
\hat{\mathfrak{f}}^l_t(x):&=f^l_t(x)+\mathbf{1}_{\{2\}}(\alpha)\sigma ^{j\varrho}_t(x)\partial_jg^l_t(x) +\int_D \left(h^l_t(x,z)- h^l_t(\tilde{H}^{-1}_t(x,z),z)\right)\pi(dz).
\end{align*}
We associate with  \eqref{eq:FKTransSIDE} the stochastic flow $%
X_{t}=X_{t}(x)=X_{t}\left( \tau ,x\right) ,$ $(t,x)\in [ 0,T]\times 
\bR^{d_1}$, given by \eqref{eq:SmallJumpSDE}. Let $\Gamma_{t}(x)=\Gamma_t
(\tau ,x),$ $(t,x)\in [ 0,T]\times \bR^{d_1}$, be the solution of the linear SDE given by
\begin{align}\label{eq:FKTransyinit}
d\Gamma_{t}(x)& =\left(c_t(X_t(x))\Gamma_{t}(x)+f_t(X_{t}(x))\right)dt+\left(\upsilon^{\varrho}_t(X_{t}(x))\Gamma_{t}(x)+g^{\varrho}_t(X_{t}(x))\right)dw^{\varrho}_{t}\notag\\
&\quad +\int_{Z}\rho_t (\tilde{H}%
^{-1}_t(X_{t-}(x),z),z)\Gamma_{t-}(x)[\mathbf{1}_{D}(z)q(dt,dz )+\mathbf{1}_E(z)p(dt,dz)]  \notag\\
& \quad +\int_{Z}h_t(\tilde{H}%
^{-1}_t(X_{t-}(x),z),z)[\mathbf{1}_{D}(z)q(dt,dz )+\mathbf{1}_E(z)p(dt,dz)],\;\;\tau<t\le  T,\notag  \\
\Gamma_{t}(x)& =0,\;\;t\leq \tau.
\end{align}
Let $\Psi_{t}(x)=\Psi _t(\tau,x),$ $(t,x)\in [ 0,T]\times \bR^{d_1}$, be the unique solution of the linear SDE given by
\begin{align}\label{eq:FKTransmult} 
d\Psi_{t}(x)& =c_t(X_{t}(x))\Psi_t(x)dt+\upsilon^{\varrho}_t(X_{t}(x))\Phi_t(x)dw^{\varrho}_{t}\notag\\
&\quad +
\int_{Z}\rho_t(\tilde{H}^{-1}_t(X_{t-}(x),z),z)\Psi_{t-}(x)[\mathbf{1}_{D}(z)q(dt,dz )+\mathbf{1}_E(z)p(dt,dz)],\;\;\tau<t\le  T,
\notag\\
\Psi _{t}(x)& =I_{d_2},\;\;t\leq \tau .  
\end{align}

In the following lemma, we obtain $p$-th moment estimates of the weighted H{\"o}lder norms of $\Gamma$ and $\Psi$.  

\begin{lemma}
\label{lem:EstimateofFKTranswoinitial} Let Assumptions \ref{asm:smalljumpinfree}$(\bar \beta)$ and \ref{asm:freeandzero}$(\tilde{\beta})$ hold for
some $\bar\beta >1\vee \alpha$ and $\tilde{\beta}>\alpha$.  For each stopping time $\tau\le T$ and $\beta'\in [0 ,\bar{\beta}\wedge \tilde{\beta} )$, there exists a $D([0,T],\cC_{loc}^{\beta'}\allowbreak( \mathbf{R}%
^{d_1};\bR^{d_2})) $-modification of $\Gamma(\tau )$
and $\Psi (\tau )$,  also denoted by $\bar {\Gamma}(\tau )$
and $\Psi (\tau)$, respectively. Moreover, for each $\epsilon >0$ and $p\ge 2$, there is a constant $N=N(d_1,d_2,p,N_{0},T,\beta ',\eta ,\allowbreak\epsilon,\theta)$ such that
\begin{equation}\label{ineq:estimateFKTransGamPhi}
\mathbf{E}\left[\sup_{t\leq T}[|r_1^{-(\theta+\epsilon)}\Gamma_t(\tau)|_{\beta'}^{p}\right]+\mathbf{E}\left[\sup_{t\leq T}|r_1^{-\epsilon}\Psi _{t}(\tau) |_{\beta'}^{p}\right]\leq N.
\end{equation}
\end{lemma}
\begin{proof}
Let $\tau\le T$ be a fixed stopping time  and  $\beta:=\bar{\beta}\wedge \tilde{\beta} $. 
Estimating \eqref{eq:FKTransyinit}   directly and using the Burkholder-Davis-Gundy inequality, Lemma \ref{i:Lp Burkholder for Random Measure}, the multiplicative decomposition
$$
h_t(x,\tilde{H}^{-1}_t(X_{t-}(x),z ),z )=r_{1}^{\theta }(X_{t-}(x)) 
\frac{r_{1}^{\theta }(\tilde{H}^{-1}_t(X_{t-}(x),z)) }{r_{1}^{\theta
}( X_{t}(x)) }\frac{h_t(\tilde{H}^{-1}_t(X_{t-}(x),z
),z )}{r_{1}^{\theta }( \tilde{H}^{-1}_t(X_{t-}(x),z)) },
$$
H\"{o}lder's inequality,  Lemma \ref{lem:compositediffeo} (1), Lemma 3.2 in \cite{LeMi14b}, and Gronwall's inequality, we get that  for all $x$ and $y$,
\begin{equation}
\mathbf{E}\left[\sup_{t\le T}|\Gamma_t(x)|^{p}\right]\leq Nr_1^{-\theta p}(x)
\end{equation}%
and
\begin{equation*}
\mathbf{E}\left[\sup_{t\le T}| \Gamma_t(x)- \Gamma_t
(y)|^{p}\right]\leq N(r_1^{-p\theta}(x)\vee r_1^{-p\theta}(y))|x-y|^{(\beta' \wedge 1)
p},
\end{equation*}
where  $N=N(d_1,p,N_{0},T,\eta ,\theta)$  is a positive constant. Now, assume that $[\beta]^-\ge 1$. As in the proof of Theorem 3.4 in  \cite{Ku04}, it follows that  $\mathfrak{U}_{t}=\nabla\Gamma_t(\tau,x)$ solves 
\begin{align*}
d\mathfrak{U}_{t}&=\left(\upsilon^{\varrho}_t(X_{t})\mathfrak{U}_{t}+\nabla
\upsilon^{\varrho}_t(X_{t})\nabla X_{t}\Gamma_t +\nabla g^{\varrho}_t
(X_{t})\nabla X_{t}\right)dw^{\varrho}_{t}\\
&\quad +\int_Z\rho_t (
\tilde{H}^{-1}_t(X_{t-},z ),z)\mathfrak{U}_{t-}[\mathbf{1}_{D}(z)q(dt,dz )+\mathbf{1}_E(z)p(dt,dz)] \\ 
&\quad +\int_Z \nabla \rho_t (\tilde{H}^{-1}_t(X_{t-},z),z)\nabla [\tilde{H}%
^{-1}_t(X_{t-})]\Gamma_{t-} [\mathbf{1}_{D}(z)q(dt,dz )+\mathbf{1}_E(z)p(dt,dz)]\\ 
&\quad +\int_Z \nabla h_t(x,\tilde{H}^{-1}_t(X_{t},z
),z) \nabla [\tilde{H}%
^{-1}_t(X_{t-})]][\mathbf{1}_{D}(z)q(dt,dz )+\mathbf{1}_E(z)p(dt,dz)]\\
&\quad +\left(c_t(X_{t})\mathfrak{U}_{t}+\nabla c_t(X_{t})\nabla X_{t}\Gamma_t +\nabla f_t(X_{t})\nabla X_{t}\right)dt,\;\; \tau <t \le T,\\
\mathfrak{U}_{t}&=0,\;\;t\le \tau.
\end{align*}
Recall that by Lemma \ref{lem:equivnorm},  a function $\phi:\bR^{d_1}\rightarrow\mathbf{R}^n$, $n\ge 1$ satisfies $|r^{-\theta}\phi |_{\beta}<\infty$ if an only if $|r^{-\theta}\phi|_0,\ldots,|r^{-\theta}\partial^{\gamma} \phi|_0$, $|\gamma|\le [\beta]^-$, and $[r^{-\theta}\partial^{\gamma}\phi]|_{\{\beta\}^+}$ are finite. 
Estimating as above and using Proposition 3.4 in \cite{LeMi14b}, we obtain that for each $p\geq 2
$, there is a constant $N=N(d_1,d_2,p,N_{0},T,\theta)$  such that for all $x$ and $y$,
\begin{equation*}
\mathbf{E}\left[\sup_{t\le T}| \nabla\Gamma_t(x)|^{p}\right]\leq r_1^{-p\theta}(x)N
\end{equation*}%
and
\begin{equation*}
\mathbf{E}\left[\sup_{t\le T}|\nabla\Gamma_t(x)-\nabla \Gamma_t
(y)| ^{p}\right]\leq N(r_1^{-p\theta}(x)\vee r_1^{-p\theta}(y))|x-y|^{((\beta -1
)\wedge 1)p}.
\end{equation*}
Using induction, we get that for each $p\geq 2$ and all multi-indices $\gamma $ with $0\leq |\gamma |\leq %
[ \beta ] ^{-}$ and all $x$, 
\begin{equation*}
\mathbf{E}\sup_{t\leq T}\left[|\partial ^{\gamma }\Gamma_t
(x)|^{p}\right]\leq  r_1^{-p\theta}(x)N,
\end{equation*}%
and for all multi-indices $\gamma $ with $|\gamma |=[\beta ]^{-}$ and all $%
x,y$, 
$$
\mathbf{E}\left[\sup_{t\leq T}|\partial ^{\gamma }\Gamma_t(x)-\partial
^{\gamma }\Gamma_t(y)|^{p}\right]\leq N(r_1^{-p\theta}(x)\vee r_1^{-p\theta}(y))|x-y|^{(\beta -[ \beta ]^-)p},
$$
for  a constant $N=N(d_1,d_2,p,N_{0},T,\beta,\eta ,\theta).$
It is also clear that for each $p\geq 2$ and all multi-indices $\gamma $ with $0\leq |\gamma |\leq %
[ \beta ] ^{-}$ and all $x$, 
\begin{equation*}
\mathbf{E}\left[\sup_{t\leq T}|\partial ^{\gamma} \Psi_t
(x)|^{p}\right]\leq N,
\end{equation*}%
and for all multi-indices $\gamma $ with $|\gamma |=[\beta ]^{-}$ and all $%
x,y$, 
$$
\mathbf{E}\left[\sup_{t\leq T}|\partial ^{\gamma }\Psi_t(x)-\partial
^{\gamma }\Psi_t(y)|^{p}\right]\leq N|x-y|^{(\beta -[\beta]^-
)p}.
$$
We obtain the existence of  a $D([0,T],\cC_{loc}^{\beta '}( \mathbf{R}%
^{d_1};\bR^{d_2})) $-modification of $ \Gamma(\tau )$
and $\Psi(\tau )$ using estimate \eqref{ineq:estimateFKTransGamPhi}  and Corollary 5.4 in \cite{LeMi14b}. This completes the proof.
\end{proof}

Let $\tilde{\Phi}_t (x)=\tilde{\Phi}_t (\tau, x),$ $(t,x)\in [0,T]\times \mathbf{R}%
^{d_1}$, be the solution of the linear SDE given by
\begin{align*}
d\tilde{\Phi}_{t}(x)& =\left(c_t(X_t(x))\tilde{\Phi}_{t}(x)+f_t(X_{t}(x))\right)dt+\left(\upsilon^{\varrho}_t(X_{t}(x))\tilde{\Phi}_{t}(x)+g^{\varrho}_t(X_{t}(x))\right)dw^{\varrho}_{t}\notag\\
&\quad +\int_{Z}\rho_t (\tilde{H}%
^{-1}_t(X_{t-}(x),z),z)\tilde{\Phi}_{t-}(x,y)[\mathbf{1}_{D}(z)q(dt,dz )+\mathbf{1}_E(z)p(dt,dz)]  \notag\\
& \quad +\int_{Z}h_t(\tilde{H}%
^{-1}_t(X_{t-}(x),z),z)[\mathbf{1}_{D}(z)q(dt,dz )+\mathbf{1}_E(z)p(dt,dz)],\;\;\tau<t\le  T,  \\
\tilde{\Phi}_{t}(x)& =\varphi(x),\;\;t\leq \tau.\notag
\end{align*}

The following is a simple corollary of Lemma \ref{lem:EstimateofFKTranswoinitial}. 

\begin{corollary}
\label{cor:EstimateofFKTranswinitial}
Let Assumptions \ref{asm:smalljumpinfree}$(\bar \beta)$ and  \ref{asm:freeandzero}$(\tilde{\beta})$ hold
for some  $\bar\beta >1\vee \alpha$ and $\tilde{\beta}>\alpha$.   For each stopping time $\tau\le T$ and $\cF%
_{\tau }\otimes \mathcal{B}(\bR^{d_1})$-measurable random field $\varphi$ such that  for some   $\beta'\in [0 ,\bar{\beta}\wedge \tilde{\beta} )$,  $\bP$-a.s.\ $\varphi\in\cC^{\beta'}_{loc}(\mathbf{R}
^{d_1};\bR^{d_2})$,  there is a $D([0,T];\cC_{loc}^{\beta'}(\bR^{d_1},\bR^{d_2})) $-modification of $%
\tilde{\Phi}(\tau)$, also denoted by $\tilde{\Phi}(\tau),$ and $\bP$-a.s.\ for  all 
$(t,x)\in[0,T]\times\bR^{d_1},$ $$\tilde{\Phi}_t (\tau,x)=\Psi_t (x)\varphi (x)+\Gamma%
_t(x).$$ Moreover,  if for some $\theta'\geq 0$ and $\beta'\in [0 ,\bar{\beta}\wedge \tilde{\beta} )$, $\bP$-a.s.\ $r_{1}^{-\theta'}\varphi
\in \cC^{\beta'}(\bR^{d_1};\bR^{d_2})$,
then for each $\epsilon>0$ and $p\geq 2$, there is a constant $N=N(d_1,d_2,p,N_{0},T,\theta,\theta',\beta ',\epsilon)$ such that 
\begin{equation}\label{ineq:EstimateofFKwithInitial}
\mathbf{E}\left[\sup_{t\le T}|r_{1}^{-(\theta\vee \theta')-\epsilon }\tilde{\Phi}_t(\tau)|
_{\beta '}^{p}\big|\cF_{\tau }\right]\leq N(|  r_{1}^{-\theta'}\varphi%
|_{\beta'}^{p}+1).
\end{equation}
\end{corollary}

Now we are ready to state our main result concerning fully-degenerate SIDEs and their connection with linear transformations of inverse flows of jump SDEs.  

\begin{proposition}\label{prop:RepresentationSmallJumps}
Let Assumptions \ref{asm:smalljumpinfree}$(\bar \beta)$ and  \ref{asm:freeandzero}$(\tilde{\beta})$ hold
for some  $\bar\beta >1\vee \alpha$ and $\tilde{\beta}>\alpha$.    For each stopping time $\tau\le T$ and $\cF%
_{\tau }\otimes \mathcal{B}(\bR^{d_1})$-measurable random field $\varphi$ such that   for some  $\beta '\in (\alpha ,\bar{\beta}\wedge \tilde{\beta} )$  and $\theta'\geq 0$, $\bP$-a.s.\
$ r_{1}^{-\theta'}\varphi\in\cC^{\beta'}(\mathbf{R}
^{d_1};\bR^{d_2})$,  we have that  $\bP$-a.s.\ $\tilde{\Phi }(\tau,X^{-1}(\tau ))\in 
D([0,T];\cC_{loc}^{\beta '}(\bR^{d_1};\bR^{d_2}))$  and $v_t(x)=v_t(\tau,x)=\tilde{\Phi}_t (\tau,X_{t}^{-1}(\tau,x))$ solves  \eqref{eq:FKTransSIDE}. Moreover,   for each $\epsilon>0$ and $p\geq 2$,
\begin{equation}\label{ineq:RepresentationEstimateproof}
\mathbf{E}\left[\sup_{t\le T}| r_{1}^{-(\theta\vee \theta')-\epsilon }v_t(\tau)| _{\beta'}^{p}\big|\cF_{\tau }%
\right] \leq N(| r_{1}^{-\theta'}\varphi%
|_{\beta'}^{p}+1),
\end{equation}
for a constant  $N=N(d_1,d_2,p,N_{0},T,\beta ',\eta ,\epsilon,\theta,\theta')$.
\end{proposition}
\begin{proof}
Fix a stopping time  $\tau\le T$ and random field $\varphi$ such that   for some   $\beta '\in (\alpha ,\bar{\beta}\wedge \tilde{\beta} )$  and $\theta'\geq 0$, $\bP$-a.s.\
$ r_{1}^{-\theta'}\varphi\in\cC^{\beta'}(\mathbf{R}
^{d_1};\bR^{d_2})$. By virtue of Corollary \ref{cor:EstimateofFKTranswinitial} and Theorem 2.1 in \cite{LeMi14b}, $\bP$-a.s.\ $$\tilde{\Phi}(\tau,X^{-1}(\tau)) \in D( [0,T];\cC_{loc}^{\beta'}( \bR^{d_1},\bR^{d_2})).$$  Then using the Ito-Wenzell formula  (Proposition \ref{prop:Ito-Ventzel}) and following a simple calculation, we obtain  that $v_t(\tau,x):=\tilde{\Phi}_{t}(\tau,X^{-1}_{t}(\tau,x)) $ solves  \eqref{eq:FKTransSIDE}.  
By Theorem 2.1 in \cite{LeMi14b} and Corollary \ref{cor:EstimateofFKTranswinitial}, for each $\epsilon>0$ and  $p\geq 2,$  there exists a constant $ N= N(d_1,p,N_{0},T,\beta',\eta ,\epsilon)$  such that 
\begin{equation} \label{ineq:inverseflowmomentest}
\mathbf{E}[\sup_{t\le T}|r_1^{-(1+\epsilon)}X_t^{-1}(\tau)|_{\beta'}^{p}]+\mathbf{E}[\sup_{t\le T}|r_1^{-\epsilon}\nabla X_t^{-1}(\tau)|_{\beta'-1}^{p}] \le 
 N.
 \end{equation}
Therefore applying Lemma \ref{lem:GeneralCompWeightedlemma} and H{\"o}lder's inequalty and using the estimates \eqref{ineq:inverseflowmomentest} and \eqref{ineq:EstimateofFKwithInitial}, we obtain \eqref{ineq:RepresentationEstimateproof}, which completes the proof.
\end{proof}

\subsection{Adding uncorrelated part (Proof of Theorem \ref{thm:repwcorrec})}\label{sec:uncorrelated}
\begin{proof}[Proof of Theorem \ref{thm:repwcorrec} ]
Fix a stopping time  $\tau\le T$ and random field $\varphi$ such that   for some   $\beta '\in (\alpha ,\bar{\beta}\wedge \tilde{\beta} )$  and $\theta'\geq 0$, $\bP$-a.s.\
$ r_{1}^{-\theta'}\varphi\in\cC^{\beta'}(\mathbf{R}
^{d_1};\bR^{d_2})$.   Consider the system of SIDEs given by 
\begin{align}\label{eq:UncorrelatedSIDE} 
d\tilde{v}^l_t &=\left((\mathcal{L}^{1;l}_t+\mathcal{L}^{2;l}_t)\tilde{v}_t+\mathbf{1}_{[1,2]}(\alpha)\hat{b}^i_t\partial_iu_t^{l}+\hat{c}^{l\bar l}_tu_t^{\bar l}(x)+\hat{f}^l_t\right)dt+\left(\mathcal{N}^{1;l\varrho}_t\tilde{v}_t+g^{l\varrho}_t\right)dw^{1;\varrho}_t \notag\\
&\quad +\mathcal{N}^{2;l\varrho}_t\tilde{v}_tdw^{2;\varrho}_t +\int_{Z^{1}}\left(\mathcal{I}^{1;l}_{t,z }\tilde{v}_{t-}+h^l_t(z)\right)[\mathbf{1}_{D^{1}}(z)q^{1}(dt,dz )+\mathbf{1}_{E^{1}}p^{1}(dt,dz)]\notag\\
&\quad +\int_{Z^{2}}\mathcal{I}^{2;l}_{t,z }\tilde{v}_{t-}[\mathbf{1}_{D^{2}}(z)q^{2}(dt,dz )+\mathbf{1}_{E^{2}}(z)p^{2}(dt,dz)]\,\;\;\tau
< t\leq T,   \\
\tilde{v}^l_t&=\varphi^l ,\;\;t\leq \tau , \;\;l\in\{1,\ldots,d_2\}, \notag
\end{align}%
where for $\phi \in C^{\infty }_{c}( \bR^{d_1};\mathbf{R}
^{d_2})$ and $l\in\{1,\ldots,d_2\}$,
\begin{align*}
\mathcal{N}^{2;l\varrho}_t\phi(x)&:=\mathbf{1}_{\{2\}}(\alpha)\sigma_t^{2;i\varrho}(x)\partial_i\phi^l(x)+\upsilon^{2;l\bar l\varrho}_t(x)\phi^{\bar l}(x), \;\;\varrho\ge 1, \\
\mathcal{I}^{2;l}_{t,z }\phi (x) &:=(I_{d_2}^{l\bar l}+\rho^{2;l\bar l}_{t}(x,z ))\phi^{\bar l} (x+H^{2}_t(x,z
))-\phi^l (x).
\end{align*}
By Proposition \ref{prop:RepresentationSmallJumps},  $\bP$-a.s.\ $\Phi (\tau,X^{-1}(\tau ))\in 
D([0,T];\cC_{loc}^{\beta '}(\bR^{d_1};\bR^{d_2}))$  and $\tilde{v}_t\allowbreak (\tau,x)\allowbreak =\Phi_{t}(\tau,\allowbreak X_{t }^{-1}(\tau,x ))$ solves  \eqref{eq:FKTransSIDE}. We write $v_t(x)\allowbreak=v_t(\tau,x)$. Moreover,  for each $\epsilon>0$ and $p\geq 2$, 
\begin{equation}\label{ineq:estimateunconditioned}
\mathbf{E}\left[\sup_{t\le T}| r_{1}^{-(\theta\vee \theta')-\epsilon}\tilde{v}_t(\tau)| _{\beta'}^{p}\big|\cF_{\tau }%
\right] \leq N(|  r_{1}^{-\theta'}\varphi%
|_{\beta'}^{p}+1),
\end{equation}
where  $N=N(d_1,d_2,p,N_{0},T,\beta ',\eta^{1} ,\eta^{2},\epsilon,\theta,\theta')$  is a positive constant.
Without loss of generality we will assume that  for all $\omega$ and $t$, $ |r_{1}^{-\theta'}\varphi%
|_{\beta'}\le N$, since we can always multiply the equation by indicator function.
For each $n\in\mathbf{N}\cup\{0\}$, let $C ^{n}_{loc}(\bR^{d_1};\bR^{d_2})$ be the separable Fr\'echet space of $n$-times continuously differentiable functions $f:\bR^{d_1}\rightarrow \bR^{d_2}$ endowed with the countable  set of semi-norms given by
\begin{equation}\label{eq:integerHolder}
|f|_{n,int}=\sum_{0 \le|\gamma|\le n}\sup_{|x|\le k}|\partial^{\gamma}f(x)|, \;\;k\in\mathbf{N}.
\end{equation}
Owing to Lemma \ref{lem:optionalprojection}, there is a 
the family of measures  $E^{t}_{\omega}(dU)$, $(\omega,t)\in \Omega\times [0,T]$ on   $D([0,T];$ $C ^{[\beta]^-}_{loc}(\bR^{d_1};\bR^{d_2}))$,   corresponding to $\mathfrak{A}=\tilde{v}$  such that for all  bounded  $G:\Omega \times [0,T]\times [0,T]\times D([0,T];C ^{[\beta]^-}_{loc}(\bR^{d_1};\bR^{d_2}))\rightarrow \bR^{d_2}$ that are  $\mathcal{O_T}\times \cB\left(
[0,T]\right) \times \mathcal{B}(D( [0,T]; C ^{[\beta]^-}_{loc}(\bR^{d_1};\bR^{d_2}))) $ measurable, $\bP$-a.s.\ for all $t$, we have
$$
E^{t}[G_t(t,\tilde{v})]=\int_{D([0,T];C_{loc}^{[\beta']^-}(\bR^{d_1};\bR^{d_2}))}G_t(t,U)E^{t}(dU)=\mathbf{E}\left[G_t(t,\tilde{v})|\cF_{t}\right],
$$
where the right-hand-side is the c\`{a}dl\`{a}g  modification of the  conditional expectation.
Set
$$
\hat{u}_t(x)=\hat{u}_t(\tau,x)=E^{t}[\tilde{v}_t(\tau,x)]=\int_{D( [0,T];C_{loc}^{[\beta']^-}(\bR^{d_1};\bR^{d_2})) } U_t(x)E^{t}(dU).
$$
Let $\lambda=(\theta\vee \theta')+\epsilon$. 
We claim that for all multi-indices $\gamma$ with $|\gamma|\le [\beta]^-$, $\bP$-a.s.\ for all $t$ and  $x$, 
$$
\partial^{\gamma} [r_1^{-\lambda}(x) \hat{u}_t(x)]=\int_{D( [0,T];C_{loc}^{[\beta']^-}(\bR^{d_1};\bR^{d_2})) } \partial^{\gamma}[r_1^{-\lambda}(x)U_t(x)] E^{t}(dU)=E^{t}[\partial^{\gamma}[r_1^{-\lambda}(x)\tilde{v}_t(x)]].
$$
Indeed, since $$
M_t=E^t\left[\sup_{s\le T}|\partial^{\gamma}[r_1^{-\lambda}\tilde{v}_s]|_0 \right], \;\;t\in [0,T],
$$
is a $\left( \mathbf{F},\mathbf{P}\right) $ martingale, we have
\begin{equation}\label{ineq:doobsquareMart}
\mathbf{E}\left[\sup_{t\le T}|M_t|^2\right]\le 4 \mathbf{E}\left[|M_T|^2\right]\le  4\mathbf{E}\left[\sup_{t\le T}|\partial^{\gamma}[r_1^{-\lambda}\tilde{v}_t]|_0^2|\right]<\infty,
\end{equation}
and hence $\bP$-a.s.\ for all $t$,
$$
\int_{D( [0,T];C_{loc}^{[\beta']^-}(\bR^{d_1};\bR^{d_2})) }\sup_{s\le T,x\in \bR^{d_1}}|\partial^{\gamma}[r_1^{-\lambda}(x)U_s(x)]| E^{t}(dU)=E^t\left[\sup_{t\le T}|\partial^{\gamma}[r_1^{-\lambda}\tilde{v}_t]|_0 \right]<\infty.
$$
Similarly, since $\mathbf{E}\left[\sup_{t\le T}|r_1^{-\lambda}\tilde{v}_t|_{\beta'}^2\right]<\infty$, $\bP$-a.s.\ for each $x$ and $y$,
\begin{align*}
\frac{|\partial^{\gamma} [r_1^{-\lambda}(x) \hat{u}_t(x)]-\partial^{\gamma} [r_1^{-\lambda}(y) \hat{u}_t(y)]|}{|x-y|^{\{\beta'\}^+}}
&\le E^{t}\left[\frac{|\partial^{\gamma}[r_1^{-\lambda}(x)\tilde{v}_t(x)]-\partial^{\gamma}[r_1^{-\lambda}(y)\tilde{v}_t(y)]|}{|x-y|^{\{\beta'\}^+}}\right]\\
&\le E^{t}[|r_1^{-\lambda}\tilde{v}_t|_{\beta'}],
\end{align*}
and hence, $\bP$-a.s.\ 
$$
\sup_{t\le T}|r_1^{-\lambda}\hat{u}_t|_{\beta'}\le \sup_{t\le T}E^{t}\left[\sup_{t\le T}|r_1^{-\lambda}\tilde{v}_t|_{\beta'}\right]<\infty.
$$
Thus, $\bP$-a.s.\ $r_1^{-\lambda }(\cdot)\hat{u}(\tau)\in D( [0,T];\cC^{\beta'}(\bR^{d_1};\bR^{d_2}))$ and  \eqref{eq:estofoptionalprojection} follows from \eqref{ineq:estimateunconditioned} (see the argument \eqref{ineq:doobsquareMart}).
 For each $l\in\{1,\ldots,d_2\}$, let%
\begin{align*}
\mathcal{A}^l_t(x) &=\varphi^l (x)+\int_{]\tau,\tau\vee t]}\left((\mathcal{L}^{1;l}_s+\mathcal{L}^{2;l}_s)\hat{u}_s(x)+\mathbf{1}_{[1,2]}(\alpha)\hat{b}^i_s(x)\partial_i\hat{u}_s^{l}(x)+\hat{c}^{l\bar l}_s(x)\hat{u}_s^{\bar l}(x)+\hat{f}^l_s(x)\right)ds\\ &\quad+\int_{]\tau,\tau\vee t]}\left(\mathcal{N}^{1;l\varrho}_s\hat{u}_s(x)+g^{l\varrho}_s(x)\right)dw^{1;\varrho}_s\notag  \\
&\quad +\int_{]\tau,\tau\vee t]}\int_{Z^{1}}\left(\mathcal{I}^{1;l}_{s,z }\hat{u}_{s-}(x)+h^l_s(x,z)\right)[\mathbf{1}_{D^{1}}(z)q^{1}(ds,dz )+\mathbf{1}_{E^{1}}(z)p^{1}(ds,dz)].
\end{align*}%
By Theorem 12.21 in \cite{Ja79}, the representation property holds for $%
\left( \mathbf{F},\mathbf{P}\right) $, and hence every bounded $\left( \mathbf{F},\mathbf{P}\right) $- martingale issuing from zero can be represented
as%
\begin{equation*}
M_{t}=\int_{]0,t]}o^{\varrho} _{s}dw_{s}^{1;\varrho}+\int_{]0,t]}\int_{Z^{1}} e_s(z)q^{1}(ds,dz),\;\;t\in [0,T],
\end{equation*}%
where 
\begin{equation*}
\mathbf{E}\int_{]0,T]}|o _{s}|^{2}ds+\mathbf{E}\int_{]0,T]}\int_{Z^{1}}|e_s(z)|^{2}\pi ^{1}(dz)
ds<\infty .
\end{equation*}%
Then for an arbitrary $\mathbf{F}$-stopping time $\bar{\tau}\leq T$ and bounded $\left( \mathbf{F},\mathbf{P}\right) $- martingale, applying It\^{o}'s product rule and taking the expectation, we obtain 
\begin{equation*}
\mathbf{E}\tilde{v}_{\bar{\tau}}(\tau,x)\bar M_{\bar{\tau}}=\mathbf{E}%
\mathcal{A}_{\bar{\tau}}(x)\bar M_{\bar{\tau}}.
\end{equation*}%
Since the optional projection is unique, $\bP$-a.s.\ for all $t$ and $x$, $\hat{u}_t(x)=\mathcal{A}_t(x)$. This completes the proof.
\end{proof}
\subsection{Interlacing a sequence of large jumps (Proof of Theorem \ref{thm:Existwocorrec})}\label{s:Large Jumps}
\begin{proof}[Proof of Theorem \ref{thm:Existwocorrec}]
Fix a stopping time  $\tau\le T$ and random field $\varphi$ such that   for some   $\beta '\in (\alpha ,\bar{\beta}\wedge \tilde{\beta} )$  and $\theta'\geq 0$, $\bP$-a.s.\
$ r_{1}^{-\theta'}\varphi\in\cC^{\beta'}(\mathbf{R}
^{d_1};\bR^{d_2})$.  For any $\delta>0$, we can rewrite  \eqref{eq:FullSIDE} as 
\begin{align}\label{eq:Largejump}
du^l_t&=\left((\bar{\mathcal{L}}_{t}^{1;l}+\mathcal{L}%
_{t}^{2;l})u_t+\mathbf{1}_{[1,2]}(\alpha)\bar{b}^i_t\partial_iu_t^{l}+\bar{c}^{l\bar l}_tu_t^{\bar l}+f^l_t\right)dt+\left(\mathcal{N}_{t}^{1;l\varrho}u_t+g^{l\varrho}_t\right)dw_{t}^{1;\varrho}\notag \\
&\quad +\int_{Z^{1} }\left(\bar{\mathcal{I}}
_{t,z}^{1;l}u_{t-}+\bar h^l_t(z)\right)[1_{D^{1}}(z)q^{1}(dt,dz)+1_{E^{1}}(z)p^{1}(dt,dz)] 
\notag \\
&\quad  +\int_{Z^{1}}\left(\mathbf{1}
_{(D^{1}\cup E^{1})\cap \{K_t^{1}> \delta \}}(z)+\mathbf{1}_{V^{1}}(z)\right)\left(\mathcal{I}
_{t,z}^{1;l}u_{t-}+h^l_t(z)\right)p^{1}(dt,dz),\;\;\tau<t\le T , \notag  \\
u^l_t&=\varphi^l,\;\;t\leq \tau , \;\;l\in\{1,\ldots,d_2\},
\end{align}
where for $\phi \in C^{\infty }_{c}( \bR^{d_1};\mathbf{R}
^{d_2})$ and $l\in\{1,\ldots,d_2\}$,
\begin{align*}
\bar{\mathcal{L}}_{t}^{1;l}\phi(x)&:=\mathbf{1}_{\{2\}}(\alpha)\frac{1}{2}\sigma_t
^{1;i\varrho}(x)\sigma_t
^{1;j\varrho}(x)\partial_{ij}\phi^l(x)+\mathbf{1}_{\{2\}}(\alpha)\sigma^{k;i\varrho}_t(x)\upsilon^{1;l\bar l\varrho}_t(x)\partial_i\phi^{\bar l}(x)\\
&\quad +\int _{D^{1}}\bar\rho^{1;l\bar l}_t(x,z )\left(\phi^{\bar l} (x+\bar H^{1}_t(x,z ))-\phi^{\bar l} (x)\right)\pi^{1}(dz ) \\
&\quad+\int _{D^{1}}\left(\phi^{ l} (x+\bar H^{1}_t(x,z ))-\phi^{ l} (x)-%
\mathbf{1}_{(1,2]}(\alpha)\bar H^{1;i}_t(x,z )\partial_i\phi^l(x)\right)\pi^{1}(dz ),
\end{align*}
\begin{align*}
\bar{\cI}^{1}_{t,z}\phi^l(x)&=(I_{d_2}^{l\bar l}+\mathbf{1}_{ \{ K_t^{1}\le\delta\}}(z)\rho^{1;l\bar l}_{t}(x,z ))\phi^{\bar l} (x+\mathbf{1}_{ \{ K_t^{1}\le\delta\}}(z)H^{1}_t(x,z
))-\phi^l (x),\\
\bar H^{1}&:=\mathbf{1}_{ \{ K_t^{1}\le\delta\}}H^{1}, \; \bar{\rho}^{1}:=\mathbf{1}_{ \{ K_t^{1}\le\delta\}}\rho^{1},\; \bar h:=\mathbf{1}_{ \{ K_t^{1}\le\delta\}}h,\\
\bar b_t^i(x)&:=b_t^i(x)-\int_{D^{1}\cap   \{ K_t^{1}>\delta\}}
\mathbf{1}_{(1,2]}(\alpha)H^{1;i}_t(x,z )\pi^{1}(dz ),\\
\bar c_t^{l\bar l}(x)&:=c_t^{l\bar l}(x)-\int _{D^{1}\cap \{ K_t^{1}>\delta \}}\rho^{1;l\bar l}_t(x,z )\pi^{1}(dz ).
\end{align*}
For an arbitrary stopping time $\tau'\leq T$ and $\cF%
_{\tau'}\otimes \mathcal{B}(\bR^{d_1})$-measurable random field  $%
\varphi ^{\tau '}:\Omega \times \bR^{d_1}\rightarrow \mathbf{R}%
^{d_2}$ satisfying  for some $\theta(\tau')>0$, $\bP$-a.s.\ $r_1^{-\theta(\tau')}\varphi^{\tau'} \in \cC^{\beta '}(\mathbf{R}%
^{d_1};\bR^{d_2}) $,   consider the system of SIDEs on $[0,T]\times \bR^{d_1}$ given by
\begin{align}\label{eq:LargeJumpArbitrarystopping}
dv^l_t&=\left((\bar{\mathcal{L}_{t}}^{1;l}+\mathcal{L}%
_{t}^{2;l})v_t+\mathbf{1}_{[1,2]}(\alpha)\bar{b}^i_t\partial_iv_t^{l}+\bar{c}^{l\bar l}_tv_t^{\bar l}+f^l_t\right)dt+\left(\mathcal{N}_{t}^{1;l\varrho}v_t+g^{l\varrho}_t\right)dw_{t}^{1;\varrho}\notag \\
&\quad +\int_{Z^{1}}\left( \bar{\mathcal{I}}
_{t,z}^{1;l}u_{t-}+\bar h^l_t(z)\right)[1_{D^{1}}(z)q^{1}(dt,dz)+1_{E^{1}}(z)p^{1}(dt,dz)],\;\;\tau'<t\le T ,\notag\\
v ^l_t&=\varphi^{\tau';l},\;\;t\leq \tau' , \;\;l\in\{1,\ldots,d_2\}.
\end{align}
Set $\bar H^{2}=H^{2}$ and $\bar \rho^{2}=\rho^{2}$. 
In order to invoke Theorem \ref{thm:repwcorrec} and   obtain a unique solution $v_t=v_t(\tau',x)=v_t(\tau',\varphi^{\tau'},x)$  of \eqref{eq:LargeJumpArbitrarystopping}, we will  show that  for all $\omega$ and $t$,
\begin{equation}\label{ineq:estimatesofmodifiedterms}
|r_1^{-1}\tilde{b}_t|_0 +|\nabla \tilde{b}_t|_{\bar{\beta}-1} + |\tilde{c}_t|_{\tilde{\beta}}+|r^{-\theta} \tilde{f}|_{\tilde{\beta}}\le N_0,
\end{equation}
where
\begin{align*}
\tilde{b}^i_t(x):&=\mathbf{1}_{[1,2]}(\alpha)\bar{b}^i_t(x)-\sum_{k=1}^{2}\mathbf{1}_{\{2\}}(\alpha)\sigma_{t}^{k;j\varrho}(x) \partial_{j}\sigma^{k;i\varrho}_{t}(x)\\
&\quad -\sum_{k=1}^{2}\int_{D^{k}} \left(\mathbf{1}_{(1,2]}(
\alpha)\bar H^{k;i}_{t}(x,z )-\bar H^{k;i}_t(\tilde{\bar H}_t^{k;-1}(x,z ),z)\right)\pi^{k} (dz ),\\
\tilde{c}^{l\bar l}_t(x):&=\bar{c}^{l\bar l}_t(x)- \sum_{k=1}^2\mathbf{1}_{\{2\}}(\alpha)\sigma
^{k;i\varrho}_t(x)\partial_i\upsilon^{k;l\bar l \varrho}_t(x)\\
&\quad-\sum_{k=1}^2\int _{D^{k}}\left(\bar \rho^{k;l\bar l}_t (x,z)-\bar \rho^{k;l\bar l}_t ( \tilde{\bar H}^{k;-1}_t(x,z),z)\right)
\pi^{k} (dz),\\
\tilde{f}^l_t(x):&=f^l_t(x)-\sigma ^{1;j\varrho}_t(x)\partial_jg^l_t(x) -\int_{D^{1}} \left(\bar h^l_t(x,z)- \bar h^l_t(\tilde{\bar H}^{1;-1}_t(x,z),z)\right)\pi^{1}(dz).
\end{align*}
Owing to   Assumption \ref{asm:regclassicalwocorrec}$(\bar \beta,\delta^1,\delta^2,\mu^1,\mu^2)$, we easily deduce that there is a constant  $N$ $=N( d_1,$ $N_0,$ $ \bar{\beta} )$ such that for each $k\in \{1,2\}$ and  all $\omega$ and $t$, 
$$
|\sigma_{t}^{k;j\varrho}\partial_{j}\sigma^{k;\varrho}_{t}|_{\bar \beta}+|\sigma
^{k;j\varrho}_t\partial_ja^{k;\varrho}_t(x)|_{\bar \beta }+|\sigma ^{1;j\varrho}_t\partial_jg^{\varrho}_t|_{\bar \beta }\le N, \textrm{ if } \alpha=2. 
$$
Since $|\nabla \bar H_t^{1}|_0\le \delta$,  for any fixed $\eta^{1}<1$, for all $(\omega ,t,x,z)\in \{(\omega ,t,x,z)\in \Omega \times
[ 0,T]\times\bR^{d_1}\times (D^{1}\cup E^{1}):|\nabla \bar H_t ^{1}(\omega,x,z)|>\eta^{1} \},$ 
$$
\left|\left( I_{d_1}+\nabla
H^{1}_t(\omega,x,z)\right) ^{-1}\right|\leq \frac{1}{1-\delta}.
$$
 Appealing to Assumption \ref{asm:regclassicalwocorrec}$(\bar \beta,\delta^1,\delta^2,\mu^1,\mu^2)$ and applying Lemma \ref{lem:compositediffeo}, we obtain that  there is a constant $N=N(d_1,d_2,N_0)$ such that for each $k\in \{1,2\}$ and all $\omega,t,$ and $z$, 
\begin{align*}
| \bar H^{k;i}_{t}(z)-\bar H^{k;i}_t(\tilde{\bar H}_t^{k;-1}(z),z)|_{\bar \beta}
&\leq N(K^k_t(z)+\bar K^k_t(z))^2+N\mathbf{1}_{(0,1]}(\{\bar \beta\}^++\delta^k)
\tilde{K}^k_t(z)K^k_t(z)^{\delta^k} \\
 &\quad +N\mathbf{1}_{(1,2]}(\{\bar \beta\}^++\delta^k)\left(\tilde{K}^k_t(z)K^k_t(z)^{\delta^k}+
\bar K^k_t(z)^2 \right),\\
| \bar \rho^{k}_{t}(z)-\bar \rho^{k}_t(\tilde{\bar H}_t^{k;-1}(z),z)|_{\bar \beta}&\le Nl^{k}_t(z)(K^k_t(z)+\bar K^k_t(z))+N\mathbf{1}_{(0,1]}(\{\bar\beta \}^++\mu^k)
\tilde{l}^k_t(z)K^k_t(z)^{\mu^k} \\
&\quad +N\mathbf{1}_{(1,2]}(\{\bar \beta\}^++\mu^k)\left(\tilde{l}^k_t(z)K^k_t(z)^{\mu^k}+
 l^k_t(z)\bar K^k_t(z) \right),
\end{align*}
and
\begin{align*}
| r_1^{-\theta}\bar h_{t}(z)-r_1^{-\theta}\bar h_t(\tilde{\bar H}_t^{1;-1}(z ),z)|_{\bar \beta}&\le Nl^{1}_t(z)(K^1_t(z)+\bar K^k_t(z))+N\mathbf{1}_{(0,1]}(\{\bar \beta\}^++\mu^1)
\tilde{l}^k_t(z)K^1_t(z)^{\mu^1} \\
&\quad +N\mathbf{1}_{(1,2]}(\{\bar \beta\}^++\mu^1)\left(\tilde{l}^k_t(z)K^1_t(z)^{\mu^1}+
 l^k_t(z)\bar K^1_t(z) \right).
\end{align*}
Moreover, using Lemma  \ref{lem:compositediffeo},  we find that  there is a constant $N=N(d_1,d_2,N_0)$ such that for each $k\in \{1,2\}$, and all $\omega,t,$ and $z$, 
 $$
 |r_1^{-1}\bar H^{k}_t(\tilde{\bar H}_t^{k;-1}(z ),z)|_0\le |r_1^{-1}H^k|_0, \quad |\nabla [\bar H^{k;i}_t(\tilde{\bar H}_t^{k;-1}(z ),z)]|_{\bar{\beta}} \le |\nabla H^k|_{\bar \beta-1}.
 $$
Combining the above estimates and using H{\"o}lder's inequality and  the integrability properties of $l^{k}_t(z)$ and $K^{k}_t(z),$ we obtain \eqref{ineq:estimatesofmodifiedterms}. Therefore,
by Theorem \ref{thm:repwcorrec}, for each stopping time $\tau'\le T$ and and $\cF%
_{\tau'}\otimes \mathcal{B}(\bR^{d_1})$-measurable random field  $%
\varphi ^{\tau '}$ satisfying for some $\theta(\tau')>0$, $\bP$-a.s.\ $r_1^{-\theta(\tau')}\varphi^{\tau'} \in \cC^{\beta '}(\mathbf{R}%
^{d_1};\bR^{d_2}) $, there exists a unique solution $v_t(x)=v_t(\tau',\varphi ^{\tau '},x) $ of \eqref{eq:LargeJumpArbitrarystopping} such that 
\begin{equation}\label{ineq:mainestimateinterlacer}
\mathbf{E}\left[\sup_{t\le T}| r_{1}^{-\theta(\tau')\vee \theta-\epsilon }v_t(\tau')| _{\beta'}^{p}\big|\cF_{\tau' }%
\right] \leq N(|  r_{1}^{-\theta(\tau')}\varphi^{\tau'}
|_{\beta'}^{p}+1),
\end{equation}
where  $N=N(d_1,d_2,p,N_{0},T,\beta ',\eta^{1} ,\eta^{2},\epsilon,\theta,\theta(\tau'))$  is a positive constant.
Let%
\begin{equation*}
A_{t}=\int_{]0,t]}\int_{Z^{1}}\left(\mathbf{1}
_{(D^{1}\cup E^{1})\cap \{K_s^{1}> \eta^{1} \}}(z)+\mathbf{1}_{V^{1}}(z)\right)p^{1}(ds,dz),\;\;t\leq T.
\end{equation*}%
Define a sequence of stopping times $(\tau _{n})_{n\geq 0}$ recursively
by $\tau _{1}=\tau$ and 
\begin{equation*}
\tau _{n+1}=\inf (t>\tau _{n}:\Delta A_{t}\neq 0)\wedge T.
\end{equation*}%
We obtain  the existence of a unique solution $u=u(\tau)$  
of  \eqref{eq:Largejump} in $\mathfrak{C}^{\beta '}(\bR^{d_1};\bR^{d_2})$ by interlacing solutions of \eqref{eq:LargeJumpArbitrarystopping} along the
sequence of stopping times $\left( \tau _{n}\right) $. For $(\omega,t)\in [[
0,\tau _{1})),$  we set $u_t(\tau,x)=v_t(\tau,\varphi ,x)$ and note that 
$$
\mathbf{E}\left[\sup_{t\le \tau_1}| r_{1}^{-\theta'\vee \theta-\epsilon }u_t(\tau)| _{\beta'}^{p}\big|\cF_{\tau }%
\right] \leq N(|  r_{1}^{-\theta'}\varphi
|_{\beta'}^{p}+1).
$$
For each $\omega$ and $x$,  we set 
\begin{align*}
u_{\tau _{1}}(x)&=u_{\tau _{1}-}(x) +\int_{Z^{1}}\left(\mathbf{1}
_{(D^{1}\cup E^{1})\cap \{K^{1}> \eta^{1} \}}(t,z)+\mathbf{1}_{V^{1}}(z)\right)\left(\mathcal{I}
_{t,z}^{1}u_{\tau_1-}(x)+h^l_{\tau_1}(x,z)\right)p^{1}(\{\tau_1\},dz).
\end{align*}
By virtue of Lemma \ref{lem:GeneralCompWeightedlemma},  there is a constant $N=N(d_1,d_2,\theta,\theta',\zeta_{\tau_1}(z),\beta')$ 
$$
|u_{\tau_1-}\circ \tilde{H}^{1}_{\tau_1}(z)\cdot r_1^{-\xi_{\tau_1}(z)(\theta\vee \theta'+\epsilon+\beta')}|_{\beta'}\le N |r_1^{-\theta\vee \theta'-\epsilon}u^l_{\tau_1-}|_{\beta'},
$$
and hence
$$
|r_1^{-\lambda_1}u_{\tau_1}(x)|_{\beta'}\le N |r_1^{-\theta\vee \theta'-\epsilon}u^l_{\tau_1-}|_{\beta'}+\zeta_{\tau_1}(z),
$$
where
$$
\lambda_1=(\xi_{\tau_1}(z)(\theta\vee \theta'+1+\epsilon+\beta')) \vee \theta \vee (\theta\vee \theta'+\epsilon).
$$
  We then proceed inductively, each time making use of the estimate \eqref{ineq:mainestimateinterlacer},  to obtain a unique solution $u=u(\tau)$ of \eqref{eq:Largejump}, and hence \eqref{eq:FullSIDE}, in $\mathfrak{C}^{\beta ^{\prime }}( \bR^{d_1};\bR^{d_2})$. This completes the proof of Theorem \ref{thm:Existwocorrec}. 
\end{proof}
\section{Appendix}

\subsection{Martingale and point measure measure moment estimates}

Set $(Z,\mathcal{Z},\pi )=(Z^{1},\mathcal{Z}^{1},\pi ^{1})$, $%
p(dt,dz)=p^{1}(dt,dz)$, and $q(dt,dz)=q^{1}(dt,dz)$. We will make use of the following moment estimates to derive the estimates of  $\Gamma_t$ and $\Psi_t$  in Lemma \ref{lem:EstimateofFKTranswoinitial}.  The notation $ a\underset{p}{\sim }b$ is used to indicate that the quantity $a$
 is bounded above and below by a constant depending
only on $p$   times $b$. 

\begin{lemma}
\label{i:Lp Burkholder for Random Measure} Let $h:\Omega \times[ 0,T]\times Z\rightarrow \mathbf{%
R}^{d_1}$ be $\mathcal{P}_{T}\otimes \mathcal{Z}$-measurable
\begin{tightenumerate}
\item For each stopping time $\tau\le T$ and $p\ge 2$,
\begin{align*}
\mathbf{E}\left[\sup_{t\leq \tau }\left\vert \int_{]0,t]
}\int_{Z}h_s(z)q(ds,dz)\right\vert ^{p}\right]&\underset{p}{\sim }\mathbf{E}\left[
\int_{]0,\tau] }\int_{Z}\left\vert h_s(z)\right\vert ^{p}\pi (dz)ds\ \right]\\
&\quad +\mathbf{E}\left[\left(
\int_{]0,\tau] }\int_{Z}\left\vert h_s(z)\right\vert ^{2}\pi (dz)ds\right)
^{p/2}\right].
\end{align*}
\item For each stopping time $\tau\le T$ and  $p\ge 1$,  
 \begin{align*}
\mathbf{E}\left[\sup_{t\leq \tau }\left( \int_{]0,t]
}\int_{Z}|h_s(z)|p(ds,dz)\right)^{p}\right]&\underset{p}{\sim }\mathbf{E}\left[
\int_{]0,\tau] }\int_{Z}\left\vert h_s(z)\right\vert ^{p}\pi (dz)ds\right]\\
&\quad +\mathbf{E}\left[\left(
\int_{]0,\tau] }\int_{Z}|h_s(z)|\pi (dz)ds\right) ^{p}\right] ,
\end{align*}%
\end{tightenumerate}
\end{lemma}

\begin{proof}
We will only prove part (2), since part (1) is
well-known (see, e.g., \cite{Ku04}) and it follows from (2) by the Burkholder-Davis-Gundy inequality.  Assume that $h_t(\omega,z)>0$ for all $\omega,t$ and $z$. Let 
\begin{equation*}
A_{t}=\int_{]0,t]}\int_{Z}h_s(z)p(ds,dz) \quad \textrm{and} \quad L_{t}=\int_{]0,t]}\int _Zh_s(z)\pi
(dz)ds,\;\;t\leq T.
\end{equation*}
It suffices to prove the claim for $p>1$, since the case $p= 1$ is obvious. Fix an arbitrary  stopping time $\tau \le T$ and $p>1$.   For all $\omega$ and $t$, we have
$$A_{t}^{p}=\sum_{s\leq t}\left[ \left( A_{s-}+\Delta A_{s}\right)
^{p}-A_{s-}^{p}\right]. $$
Thus, using the inequality
$$
b^p\le (a+b)^p-a^p\le p(a+b)^{p-1}b\le p2^{p-1}[a^{p-1}b+b^p], \;\;a,b\ge 0,
$$
for all $\omega$ and $t$, we get
$$
A_t^p \leq p 2^{p-2}\left[\int_{0}^{t}\int_ZA_{s-}^{p-1}
h_s(z)p(ds,dz)+\int_{]0,t]}\int_{h}h_s(z)^{p}p(ds,dz)\right].
$$
and
$$
A_t^p \ge\int_{]0,t]}\int_{Z}h_s(z)^{p}p(ds,dz).
$$ 
Then since $A_t$ is an increasing process,  we have
$$
\mathbf{E}\int_{]0,\tau]}\int_{Z}h_s(z)^{p}p(ds,dz)\le
\mathbf{E}A_{\tau}^p \le  p 2^{p-2}\mathbf{E}\left[A_{\tau}^{p-1}L_{\tau}+\int_{]0,\tau]}\int_{Z}h_s(z)^{p}p(ds,dz)\right].
$$
It is easy to see that  
\begin{equation*}
\mathbf{E}L_{\tau }^{p}=p\mathbf{E}\int_{]0,\tau] }L_{s}^{p-1}dL_{s}=p%
\mathbf{E}\int_{]0,\tau] }L_{s}^{p-1}dA_{s}\leq p\mathbf{E}[L_{\tau
}^{p-1}A_{\tau }].
\end{equation*}%
Applying  Young's inequality, for all $\varepsilon>0$,  $\bP$-a.s.,
\begin{equation*}
A_{\tau }^{p-1}L_{\tau }\leq \varepsilon A_{\tau }^{p}+\frac{(p-1)^{p-1}}{\varepsilon^{p-1}p^p}L_{\tau }^{p}\quad \textrm{and}\quad L_{\tau }^{p-1}A_{\tau }\leq \varepsilon L_{\tau
}^{p}+\frac{(p-1)^{p-1}}{\varepsilon^{p-1}p^p}A_{\tau }^{p}.
\end{equation*}
Combining the above estimates, for any $\varepsilon_1\in (0,\frac{1}{p})$, we have
$$
\left(\frac{\varepsilon_1^{p-1}p^p(1-p\varepsilon_1) }{p(p-1)^{p-1}}\mathbf{E}L_{\tau }^{p}\right)\vee \mathbf{E}\int_{]0,\tau]}\int_{Z}h_s(z)^{p}p(ds,dz) \le \mathbf{E}A_{\tau }^{p}.
$$
and for any $\varepsilon_2 \in (0,\frac{1}{p2^{p-2}})$
$$
\mathbf{E} A_{\tau}^p\le \frac{p 2^{p-2}}{(1-p2^{p-2}\varepsilon_2 )}\mathbf{E}\left[\int_{]0,\tau]}\int_{Z}h_s(z)^{p}p(ds,dz)+\frac{(p-1)^{p-1}}{\varepsilon_2^{p-1}p^p}L^p_{\tau}\right],
$$
which completes the proof.
\end{proof}

\subsection{Optional projection}
The following lemma concerning the
optional projection plays an integral role in Section \ref{sec:uncorrelated}
and the proof of Theorem \ref{thm:repwcorrec}.

\begin{lemma}
\label{lem:optionalprojection}(cf. Theorem 1 in \cite{Me76}) Let $\mathcal{X}
$ be a Polish space and $D\left( [0,T];\mathcal{X}\right) $ be the space of $%
\mathcal{X}$-valued c\`{a}dl\`{a}g trajectories with the Skorokhod $\mathcal{%
J}_{1}$-topology. If $\mathfrak{A}$ is a random variable taking values in $%
D\left( [0,T];\mathcal{X}\right) $, then there exists a family of $\mathcal{B%
}([0,T])\times \cF$-measurable non-negative measures $E^{t}(dU),$ $%
(\omega,t)\in \Omega\times[0,T],$ on $D\left( [0,T];\mathcal{X}\right) $ and a random-variable $%
\zeta $ satisfying $\mathbf{P}\left( \zeta <T\right) =0$ such that $%
E^{t}(D\left( [0,T];\mathcal{X}\right) )=1$ for $t<\zeta $ and $%
E^{t}(D([0,T];\mathcal{X}) )=0$ for $t\geq \zeta .$ In addition, $E^{t}$ is c%
\`{a}dl\`{a}g in the topology of weak convergence, $E^{t}=E^{t+}$ for all $%
t\in \lbrack 0,T]$, and for each continuous and bounded functional $F$ on $%
D\left( [0,T];\mathcal{X}\right) ,$ the process $E^{t}\left( F\right) $ is
the c\`{a}dl\`{a}g version of $\mathbf{E}[ F\left( \mathfrak{A}\right) |%
\cF_{t}] $. If $G:\Omega \times [0,T]\times
[0,T]\times D\left( [0,T]; \mathcal{X}\right)\rightarrow \mathbf{R}^{d_2}$
is bounded and $\mathcal{O\times B}\left( [0,T]\right) \times \mathcal{B}%
\left( D\left( [0,T]; \mathcal{X}\right) \right) $-measurable, then 
\begin{equation*}
\int_{D( [0,T];\mathcal{X}) } G_t(\omega ,t,U)E^{t}(dU)=E^{t}( G_t)
\end{equation*}%
is the optional projection of $G_{t}(\mathfrak{A})=G_t(\omega ,t,\mathfrak{A}%
)$. Furthermore, if $G=G_t(\omega ,t,U)$ is bounded and $\mathcal{P}\times 
\mathcal{B}( [0,T]) \times \mathcal{B}( D( [0,T]; \mathcal{X})) $%
-measurable, then $E^{t-}(G_{t}) $ is the predictable projection
of $G_{t}(\mathfrak{A}) =G_t(\omega ,t,\mathfrak{A}).$
\end{lemma}

\begin{proof}
We follow the proof of Theorem 1 in \cite{Me76}. Since $D([0,T];\mathcal{X})$
is a Polish space, for each $t\in \lbrack 0,T]$, there is family of
probability measures $\tilde{E}_{\omega }^{t}(dw)$, $\omega \in \Omega $%
, on $D([0,T];\mathcal{X})$ such that for each $A\in \mathcal{B}(D([0,T];%
\mathcal{X})),$ $\tilde{E}^{t}(A)$ is $\cF_{t}$-measurable and $%
\mathbf{P}$-a.s. , 
\begin{equation*}
\mathbf{P}\left( \mathfrak{A}\in A|\cF_{t}\right) =\tilde{E}%
^{t}\left( A\right) .
\end{equation*}%
For each $\omega \in \Omega $, let $I\left( \omega \right) $ be the set of
all $t\in (0,T]$ such that for each bounded continuous function $F$ on $%
D(([0,T];\mathcal{X})$, the function 
\begin{equation*}
r\mapsto \tilde{E}_{\omega }^{r}(F)=\int_{D\left( [0,T];\mathcal{X}\right)
}F(w)\tilde{E}^{r}(dw)
\end{equation*}%
has a right-hand limit on $[0,s)\cap \mathbf{Q}$ and a left-hand limit on $%
(0,s]\cap \mathbf{Q}$ for every rational $s\in \lbrack 0,T]\cap \mathbf{Q}$.
Let $\zeta \left( \omega \right) =\sup \left( t:t\in I(\omega )\right)
\wedge T.$ It is easy to see that $\mathbf{P}\left( \xi <T\right) =0.$ We
set $\tilde{E}_{\omega }^{t}=0$ if $\xi (\omega )<t\leq T$. The function $%
\tilde{E}_{\omega }^{t}$ has left-hand and right-hand limits for all $t\in 
\mathbf{Q}\cap \left[ 0,T\right] $. We define $E_{\omega }^{t}=\tilde{E}%
_{\omega }^{t+}$ for each $t\in \lbrack 0,T)$ (the limit is taken along the
rationals), and $E_{\omega }^{T}$ is the left-hand limit at $T$ along the
rationals$.$ The statement follows by repeating the proof of Theorem 1 in 
\cite{Me76} in an obvious way.
\end{proof}

\subsection{Estimates of H\"{o}lder continuous functions}
In the coming lemmas, we establish some properties of weighted H\"{o}lder
spaces that are used Section \ref{s:Large Jumps} and the proof of Theorem %
\ref{thm:Existwocorrec}.

\begin{lemma}
\label{lem:characteriziationofweighted} Let $\beta\in (0,1]$ and $%
\theta_1,\theta_2\in\mathbf{R}$ with $\theta _{1}-\theta _{2}\le \beta.$

\begin{tightenumerate}
\item There is a constant $c_{1}=c_{1}\left( \theta _{2},\beta \right) $
such that for all $\phi:\mathbf{R}^{d_1}\rightarrow \mathbf{R}$ with $%
|r_{1}^{-\theta _{1}}\phi|_{0}+[r_{1}^{-\theta _{2}}\phi]_{\beta
}=:N_{1}<\infty, $ 
\begin{equation*}
|\phi(x)-\phi(y)|\leq c_{1}N_{1}(r_{1}(x)^{\theta _{2}}\vee r_{1}(y)^{\theta
_{2}})|x-y|^{\beta },
\end{equation*}
for all $x,y\in \mathbf{R}^{d_1}$.

\item Conversely, if $\phi:\mathbf{R}^{d_1}\rightarrow \mathbf{R}$ satisfies 
$|r_{1}^{-\theta _{1}}\phi|_{0}<\infty $ and there is a constant $N_{2}$
such that for all $x,y\in \mathbf{R}^{d_1}$, 
\begin{equation*}
|\phi(x)-\phi(y)|\leq N_{2}( r_{1}(x)^{\theta _{2}}\vee r_{1}(y)^{\theta
_{2}}) |x-y|^{\beta },
\end{equation*}%
then 
\begin{equation*}
[r_{1}^{-\theta _{2}}\phi]_{\beta }\leq c_{1}|r_{1}^{-\theta
_{1}}\phi|_{0}+N_{2}.
\end{equation*}
\end{tightenumerate}
\end{lemma}

\begin{proof}
(1) For all $x,y$ with $r_{1}\left( x\right) ^{\theta _{2}}\geq r_{1}\left(
y\right) ^{\theta _{2}}$, we have 
\begin{align*}
|\phi(x)-\phi(y)|&\leq r_{1}(x)^{\theta _{2}}[r_{1}^{-\theta
_{2}}\phi]_{\beta }|x-y|^{\beta }+r_{1}(y) ^{\theta _{1}-\theta
_{2}}|r_{1}^{-\theta _{1}}\phi|_{0} |r_{1}^{\theta _{2}}(x)-r_{1}(y)
^{\theta _{2}}| \\
&\le ([r_{1}^{-\theta _{2}}\phi]_{\beta }+c_{1}|r_{1}^{-\theta
}\phi|_{0})r_{1}(x)^{\theta _{2}}|x-y|^{\beta },
\end{align*}
where  $c_{1} :=1+\sup_{t\in(0,1) }\frac{1-t^{\theta _{2}}}{( 1-t)^{\beta }}$
if $\theta _{2}\geq 0$ and $c_{1} :=1+\sup_{t\in (1,\infty )}\frac{%
(t^{\theta _{2}}-1)t^{\beta }}{( t-1) ^{\beta }}$ if $\theta _{2}<0,$ which
proves the first claim. (2) For all $x$ and $y$ with $r_{1}(x)^{\theta
_{2}}>r_{1}(y)^{\theta _{2}}$, we have 
\begin{equation*}
|r_{1}(x)^{-\theta _{2}}\phi(x)-r_{1}(y)^{-\theta _{2}}\phi(y)| 
\end{equation*}
\begin{equation*}
\leq r_{1}(x)^{-\theta _{2}}|\phi(x)-\phi(y)|+r_{1}(y)^{\theta _{1}-\theta
_{2}}|r_{1}^{-\theta _{1}}(y)\phi(y)| |r_{1}(y)^{\theta
_{2}}r_{1}(x)^{-\theta _{2}}-1| 
\end{equation*}
\begin{equation*}
\leq (c_{1}| r^{-\theta _{1}}\phi | _{0}+N_{2})|x-y|^{\beta }, 
\end{equation*}
which proves the second claim.
\end{proof}

\begin{lemma}
\label{lem:compositionofweightedholder}Let $\beta,\mu\in (0,1]$ and $%
\theta_1,\theta_2,\theta_3,\theta_4\in\mathbf{R}$ with $\theta _{1}-\theta
_{2}\le \beta$, $\theta _{3}-\theta _{4}\le \mu$, and $\theta _{3}\geq 0.$
If $\phi:\mathbf{R}^{d_1}\rightarrow \mathbf{R}$ and $H:\mathbf{R}%
^{d_1}\rightarrow \mathbf{R}^{d_1}$ are such that 
\begin{equation*}
|r_{1}^{-\theta _{1}}\phi|_{0}+[r_{1}^{-\theta _{2}}\phi]_{\beta
}=:N_1<\infty \quad \text{and} \quad |r_{1}^{-\theta
_{3}}H|_{0}+[r_{1}^{-\theta _{4}}H]_{\mu }=:N_2<\infty, 
\end{equation*}
then 
\begin{equation*}
|\phi \circ H\cdot r_{1}^{-\theta _{1}\theta _{3}}|_{0} \leq |r_{1}^{-\theta
_{1}}\phi|_{0} (1+|r_{1}^{-\theta _{3}}H|_{0}) \le N_{1}\left(
1+N_{2}\right) ^{\theta _{1}} 
\end{equation*}
and there is a constant $N=N(\beta ,\mu ,\theta _{1},\theta _{2})$ such that 
\begin{equation*}
[\phi\circ H\cdot r_{1}^{-\theta _{2}\theta _{3}-\beta \theta _{4}}]_{\beta
\mu }\leq NN_{1}( 1+N_{2})^{\theta _{2}+\beta }.
\end{equation*}
\end{lemma}

\begin{proof}
For each $x$, we have 
\begin{equation*}
r_{1}(H(x))\leq (1+|r_{1}^{-\theta _{3}}H|_{0})r_{1}(x)^{\theta _{3}}\leq (
1+N_{2})r_{1}(x)^{\theta _{3}},
\end{equation*}%
and hence 
\begin{equation*}
|\phi\circ H\cdot r_{1}^{-\theta _{1}\theta _{3}}|_{0}\leq |r_{1}^{-\theta
_{1}}\phi|_{0}|r_{1}^{\theta _{1}}\circ H\cdot r_{1}^{-\theta _{1}\theta
_{3}}|_{0}\leq N_{1}(1+N_{2})^{\theta _{1}}.
\end{equation*}%
Using Lemma \ref{lem:characteriziationofweighted}, for all $x$ and $y$, we
get%
\begin{align*}
|\phi(H(x))-\phi(H(y))| & \leq N N_{1}( r_{1}(H(x))\vee r_{1}(H(y)))^{\theta
_{2}}|H(x)-H(y)|^{\beta } \\
&\leq NN_{1}(1+N_{2})^{\theta _{2}}(r_{1}(x)\vee r_{1}(y))^{\theta
_{2}\theta _{3}}N_{2}^{\beta }( r_{1}(x)\vee r_{1}(y)) ^{\beta \theta
_{4}}|x-y|^{\beta \mu } \\
&\leq NN_{1}(1+N_{2})^{\theta _{2}+\beta }(r_{1}(x)\vee r_{1}(y))^{\theta
_{2}\theta _{3}+\beta \theta _{4}}| x-y| ^{\beta \mu },
\end{align*}%
for some constant $N=N(\beta,\mu,\theta _{1},\theta _{2})$. Noting that 
\begin{equation*}
\theta_1\theta_3-\theta_2\theta_3-\beta\theta_4=(\theta_1-\theta_2)\theta_3-%
\beta\theta_4\le \beta(\theta_3-\theta_4)\le \beta\mu, 
\end{equation*}
we apply Lemma \ref{lem:characteriziationofweighted} to complete the proof.
\end{proof}

\begin{remark}
\label{r22}Let $\beta \in (0,1]$ and $\theta_1,\theta_2\in\mathbf{R}.$ Then there
is a constant $N=N(\beta,\theta_1,\theta_2)$ such that for all $\phi:\mathbf{%
R}^{d_1}\rightarrow \mathbf{R}$ with $|r_{1}^{-\theta
_{1}}\phi|_{0}+[r_{1}^{-\theta _{2}}\phi]_{\beta }=:N_{1}<\infty, $ we have $%
|r^{-\theta }\phi|_{\beta }\leq NN_{1}$, where $\theta =\max \left\{ \theta
_{1},\theta _{2}\right\} .$ In particular, if in Lemma \ref%
{lem:compositionofweightedholder}, $\theta _{1}=\theta _{2}$ and $\theta
_{4}\geq 0$, then 
\begin{equation*}
|\phi\circ H\cdot r^{-\theta _{1}\theta _{3}-\beta \theta _{4}}|_{\beta \mu
}\leq NN_{1}( 1+N_{2}) ^{\theta _{1}+\beta }.
\end{equation*}
\end{remark}

\begin{proof}
If $\theta _{2}\geq \theta _{1}$, then the claim is obvious and if $\theta
_{1}>\theta _{2}$, for all $x$ and $y$, we have 
\begin{align*}
|r_{1}(x)^{-\theta _{1}}\phi(x)-r_{1}(y)^{-\theta _{1}}\phi(y)|&\le r_{1}(x)
^{\theta _{2}-\theta _{1}}| r_{1}(x)^{-\theta _{2}}\phi(x)-r_{1}(y)^{-\theta
_{2}}\phi(y)| \\
&\quad +\left|\frac{r(y)^{\theta _{1}-\theta _{2}}}{r(x) ^{\theta
_{1}-\theta _{2}}}-1\right||r_{1}^{-\theta _{1}}\phi |_0 \le
N_1(1+c_1)|x-y|^{\beta},
\end{align*}%
where $c_{1}:=\sup_{t\in (0,1)}\frac{1-t^{\theta _{1}-\theta _{2}}}{\left(
1-t\right) ^{\beta }}.$
\end{proof}

\begin{lemma}
\label{lem:equivnorm}For each $\theta \geq 0$ and $\beta >1$ , there are
constants $N_{1}=N_{1}(d_1,\theta ,\beta )$ and $N_{2}(d_1,\theta ,\beta )$
such that for all $\phi:\mathbf{R}^{d_1}\rightarrow \mathbf{R}$, 
\begin{equation}  \label{ineq:norm-equivalence}
N_{1}|r_{1}^{-\theta }\phi|_{\beta }\leq \sum_{| \gamma | \leq \left[ \beta %
\right] ^{-}}|r_{1}^{-\theta }\partial ^{\gamma }\phi|_{0}+\sum_{|\gamma
|=[\beta ]^{-}}|r_{1}^{-\theta }\partial ^{\gamma }\phi|_{\{\beta
\}^{+}}\leq N_{2}|r_{1}^{-\theta }\phi|_{\beta }.
\end{equation}
\end{lemma}

\begin{proof}
For each multi-index $\gamma $ with $|\gamma |\leq \lbrack \beta ]^{-}$ and $%
x$, we have 
\begin{equation*}
\partial ^{\gamma }(r_{1}^{-\theta }\phi)(x)=\sum_{\underset{|\gamma
_{1}|\geq 1}{\gamma _{1}+\gamma _{2}+=\gamma }}r_{1}(x)^{\theta }\partial
^{\gamma _{1}}(r_{1}^{-\theta })(x)r_{1}(x)^{-\theta }\partial ^{\gamma
_{2}}\phi(x)+r_{1}(x)^{-\theta }\partial ^{\gamma }\phi(x).
\end{equation*}%
It is easy to show by induction that for all multi-indices $\gamma $, $%
|r_{1}^{\theta }\partial ^{\gamma }(r_{1}^{-\theta })|_{1}<\infty .$
Moreover, for all multi-indices $\gamma $ with $|\gamma |<[\beta ]^{-}$, 
\begin{equation*}
|r_{1}^{-\theta }\partial ^{\gamma }\phi|_{1}\leq |\nabla ( r_{1}^{-\theta
}\partial ^{\gamma }\phi)|\leq |r_{1}^{-\theta }\nabla ( r_{1}^{-\theta
})|_{0}|r_{1}^{-\theta }\partial ^{\gamma }\nabla \phi|_{0}.
\end{equation*}%
Thus, for each multi-index $\gamma $ with $|\gamma |\leq \lbrack \beta ]^{-}$%
, 
\begin{equation*}
|\partial ^{\gamma }(r_{1}^{-\theta }\phi)|_{0}\leq \sum_{\underset{|\gamma
_{1}|\geq 1}{\gamma _{1}+\gamma _{2}+=\gamma }}|r_{1}^{\theta }\partial
^{\gamma _{1}}(r_{1}^{-\theta })|_{0}|r_{1}^{-\theta }\partial ^{\gamma
_{2}}\phi|_{0}+|r_{1}^{-\theta }\partial ^{\gamma }\phi|_{0}
\end{equation*}%
and for each multi-index $\gamma $ with $|\gamma |=[\beta ]^{-}$, 
\begin{equation*}
|\partial ^{\gamma }(r_{1}^{-\theta }\phi)|_{\{\beta \}^{+}}\leq \sum_{%
\underset{|\gamma _{1}|\geq 1}{\gamma _{1}+\gamma _{2}+=\gamma }%
}|r_{1}^{\theta }\partial ^{\gamma _{1}}(r_{1}^{-\theta
})|_{1}|r_{1}^{-\theta }\nabla ( r_{1}^{-\theta })|_{0}|r_{1}^{-\theta
}\partial ^{\gamma _{2}}\nabla \phi|_{0}+|r_{1}^{-\theta }\partial ^{\gamma
}\phi|_{0}.
\end{equation*}%
This proves the leftmost inequality in \eqref{ineq:norm-equivalence}. For
all $i\in \{1,\ldots ,d\}$ and $x$, 
\begin{equation*}
r_{1}^{-\theta }\partial _{i}\phi(x)=\partial _{i}(r_{1}^{-\theta
}\phi)(x)-r_{1}(x)^{-\theta }\phi(x)r_{1}(x)^{\theta }\partial
_{i}(r_{1}^{-\theta })(x).
\end{equation*}%
It follows by induction that for all multi-indices $\gamma $ with $|\gamma
|\leq \lbrack \beta ]^{-}$ and $x$, $r_{1}^{-\theta }\partial ^{\gamma
}\phi(x)$ is a sum of $\partial ^{\gamma }(r_{1}^{-\theta }\phi)(x),$ a
finite sum of terms, each of which is a product of one term of the form $%
\partial ^{\tilde{\gamma}}(r_{1}^{-\theta }\phi)(x),$ $|\tilde{\gamma}%
|<|\gamma |$, and a finite number of terms of the form $\partial ^{\gamma
_{1}}(r_{1}^{\theta })\partial ^{\gamma _{2}}(r_{1}^{-\theta }),$ $|\gamma
_{1}|,|\gamma _{2}|\leq |\gamma |$. Since for all multi-indices $\gamma _{1}$
and $\gamma _{2}$, we have $|\partial ^{\gamma _{1}}(r_{1}^{\theta
})\partial ^{\gamma _{2}}(r_{1}^{-\theta })|_{1}<\infty ,$ the rightmost
inequality in \eqref{ineq:norm-equivalence} follows.
\end{proof}

\begin{corollary}
\label{co8}For each $\theta \geq 0$ and $\beta >1$ , there are constants $%
N_{1}=N_{1}(d_1,\theta ,\beta )$ and $N_{2}(d_1,\theta ,\beta )$ such that
for all $\phi:\mathbf{R}^{d_1}\rightarrow \mathbf{R}$, 
\begin{equation}
N_{1}|r_{1}^{-\theta }\phi|_{\beta }\leq |r_{1}^{-\theta
}\phi|_{0}+\sum_{|\gamma |=[\beta ]^{-}}|r_{1}^{-\theta }\partial ^{\gamma
}\phi|_{\{\beta \}^{+}}\leq N_{2}|r_{1}^{-\theta }\phi|_{\beta }.  \label{f1}
\end{equation}
\end{corollary}

\begin{proof}
It is well known that for an arbitrary unit ball $B\subset \mathbf{R}^{d_1}$
and any $1\leq k<[\beta]^{-}$, there is a constant $N$ such that for any $%
\varepsilon >0,$%
\begin{equation*}
\sup_{x\in B,|\gamma |=k}|\partial ^{\gamma }\phi|\leq N( \varepsilon
\sup_{x\in B,|\gamma | =[\beta ] ^{-}}|\partial ^{\gamma
}\phi(x)|+\varepsilon ^{-k}\sup_{x\in B}|\phi(x)|).
\end{equation*}%
Let $U_{0}=\{ x\in \mathbf{R}^{d_1}:|x|\leq 1\}$ and $U_{j}=\{ x\in \mathbf{R%
}^{d_1}:2^{j-1}\leq | x| \leq 2^{j}\} ,j\geq 1.$ For each $j$, we have 
\begin{align*}
\sup_{x\in U_{j},|\gamma| =k}| \partial ^{\gamma }\phi(x)|
&=\sup_{B\subseteq U_{j}}\sup_{x\in B,| \gamma | =k}| \partial ^{\gamma
}\phi(x)| \leq N(\varepsilon \sup_{B\subseteq U_{j}}\sup_{x\in B,| \gamma |
=[ \beta ] ^{-}}| \partial ^{\gamma }\phi(x)| +\varepsilon
^{-k}\sup_{B\subseteq U_{j}}\sup_{x\in B}| \phi(x)| ) \\
&\le N(\varepsilon \sup_{x\in U_{j},| \gamma | =[ \beta ] ^{-}}| \partial
^{\gamma }\phi(x)| +\varepsilon ^{-k}\sup_{x\in U_{j}}| \phi(x)| ).
\end{align*}%
Since for every $j$, 
\begin{equation*}
2^{-\theta /2}2^{-j\theta }\sup_{x\in U_{j},| \gamma | =k}| \partial
^{\gamma }\phi(x)| \leq \sup_{x\in U_{j},| \gamma | =k}| r^{-\theta
}\partial ^{\gamma }\phi(x)| \leq 2^{^{\theta }}2^{-(j-1)\theta }\sup_{x\in
U_{j},| \gamma | =k}| \partial ^{\gamma }\phi(x)| ,
\end{equation*}%
we see that%
\begin{align*}
2^{-\theta /2}\sup_{j}2^{-j\theta }\sup_{x\in U_{j},| \gamma | =k}| \partial
^{\gamma }\phi(x)| &\leq \sup_{j}\sup_{x\in U_{j},| \gamma | =k}| r^{-\theta
}\partial ^{\gamma }\phi(x)| =| r^{-\theta }\partial ^{\gamma }\phi| _{0} \\
&\leq 2^{\theta }\sup_{j}2^{-j\theta }\sup_{x\in U_{j},| \gamma | =k}|
\partial ^{\gamma }\phi(x)|,
\end{align*}%
and the statement follows.
\end{proof}

\begin{remark}
\label{rr3}If $\phi:\mathbf{R}^{d_1}\rightarrow\mathbf{R}$ is such that $|
r^{-\theta _{1}}\phi| _{0}+| r^{-\theta _{2}}\nabla \phi| _{0}<\infty $ for $%
\theta_1,\theta_2\in\mathbf{R}$ with $\theta _{1}-\theta _{2}\le 1$, then 
\begin{equation*}
[ r^{-\theta _{2}}\phi]_{1}\leq N(| r^{-\theta _{1}}\phi| _{0}+| r^{-\theta
_{2}}\nabla \phi| _{0})
\end{equation*}
\end{remark}

\begin{proof}
Indeed, for each $x$ and $y$, we have 
\begin{align*}
| \phi(x)-\phi(y) | &\leq | r^{-\theta _{2}}\nabla \phi|
_{0}\int_{0}^{1}r^{\theta _{2}}( x+s( y-x) ) ds| y-x| \le | r^{-\theta
_{2}}\nabla \phi| _{0}( r( y) ^{\theta _{2}}\vee r( x) ^{\theta _{2}}) |
y-x|,
\end{align*}%
and hence the claim follows from Lemma \ref{lem:characteriziationofweighted}.
\end{proof}

\begin{lemma}
\label{lem:GeneralCompWeightedlemma} Let $n\in\mathbf{N}$, $\beta ,\mu \in
(0,1]$, $\theta _{3},\theta _{4}\geq 0$ be such that $\theta _{3}-\theta
_{4}\le 1$. There is a constant $N=N(d_1,\theta _{1},\theta _{3},\theta
_{4},n,\beta ) $ such that for all $\phi:\mathbf{R}^{d_1}\rightarrow\mathbf{R%
}$ with $r_{1}^{-\theta _{1}}\phi\in \mathcal{C}^{n+\beta }(\mathbf{R}^{d_1},%
\mathbf{R}^{d_1})$ and $H:\mathbf{R}^{d_1}\rightarrow\mathbf{R}^{d_1}$ with 
\begin{equation*}
|r_{1}^{-\theta _{3}}H|_{0}+|r_{1}^{-\theta _{4}}\nabla H|_{n-1+\mu
}=:N_2<\infty ,
\end{equation*}%
we have 
\begin{equation*}
|\phi\circ H\cdot r^{-\theta _{1}\theta _{3}}|_{0} \leq | r_{1}^{-\theta
_{1}}\phi| _{0} (1+|r_{1}^{-\theta _{3}}H|_{0})^{\theta _{1}} 
\end{equation*}
and 
\begin{equation*}
|r_{1}^{-\theta _{1}\theta _{3}-\theta _{4}(n+\mu\wedge \beta)}\nabla (\phi\circ
H)|_{n-1+\mu \wedge \beta } \leq N|r_{1}^{-\theta _{1}}\phi|_{n+\beta }(
1+N_{2}) ^{\theta _{1}+\mu\wedge \beta+n}. 
\end{equation*}
\end{lemma}

\begin{proof}
It follows immediately from Lemma \ref{lem:compositionofweightedholder} and
Remark \ref{rr3} that 
\begin{equation*}
|\phi \circ H\cdot r^{-\theta _{1}\theta _{3}}|_{0}\leq |r_{1}^{-\theta
_{1}}\phi |_{0}(1+|r_{1}^{-\theta _{3}}H|_{0})^{\theta _{1}}.
\end{equation*}%
Using induction, we get that for each $x$ and $|\gamma |=n$, 
\begin{equation*}
\partial ^{\gamma }(\phi (H(x)))=\mathcal{I}_{1}^{\gamma }(x)+\mathcal{I}%
_{2}^{\gamma }(x)+I_{3}^{\gamma }(x),
\end{equation*}%
where 
\begin{equation*}
\mathcal{I}_{1}^{\gamma }(x)=\sum_{i=1}^{d_{1}}\partial _{i}\phi
(H(x))\partial ^{\gamma }H^{i}(x)
\end{equation*}%
$\mathcal{I}_{2}^{\gamma }(x)$ is a finite sum of terms of the form 
\begin{equation*}
\partial _{i_{1}}\cdots \partial _{i_{|\gamma |}}\phi (H(x))\partial ^{%
\tilde{\gamma}_{1}}H^{i_{1}}\cdots \partial ^{\tilde{\gamma}_{|\gamma
|}}H^{i_{|\gamma |}}
\end{equation*}%
with $i_{1},\ldots ,i_{|\gamma |}\in \{1,2,\ldots ,d\}$, $|\tilde{\gamma}%
_{1}|=\cdots =|\tilde{\gamma}_{|\gamma |}|=1$, and $\sum_{k=1}^{|\gamma |}%
\tilde{\gamma}_{k}=\gamma $, if $n\geq 2$ and zero otherwise, and where $%
\mathcal{I}_{3}^{\gamma }(x)$ is a finite sum of terms of the form 
\begin{equation*}
\partial _{i_{1}}\cdots \partial _{i_{k}}\phi (H(x))\partial ^{\tilde{\gamma}%
_{1}}H^{i_{1}}(x)\cdots \partial ^{\tilde{\gamma}_{k}}H^{i_{k}}(x)
\end{equation*}%
with $2\leq k<n$, $i_{1},i_{2},\ldots ,i_{k}\in \{1,\ldots ,d\}$, and $%
\sum_{j=1}^{k}\tilde{\gamma}_{j}=\gamma $, $1\leq |\tilde{\gamma}%
_{j}|<|\gamma |,$ if $n\geq 3$, and zero otherwise. Thus, owing to Lemmas %
\ref{lem:compositionofweightedholder} and \ref{lem:equivnorm}, for any
multi-index $\gamma $ with $|\gamma |=n$, we have 
\begin{equation*}
|r_{1}^{-\theta _{3}\theta _{1}-\theta _{4}}\mathcal{I}_{1}^{\gamma
}|_{0}\leq N|r_{1}^{-\theta _{1}}\nabla \phi |_{0}(1+|r_{1}^{-\theta
_{3}}H|_{0})^{\theta _{1}}|r_{1}^{-\theta _{4}}\partial ^{\gamma }H|_{0},
\end{equation*}%
\begin{equation*}
|r_{1}^{-\theta _{3}\theta _{1}-n\theta _{4}}\mathcal{I}_{2}^{\gamma
}|_{0}\leq N|r_{1}^{-\theta _{1}}\partial ^{\gamma }\phi
|_{0}(1+|r_{1}^{-\theta _{3}}H|_{0})^{\theta _{1}}|r_{1}^{-\theta
_{4}}\nabla H|_{0}^{n},
\end{equation*}%
and 
\begin{equation*}
|r_{1}^{-\theta _{3}\theta _{1}-(n-1)\theta _{4}}\mathcal{I}_{3}^{\gamma
}|_{0}\leq N|r_{1}^{-\theta _{1}}\phi |_{n-1}(1+|r_{1}^{-\theta
_{3}}H|_{0}+|r_{1}^{-\theta _{4}}\nabla H|)^{\theta _{1}+n-1},
\end{equation*}%
and hence 
\begin{equation*}
|r^{-\theta _{1}\theta _{3}-n\theta _{4}}\partial ^{\gamma }(\phi \circ
H)|_{0}\leq N|r_{1}^{-\theta _{1}}\phi |_{n}(1+|r_{1}^{-\theta
_{3}}H|_{0}+|r_{1}^{-\theta _{4}}\nabla H|)^{\theta _{1}+n}.
\end{equation*}%
Once again appealing to Lemmas \ref{lem:compositionofweightedholder}
and \ref{lem:equivnorm}, for all multi-indices $\gamma $ with $|\gamma |=n$,
we get 
\begin{equation*}
|r_{1}^{-\theta _{1}\theta _{3}-(1+\mu\wedge \beta)\theta _{4}}\mathcal{I}_{1}^{\gamma }|_{\mu
\wedge \beta }\leq N|r_{1}^{-\theta _{1}}\phi |_{1+\mu \wedge \beta }\left(
1+N_{2}\right) ^{\theta _{1}+\mu\wedge \beta+1},
\end{equation*}%
\begin{equation*}
|r_{1}^{-\theta _{1}\theta _{3}-(n+\mu\wedge \beta)\theta _{4}}\mathcal{I}_{2}^{\gamma
}|_{\mu \wedge \beta }+|r_{1}^{-\theta _{1}\theta _{3}-(n-1+\mu\wedge \beta)\theta _{4}}%
\mathcal{I}_{3}^{\gamma }|_{\mu \wedge \beta }\leq N|r_{1}^{-\theta
_{1}}\phi |_{n+\mu \wedge \beta }\left( 1+N_{2}\right) ^{\theta _{1}+n+\mu\wedge \beta}.
\end{equation*}%
Then applying Lemmas \ref{lem:compositionofweightedholder} and \ref%
{lem:equivnorm}, we complete the proof. 
\end{proof}

We shall now provide some useful estimates of composite functions of
diffeomorphisms.

\begin{lemma}
\label{lem:compositediffeo}Let $H:\mathbf{R}^{d_1}\rightarrow \mathbf{R}%
^{d_1} $ be continuously differentiable and assume that for all $x\in 
\mathbf{R}^{d_1} $, 
\begin{equation*}
| H(x)| \leq L_{0}+L_{1}|x|\quad \text{and}\quad |\nabla H(x)|\leq L_{2}.
\end{equation*}%
Assume that for all $x\in \mathbf{R}^{d_1}$, $\kappa (x)=(I_{d_1}+\nabla
H(x))^{-1}$ exists and $|\kappa (x)|\leq N_{\kappa }$.

\begin{tightenumerate}
\item Then  the mapping $\tilde{H}(x):=x+H(x)$ is a diffeomorphism with $\tilde{H}^{-1}(x)=x-H( 
\tilde{H}^{-1}(x))=:x+F(x)$ and  for all $x\in\mathbf{R}^{d_1}$, 
\begin{equation*}
| F(x)| \le L_0+L_1L_0N_{\kappa}+L_1N_{\kappa}| x|, \quad |\nabla F(x)|\le
N_{\kappa}L_{2}, \quad \quad |\left( I_{d_1}+\nabla F(x)\right)
^{-1}| \leq 1+L_{2}.
\end{equation*}
For all $p\in\bR$,  there is a constant $N=N(L_0,L_1,N_{\kappa},p)$ such that for all $x\in\mathbf{R}^{d_1}$, 
$$
\frac{r_1^p(\tilde{H}(x))}{r_1^p(x)}+\frac{r_1^p(\tilde{H}^{-1}(x))}{r_1^p(x)}\le N, \quad 
r_1^{-1}(x)|H^{i}(x)+F^{k;i}(x)|\le N[H]_1|r_1^{-1}H|_0.
$$
Moreover,  there is a constant $N=N(L_0,L_1,N_{\kappa},p)$ such that
$$
\left|\frac{r_1^p(\tilde{H})}{r_1^p}-1+\mathbf{1}_{(1,2]}(\alpha)pH^{i}r_1^{-2}x^i\right|_\alpha+\left|\frac{r_1^p(\tilde{H}^{-1})}{r_1^p}-1-\mathbf{1}_{(1,2]}(\alpha)pF^{i}r_1^{-2}x^i\right|_{\alpha}$$
$$\le N(|r_1^{-1}H|_0^{[\alpha]^-+1}+[H]^{[\alpha]^-+1}_1).
$$

\item If  for some $\beta>1$, $|\nabla H|_{\beta-1}\le L_3$, then there is a
constant $N=N(d_1,\beta,N_{\kappa},L_{3})$ such that%
\begin{equation}  \label{ineq:estimateofF}
|\nabla F| _{\beta -1}\leq N|\nabla H| _{\beta -1}.
\end{equation}

\item If for some $\beta\ge 1$, $|\nabla H|_{\beta-1}\le L_3$, then for each $%
\theta\ge 0$, there is a constant $N=N(d_1,$ $\beta,N_{\kappa},$ $L_1,L_3,\theta)$
such that 
\begin{equation*}
\left|\frac{r^{\theta}_1\circ \tilde{H}^{-1}}{r_1^{\theta}}%
-1\right|_{\beta}\le N[ |r_1^{-1}H|_0+|\nabla H|_{\beta-1}].
\end{equation*}

\item If $|H|_{0 }\leq L_{4}$,  and for some $\beta>0$, $|\nabla H|_{\beta\vee 1-1 }\leq L_{5}$  and $\phi:\mathbf{R}%
^{d_1}\rightarrow\mathbf{R}$ is such that for some $\mu\in (0,1]$ and $%
\theta\ge 0$, $r_1^{-\theta}\phi\in
C^{\beta+\mu}( \mathbf{R}^{d_1};\mathbf{R})$, then there is a constant $N=N(d_1,\beta,\mu,N_{%
\kappa},L_{4},L_5,\theta)$ such that 
\begin{align*}
|r_1^{-\theta}(\phi\circ \tilde{H}^{-1}-\phi)| _{\beta }&\leq N|r_1^{-\theta}\phi|_{\beta}(|H|_0+|\nabla H|_{\beta \vee 1-1})\\
&\quad + N\mathbf{1}_{(0,1]}(\{\beta\}^++\mu)\sum_{|\gamma|=[\beta]^-}
[\partial^{\gamma}(r_1^{-\theta}\phi)]_{\{\beta\}^++\mu}L_4^{\mu} \\
&\quad +N\mathbf{1}_{(1,2]}(\{\beta\}^++\mu)\sum_{|\gamma|=[\beta]^-}\left([\nabla 
\partial^{\gamma}(r_1^{-\theta}\phi)]_{\{\beta\}^++\mu-1}L_4^{\mu}+
|\nabla \partial^{\gamma}(r_1^{-\theta}\phi)|_{0}|\nabla H|_0 \right).
\end{align*}
\end{tightenumerate}
\end{lemma}

\begin{proof}
(1) Since $(I_{d_1}+\nabla
H(x))^{-1}$ exists for each $x$, it follows from Theorem 0.2 in \cite{DeHoIm13} that the mapping $\tilde{H%
}$ is a global diffeomorphism. For each $x$, we easily verify $\tilde{H}%
^{-1}(x)=x-H(\tilde{H}^{-1}$$(x))$ by substituting $\tilde{H}(x)$ into the
expression. Simple computations show that for all $x$, we have 
\begin{equation*}
|\nabla \tilde{H}(x)|\leq 1+L_{2},\;\;|\nabla \tilde{H}^{-1}(x)|=|\kappa (%
\tilde{H}^{-1}(x))|\leq N_{\kappa },\;\;|\nabla F(x)|=|\nabla H(\tilde{H}%
^{-1}(x))\nabla \tilde{H}^{-1}(x)|\leq N_{\kappa }L_{2},
\end{equation*}%
and 
\begin{equation*}
|(I_{d_1}+\nabla F(x))^{-1}|=|\nabla \tilde{H}^{-1}(x)^{-1}|=|\kappa (\tilde{H}%
^{-1}(x))^{-1}|=|I_{d_1}+\nabla H(\tilde{H}^{-1}(x))|\leq 1+L_{2}.
\end{equation*}%
For all $x$ and $y$, we easily obtain 
\begin{equation*}
|\tilde{H}(x)-\tilde{H}(y)|\leq (1+L_{2})|x-y|,\quad |\tilde{H}^{-1}(x)-%
\tilde{H}^{-1}(y)|\leq N_{\kappa }|x-y|,
\end{equation*}%
and hence 
\begin{equation}
N_{\kappa }^{-1}|x-y|\leq |\tilde{H}(x)-\tilde{H}(y)|,\quad
(1+L_{2})^{-1}|x-y|\leq |\tilde{H}^{-1}(x)-\tilde{H}^{-1}(y)|.
\label{eq:boundedifferencefrombelow}
\end{equation}%
Making use of \eqref{eq:boundedifferencefrombelow}, for all $x$, we get 
\begin{equation}
N_{\kappa }^{-1}|x|\leq L_{0}+|\tilde{H}(x)|,\quad |\tilde{H}^{-1}(x)|\leq
N_{\kappa }L_{0}+N_{\kappa }|x|,\quad |x| \leq L_{0}+L_{1}|\tilde{H}^{-1}(x)|,  \label{ineq:boundinversegrowthabovebelow}
\end{equation}%
and thus 
\begin{equation*}
|F(x)|\leq L_{0}+L_{1}N_{\kappa }L_{0}+L_{1}N_{\kappa }|x|.
\end{equation*}
The rest of the estimates then follow easily from the above estimates and Taylor's theorem. 

\noindent (2) Using the chain rule, for all $x$, we obtain 
\begin{equation}
\nabla F(x)=-\nabla H(\tilde{H}^{-1}(x))\nabla \tilde{H}^{-1}(x)=-\nabla H(%
\tilde{H}^{-1}(x))\kappa (\tilde{H}^{-1}(x)),  \label{eq:gradientofF}
\end{equation}%
and hence $|\nabla F|_{0}\leq N_{\kappa }|\nabla H|_{0}.$ For all $x$ and $y$%
, we have 
\begin{equation}  \label{eq:diffofkappa}
\kappa (\tilde{H}^{-1}(y))-\kappa (\tilde{H}^{-1}(x))=\kappa (y)[\nabla H(%
\tilde{H}^{-1}(x))-\nabla H(\tilde{H}^{-1}(y))]\kappa (x),
\end{equation}%
and thus since $[\tilde{H}^{-1}]_{1}\leq(1+ N_{\kappa }L_{3})$ by part (1), we
have for all $\delta \in (0,1\wedge \beta ],$ 
\begin{equation*}
\lbrack \kappa (\tilde{H}^{-1})]_{\delta }\leq N_{\kappa }^{2}(1+ N_{\kappa }L_{3})^{\delta}[\nabla
H]_{\delta }.
\end{equation*}%
It follows that there is a constant $N=N(N_{\kappa },L_{3})$ such that for
all $\delta \in (0,1\wedge \beta ]$, 
\begin{equation*}
|\nabla F|_{\delta }\leq N|H|_{\delta }.
\end{equation*}%
It is well-known that the inverse map $\mathfrak{I}$ on the set of
invertible $d_1\times d_1$ matrices is infinitely differentiable and for
each $n$, there exists a constant $N=N(n,d_1)$ such that for all invertible
matrices $A $, the $n$-th derivative of $\mathfrak{I}$ evaluated at $A$,
denoted $\mathfrak{I}^{(n)}(A)$, satisfies 
\begin{equation*}
| \mathfrak{I}^{(n)}(A)| \leq N|A^{-n-1}|\leq N| A^{-1}| ^{n+1}.
\end{equation*}%
Using induction, we find that for all multi-indices $\gamma $ with $|\gamma
|\leq \lbrack \beta ]^{-}$ and for each $x$, $\partial ^{\gamma }F(x)$ is a
finite sum of terms, each of which is a finite product of 
\begin{equation*}
\partial ^{\bar{\gamma}}H(\tilde{H}^{-1}(x)),\quad \kappa (\tilde{H}%
^{-1}(x))^{\bar{n}},\quad \text{and}\quad \mathfrak{I}^{(\bar{n}%
-1)}(I+\nabla H(\tilde{H}^{-1}(x))),\;\;|\bar{\gamma}|\leq |\gamma |,\;\;%
\bar{n}\in \{1,\ldots ,|\gamma |\}.
\end{equation*}%
Therefore, differentiating \eqref{eq:gradientofF} and estimating directly,
we easily obtain \eqref{ineq:estimateofF}. \newline
\noindent (3) For each $x$, we have
\begin{align}  \label{eq:weightminus}
\frac{r_{1}(\tilde{H}^{-1}(x))^{\theta }}{r_{1}(x)^{\theta }}%
-1&=r_{1}(x)^{-\theta }\int_{0}^{1}r_{1}(G_s(x))^{\theta -2}G_s(x)^{* }F(x)ds
\notag \\
&=\int_{0}^{1}\frac{r_{1}^{\theta -1}( G_s(x)) }{r_{1}(x)^{\theta -1}}%
K(G_s(x))^*dsr_{1}(x)^{-1}F(x),
\end{align}%
where $G_s(x): =x+sF(x)$, $s\in [0,1]$, and $J(x):=r_{1}(x)^{-1}x.$
According to part (1) and (2), we have $| r_{1}^{-1}F| _{0}\leq N| r_{1}^{-1}H| _{0}$
and $|\nabla F|_{\beta -1}\leq N|\nabla H|_{\beta -1}$, and hence 
\begin{equation*}
| r_{1}^{-1}G_s| _{0} \leq N(1+| r_{1}^{-1}H| _{0}),\quad |\nabla
G_s(x)|_{\beta -1}\leq N(1+|\nabla H|_{\beta -1}).
\end{equation*}
and 
\begin{equation*}
| J\circ G_s| _{\beta }\leq N(1+| r_{1}^{-1}H| _{0}+|\nabla H|_{\beta -1}),
\end{equation*}%
for some constant $N$ independent of $s$.  
Moreover, using Lemma \ref{lem:GeneralCompWeightedlemma}, we find 
\begin{equation*}
| r_{1}^{\theta -1}\circ G_s \cdot r_1^{1-\theta }| _{\beta
}\leq N\left( 1+| r_{1}^{-1}H| _{0}+|\nabla H|_{\beta -1}\right) ^{\theta
+\beta }.
\end{equation*}%
The statement then  follows.

\noindent (4) First, we will consider the case $\theta=0$. By part (1), we have
 for each $\bar \mu\in
(0,(\beta+\mu )\wedge 1] $, 
\begin{equation}  \label{ineq:estimatePnotnorm}
|\phi\circ \tilde{H}^{-1}-\phi|_0\le [\phi]_{\bar \mu}|H\circ \tilde{H}%
^{-1}|_0^{\mu}\le  [\phi]_{\bar \mu}|H|_0^{\bar \mu} .
\end{equation}
First, let us consider the case  $\beta\le 1$. For each $x$, let $\mathcal{J}(x)=\phi(\tilde{H}%
^{-1}(x))-\phi(x)$. For all $x$ and $y$, it is clear that 
\begin{equation*}
|\mathcal{J}(x)-\mathcal{J}(y)|\le A(x,y)+B(x,y)+C(x,y),
\end{equation*}
where 
\begin{equation*}
A(x,y):=|\mathcal{J}(x)|\mathbf{1}_{[L_4,\infty)}(|x-y|),\quad B(x,y):=|%
\mathcal{J}(y)|\mathbf{1}_{[L_4,\infty)}(|x-y|),
\end{equation*}
and 
\begin{equation*}
C(x,y):=|\mathcal{J}(x)-\mathcal{J}(y)|\mathbf{1}_{[0,L_4)}(|x-y|).
\end{equation*}
Moreover, owing to part (1), if $\beta+\mu\le 1$, then for all  $x,$ and $%
y$, we have 
\begin{gather*}
A(x,y)\le [\phi]_{\beta+\mu }L_4^{\beta+\mu }\mathbf{1}_{[L_4,\infty)}(|x-y|)\le
[\phi]_{\beta+\mu } L_4^{\mu}|x-y|^{\{\beta\}^+},\\
B(x,y)\le [\phi]_{\beta+\mu}L_4^{\mu}|x-y|^{\beta },
\end{gather*}
and 
\begin{gather*}
C(x,y)\le [\phi]_{\beta+\mu }|[\tilde{H}^{-1}]^{\beta+\mu }_1|x-y|^{\beta+\mu }%
\mathbf{1}_{[0,L_4)}(|x-y|)+ [\phi]_{\beta+\mu }|x-y|^{\beta+\mu}\mathbf{1}_{[0,L_4)}(|x-y|)\\
\le N[\phi]_{\beta+\mu}L_4^{
\mu}|x-y|^{\beta}
\end{gather*}
for some constant $N=N(\mu,N_{\kappa},L_4)$. Using the identity 
\begin{gather*}
\mathcal{J}(x)-\mathcal{J}(y)\\
=-\int_0^1 \left(\nabla \phi\left(x-\theta H(\tilde{H}^{-1}(x))\right) - \nabla
\phi\left(y-\theta H(\tilde{H}^{-1}(y))\right)\right) H(\tilde{H}^{-1}(x))d\theta\\
- \int_0^1\nabla \phi\left(y-\theta H(\tilde{H}^{-1}(y))\right) (H(\tilde{H}^{-1}(y))-H(%
\tilde{H}^{-1}(x))),
\end{gather*}
and part (1), if  $\beta+\mu>1$, we get that there is a
constant $N=N(\mu,N_{\kappa},L_4)$ such that for all $x$ and $y$, 
\begin{gather*}
|\mathcal{J}(x)-\mathcal{J}(y)|\mathbf{1}_{[L_4,\infty)}(|x-y|)\le N([\nabla
\phi]_{\beta+\mu-1} |x-y|^{\beta+\mu-1}L_4+|\nabla \phi|_0|x-y|[H]_1)\mathbf{1}%
_{[L_4,\infty)}(|x-y|)\\
\le N[\nabla
\phi]_{\beta+\mu-1}L_4^{\mu}|x-y|^{\beta}+N|\nabla
\phi|_0|\nabla H|_0|x-y|.
\end{gather*}
Moreover, since
\begin{gather*}
\mathcal{J}(x)-\mathcal{J}(y)\\
=\int_0^1 \nabla \phi\left( \tilde{H}^{-1}(x+\theta(y-x))\right)\left(\nabla \tilde{H}%
^{-1}(x+\theta(y-x))-I_d\right)(x-y)d\theta\\
+\int_0^1 \left(\nabla \phi\left( \tilde{H}^{-1}\left(x+\theta(y-x)\right)\right)-\nabla \phi\left(
x+\theta(y-x)\right)\right)(x-y)d\theta,
\end{gather*}
by part (1) and \eqref{ineq:estimateofF}, if $\beta+\mu>1$, we have that there is a constant $N=N(\mu,N_{\kappa},L_4)$
such that for all $x$ and $y$, 
\begin{gather*}
|\mathcal{J}(x)-\mathcal{J}(y)|\mathbf{1}_{[0,L_4)}(|x-y|)\le (|\nabla
\phi|_0|\nabla H|_0+[\nabla \phi]_{\beta+\mu-1}L_4^{\beta+\mu-1})|x-y|\mathbf{1}%
_{[0,L_4)}(|x-y|)\\
\le |\nabla \phi|_0|\nabla H|_0|x-y|+[\nabla \phi]_{\beta+\mu-1}L_4^{\mu}|x-y|^{\beta}.
\end{gather*}
Combining the above estimates, we get that for all $\beta\le1$ and $%
\mu\in(0,1]$, there is a constant $N=N(\mu,N_{\kappa},L_4)$ such that 
\begin{equation}\label{ineq:mainestpart4}
[\phi\circ \tilde{H}^{-1}-\phi]_{\beta}\le  N\mathbf{1}_{[0,1]}(\beta+\mu)[\phi]_{\beta+\mu}L_4^{\mu}+N\mathbf{1}_{(1,2]}(\beta+\mu)\left([\nabla \phi]_{\beta+\mu-1}+|\nabla \phi|_{0}|\nabla H|_0\right).
\end{equation}
This proves the desired estimate for $\beta \le 1$ and $\theta=0$. We now consider the case $\beta>1$. For $\beta>1$, it is straightforward to prove by
induction that for all multi-indices $\gamma$ with $1\le |\gamma|\le
[\beta]^-$ and for all $x$, 
\begin{equation*}
\partial^{\gamma} (\phi(\tilde{H}^{-1}))(x)=\mathcal{J}^{\gamma}_1(x)+%
\mathcal{J}^{\gamma}_2(x)+\mathcal{J}^{\gamma}_3(x)+\mathcal{J}%
^{\gamma}_4(x),
\end{equation*}
where 
\begin{equation*}
\mathcal{J}^{\gamma}_1(x) :=\partial ^{\gamma }\phi(\tilde{H}^{-1}(x)),
\end{equation*}
\begin{equation*}
\mathcal{J}^{\gamma}_2(x)=\partial^{\gamma}\phi(\tilde{H}^{-1})(\partial_{1}%
\tilde{H}^{-1;1})^{\gamma_1}\cdots (\partial_{d}\tilde{H}^{-1;d})^{%
\gamma_d}-1,
\end{equation*}
$\mathcal{J}^{\gamma}_3(x)$ is a finite sum of terms of the form 
\begin{equation*}
\partial _{j_{1}}\cdots \partial _{j_{k}}\phi(\tilde{H}^{-1}(x))\partial ^{%
\tilde{\gamma}_{1}}\tilde{H}^{-1;j_{1}}(x)\cdots \partial ^{\tilde{\gamma}%
_{k}}\tilde{H}^{-1;j_{k}}(x)
\end{equation*}
with $1\le k<[\beta]^-$, $j_1,\ldots,j_k\in\{1,\ldots,d\}$, and $%
\sum_{j=1}^{k}\tilde{\gamma}_{j}=\gamma$, and $\mathcal{J}_4(x)$ is a finite
sum of terms of the form 
\begin{equation*}
\partial _{j_{1}}\ldots \partial _{j_{\left[ \beta \right] ^{-}}}\phi(\tilde{%
H}^{-1}(x))\partial _{i_{1}}\tilde{H}^{-1;j_{1}}(x)\cdots \partial _{i_{%
\left[ \beta \right] ^{-}}}\tilde{H}^{-1;j_{\left[ \beta \right] ^{-}}}(x)
\end{equation*}
with $i_1,j_1,\ldots,i_{[\beta]^-},j_{[\beta]^-}\in\{1,\ldots,d\}$ and at
least one pair $i_{k}\neq j_{k}$. Since for each $x$, 
\begin{equation*}
\nabla \tilde{H}^{-1}(x)=I+\nabla F(x)
\end{equation*}
and \eqref{ineq:estimateofF} holds, there is a constant $N=N(d_1,\beta)$
such that 
\begin{equation*}
\sum_{1\le |\gamma|\le \beta}\sum_{i=2}^{4}|\mathcal{J}^{\gamma}_{i}|_{0}+%
\sum_{|\gamma|=\beta}\sum_{i=2}^{4}|\mathcal{J}^{\gamma}_{i}|_{\{\beta\}^+}%
\leq N|\nabla \phi|_{\beta-1 }|\nabla F|_{\beta-1 } \le N|\nabla
\phi|_{\beta-1} |\nabla H|_{\beta-1 }.
\end{equation*}
If $\beta>2$, then for all multi-indices $\gamma$ with $1\le |\gamma|<[\beta]^-$, we get
\begin{equation*}
|\mathcal{J}^{\gamma}_1-\partial^{\gamma}\phi|_0=|\partial^{\gamma}\phi\circ 
\tilde{H}^{-1}-\partial^{\gamma}\phi|_0\le
[\partial^{\gamma}\phi]_{1}|H|_0.
\end{equation*}
It is easy to see that there is a constant $N=N(L_4,N_{\kappa})$ such that for all $\gamma$ with $|\gamma|=[\beta]^-$ and all $\bar \mu\in
(0,(\{\beta\}^++\mu )\wedge 1] $,
$$
|\mathcal{J}^{\gamma}_1-\partial^{\gamma}\phi|_{0}=|\partial^{%
\gamma}\phi\circ \tilde{H}^{-1}-\partial^{\gamma}\phi|_{0}\le [\partial^{\gamma}\phi]_{\bar \mu}|H|_0^{\bar \mu}.
$$
Moreover, appealing to the estimate \eqref{ineq:mainestpart4}, we obtain
\begin{gather*}
[\mathcal{J}^{\gamma}_1-\partial^{\gamma}\phi]_{\{\beta\}^+}\\
\le N\mathbf{1}_{[0,1]}(\{\beta\}^++\mu)[\partial^{\gamma}\phi]_{\{\beta\}^++\mu}L_4^{\mu}+N\mathbf{1}_{(1,2]}(\{\beta\}^++\mu)\left([\nabla \partial^{\gamma}\phi]_{\{\beta\}^++\mu-1}+|\nabla \partial^{\gamma} \phi|_{0}|\nabla H|_0\right) .
\end{gather*}
 Let us now
consider the case $\theta>0$. The following decomposition obviously holds
for all $x$: 
\begin{equation*}
r_1(x)^{-\theta}\phi(\tilde{H}^{-1}(x))-r_1(x)^{-\theta}\phi(x)=\hat{\phi}(%
\tilde{H}^{-1})-\hat{\phi}(x)+\left(\frac{r_1(\tilde{H}^{-1}(x))^{\theta}}{%
r_1(x)^{\theta}}-1\right)\hat{\phi}(\tilde{H}^{-1}(x)),
\end{equation*}
where $\hat{\phi}=r_1^{-\theta}\phi\in C^{\beta}( \mathbf{R}^{d_1};\mathbf{R}%
^{d_1}).$ Thus, to complete the proof we require
\begin{equation*}
|\hat{\phi}\circ \tilde{H}^{-1}|_{\beta}\le N|\hat{\phi}|_{\beta} \quad 
\text{and}\quad \left|\frac{r^{\theta}_1\circ\tilde{H}^{-1}}{r_1^{\theta}}%
-1\right|_{\beta}\le N(|H|_0+|\nabla H|_{\beta\vee 1-1}).
\end{equation*}
The latter inequality was proved in part (3) and the first inequality
follows from part (2) and Lemma \ref{lem:GeneralCompWeightedlemma}.
\end{proof}

\begin{remark}\label{rem:GradientSmallInverse}
Let $H:\bR^{d_1}\rightarrow\bR^{d_1}$ be continuously differentiable and assume that for all $x$, 
\begin{equation*}
|\nabla H(x)|\leq \eta <1.
\end{equation*}%
Then for each $x\in \mathbf{R}^{d_1}$, 
\begin{equation*}
|(I_{d_1}+\nabla H(x))^{-1}|\leq | I_{d_1}+\sum_{k=1}^{\infty }(-1)^{k}\nabla H(x)^{k}| \leq 
\frac{1}{1-\eta }.
\end{equation*}
\end{remark}

\subsection{Stochastic Fubini thoerem}

Let $m=(m^{\varrho})_{t\le T}$, $\varrho\ge 1$, be a sequence of $\bF$-adapted locally square integrable continuous martingales issuing from zero such that $\bP$-a.s.\ for all $t\in [0,T]$, $\langle m^{\varrho_1},m^{\varrho_2}\rangle_t=0$ for $\varrho_1\ne \varrho_2$ and $\langle m^{\varrho}\rangle _t=N_t$ for $\varrho\ge 1$, where $N_t$ is a $\cP_T$-measurable continuous increasing processes issuing from zero.    Let $\eta(dt,dz)$ be a $\bF$-adapted  integer-valued random  measure on  $([0,T]\times E,\cB([0,T])\otimes \cE)$, where $(U,\cU)$ is a Blackwell space.  We assume that $\eta(dt,dz)$ is optional,  $\tilde{\cP}_T$-sigma-finite, and quasi-left continuous.  Thus, there exists a unique (up to a $\bP$-null set) dual predictable projection (or compensator) $\eta^p(dt,dz)$ of $\eta(dt,dz)$ such that $\mu(\omega,\{t\}\times U)=0$ for all $\omega$ and $t$. We refer the reader to Ch.  II, Sec. 1, in \cite{JaSh03} for any unexplained concepts relating to random measures.

Let $(X,\Sigma,\mu)$ be a sigma-finite measure space; that is, there is an increasing sequence of $\Sigma$-measurable sets $X_n$, $n\in \bN$, such that $X=\cup_{n=1}^{\infty} X_n$ and $\mu(X_n)<\infty$ for each $n$.   Let $f:\Omega\times [0,T] \times X\rightarrow \bR^{d_2}$ be $\mathcal{R}%
_{T}\otimes \Sigma$-measurable, $g:\Omega\times [0,T]
\times X\rightarrow \ell_{2}(\bR^{d_2})$ be $\mathcal{R}_{T}\otimes 
\Sigma/\mathcal{B}(\ell_{2}(\bR^{d_2}))$-measurable,  and $%
h:\Omega\times [0,T] \times X\times U\rightarrow \bR^{d_2}$
be $\mathcal{P}_{T}\otimes \Sigma\otimes \mathcal{U}$%
-measurable. Moreover, assume that for each $t\in [0,T]$ and $x\in X$, $\bP$-a.s.\ 
\begin{equation*}
\int_{]0,T]}|g_t(x)|^{2}dN_t+\int_{]0,T]}\int_{U}|h_t(x,z)|^{2}\eta^p(dt,dz)<\infty.
\end{equation*}%
Let  $F=F_t(x):\Omega\times [0,T]\times X\rightarrow \bR^{d_2}$ be $\cO_T\otimes \cB(X)$-measurable  and assume that for $d\bP\otimes \mu$-almost all $(t,x)\in [0,T]\times X$, 
$$
F_t(x) =\int_{]0,t]}g^{\varrho}_s(x)dm^{\varrho}_s+\int_{]0,t]}\int_{U}h_s(x,z)\tilde{\eta}(dt,dz)
$$
where $\tilde{\eta}(dt,dz)=\eta(dt,dz)-\eta^p(dt,dz)$.

The following version of the stochastic Fubini theorem is a straightforward extension of Lemma 2.6 \cite{Kr11}  and  Corollary 1 in 
\cite{Mi83}. See also Proposition 3.1 in \cite{Zh13f}, Theorem 2.2 in \cite{Ve12}, and Theorem 1.4.8 in  \cite{Ro90}. Indeed, to prove it for a bounded measure, we can use   a  monotone class argument as in Theorem 64 in \cite{Pr05}. To handle the general setting with possibly infinite $\mu$, we use assumptions (ii) and (iii) below and take limits on the sets $X_n$ using the Lenglart domination lemma  Lenglart domination lemma (Theorem 1.4.5 on page 66 in \cite{LiSh89}) and the following well known inequalities:   
$$
\bE \sup_{t\le T}\left|\int_{]0,t]} g_s^{\varrho}dw_s^{\varrho}\right|\le N\bE \left(\int_{]0,T]} |g_t(x)|^2dw_t^{\varrho}\right)^{1/2}
$$
$$
\bE \sup_{t\le T}\left|\int_{]0,t]} \int_{U}h_t(x,z)\tilde{\eta}(dt,dz)\right|\le N\bE \left(\int_{]0,T]} \int_{U}|h_t(x,z)|^2\eta^p(dt,dz)\right)^{1/2},
$$
where $\tau\le T$ is an arbitrary stopping time and  $N=N(T)$  is a constant independent of $g$ and $h$. 

\begin{proposition}[c.f. Corollary 1 in 
\cite{Mi83} and Lemma 2.6 in \cite{Kr11}]
Assume that 
\begin{tightenumerate}
\item  $\bP$-a.s.\ for each $n\ge 1$,
\begin{equation}\label{asm:FubiniAssumptiongen}
\int_{X_n}\left(\int_{]0,T]}|g_t(x)|^2dN_t\right)^{1/2}\mu(dx) +\int_{X_n}\left(\int_{]0,T]}\int_{U_1}|h_t(x,z)|^2\eta^p(dt,dz)\right)^{1/2}\mu(dx)<\infty;
\end{equation}
\item $\bP$-a.s.\ 
$$
\int_{]0,T]} \left(\int_X |g_t(x)|\mu(dx)\right)^2dt+\int_{]0,T]} \int_U  \left(\int_X |h_t(x,z)|\mu(dx)\right)^2\eta^p(dt,dz);
$$
\item $\bP$-a.s.\ for al $t\in [0,T]$,
$$
\int_X |F_t(x)|\mu(dx)<\infty.
$$
\end{tightenumerate}
Then 
$\bP$-a.s.\  for all $t\in [0,T]$,
\begin{equation}\label{eq:stochasticFubinigen}
\int_{X}F_t(x)\mu(dx) =\int_{]0,t]}\int_{X}g^{\varrho}_s(x)\mu(dx)dm^{\varrho}_{s}+\int_{]0,t]}\int_{U}\int_{X}h_s(x,z)\mu(dx)\tilde{\eta}(dr,dz)
\end{equation}
\end{proposition}

We obtain the following corollary by applying Minkowski's integral inequaility.  

\begin{corollary}\label{cor:StochasticFubini}
Assume that $\bP$-a.s.\
 \begin{equation}\label{asm:FubiniAssumption}
\int_{X}\left(\int_{]0,T]}|g_t(x)|^2dN_t\right)^{1/2}\mu(dx) +\int_{X}\left(\int_{]0,T]}\int_{U_1}|h_t(x,z)|^2\eta^p(dt,dz)\right)^{1/2}\mu(dx)<\infty.
\end{equation}
Then $\bP$-a.s.\  for all $t\in [0,T]$,
\begin{equation}\label{eq:stochasticFubini}
\int_{X}F_t(x)\mu(dx) =\int_{]0,t]}\int_{X}g^{\varrho}_s(x)\mu(dx)dm^{\varrho}_{s}+\int_{]0,t]}\int_{U}\int_{X}h_s(x,z)\mu(dx)\tilde{\eta}(dr,dz).
\end{equation}
\end{corollary}

\begin{remark}\label{rem:stochasticfubini}
If $\mu$ is a finite-measure and  $\bP$-a.s.\
$$
\int_{X}\int_{]0,T]}|g_t(x)|^2dN_t\mu(dx)+\int_{X}\int_{]0,T]}\int_{U_1}|h_t(x,z)|^2\eta^p(dt,dz)\mu(dx)<\infty,
$$ 
then  \eqref{asm:FubiniAssumption} holds by H{\"o}lder's inequality. 
\end{remark}

\subsection{It\^{o}-Wentzell formula}

\begin{definition}
We say that an $\bR^{d_1}$-valued $\mathbf{F}$-adapted quasi-left continuous semimartingale $L_{t}=(L_t^k)_{1\le k\le d_1}$, $t\ge 0$,
is of $\alpha $-order for some  $\alpha \in (0,2]$ if $\bP$-a.s.\ for all $t\ge 0$,
$$
\sum_{s\leq t}|\Delta L_{s}|^{\alpha }<\infty
$$
and 
\begin{align*}
L_{t} &=L_{0}+\int_{]0,t]}\int_{\mathbf{R}_0^{d_1}} z p^{L}(ds,dz),\text{ if }%
\alpha \in (0,1), \\
L_{t} &=L_{0}+A_{t}+\int_{]0,t]}\int_{| z| \leq 1}z
q^{L}(ds,dz)+\int_{]0,t]}\int_{| z| >1}z p^{L}(ds,dz), 
\text{ if }\alpha \in \lbrack 1,2), \\
L_{t} &=L_{0}+A_{t}+L_{t}^{c}+\int_{]0,t]}\int_{| z|
\leq 1}z q^{L}(ds,dz)+\int_{]0,t]}\int_{| z| >1}z
p^{L}(ds,dz), \text{ if }\alpha =2,
\end{align*}%
where $p^{L}(dt,dz)$ is the jump measure of $L$ with dual predictable
projection $\pi^L (dt,dz)$, $q^{L}$ $(dt,dz)$ $=p^{L}(dt,dz) -\pi^L (dt,dz)$
is a martingale measure, $A_{t}=(A_t^i)_{1\le i\le d_1}$ is a continuous process of finite
variation with $A_0=0$, and $L_{t}^{c}=(L_t^{c;i})_{1\le i\le d_1}$ is a continuous local
martingale issuing from zero.
\end{definition}

Set $(w^{\varrho})_{\varrho\ge 1}=(w^{1;\varrho})_{\varrho\ge 1}$, $(Z, \cZ,\pi )=(\cZ^{1},\cZ^{1},\pi ^{1}),$ $%
p(dt,dz )=p^{1}(dt,dz )$, and $q(dt,dz )=q^{1}(dt,dz )$.  Also, set $D=D^{1}$, $E=E^{1},$ and assume $Z=D\cup E$.

Let $f:\Omega\times [0,T] \times \bR^{d_1}\rightarrow \bR^{d_2}$ be $\mathcal{R}%
_{T}\otimes \cB(\bR^{d_1})$-measurable, $g:\Omega\times [0,T]
\times \bR^{d_1}\rightarrow \ell_{2}(\bR^{d_2})$ be $\mathcal{R}_{T}\otimes 
\cB(\bR^{d_1})/\mathcal{B}(\ell_{2}(\bR^{d_2}))$-measurable,  and $%
h:\Omega\times [0,T] \times \bR^{d_1}\times Z\rightarrow \bR^{d_2}$
be $\mathcal{P}_{T}\otimes \cB(\bR^{d_1})\otimes \cZ$%
-measurable. Moreover, assume that, $\bP$-a.s.\ for all $x\in \bR^{d_1}$,
\begin{equation*}
\int_{]0,T]}|f_t(x)|dt+\int_{]0,T]}|g_t(x)|^{2}dt+\int_{]0,T]}\int_{D}|h_t(x,z)|^{2}\pi(dz)dt+\int_{]0,T]}\int_{E}|h_t(x,z)|\pi(dz)dt<\infty.
\end{equation*}%
Let $F=F_t(x):\Omega\times [0,T]\times \bR^{d_1}\rightarrow \bR^{d_2}$ be $\cO_T\otimes \cB(\bR^{d_1})$-measurable and  assume that for each $x$, $\bP$-a.s.\ for all $t$,
$$
F_t(x) =F_0(x)+\int_{]0,t]}f_s(x)ds+\int_{]0,t]}g^{\varrho}_s(x)dw^{\varrho}_{s}+\int_{]0,t]}\int_{Z}h_s(x,z)[\mathbf{1}_{D}(z)q(ds,dz) +\mathbf{1}_{E}(z)p(ds,dz)].
$$
For each $n\in \{1,2\}$, let $C ^{n}_{loc}(\bR^{d_1};\bR^{d_2})$ be space of $n$-times continuously differentiable functions $f:\bR^{d_1}\rightarrow \bR^{d_2}$.  We now state our version of the It{\^o}-Wentzell formula. For each $\omega,t$ and $x$, we denote $\Delta F(x)=F_t(x)-F_{t-}(x)$.

\begin{proposition}[cf. Proposition 1 in \cite{Mi83} ]
\label{prop:Ito-Ventzel} Let $(L_t)_{t\ge 0}$ be an $\bR^{d_1}$-valued quasi-left continuous
semimartingale of  order $\alpha \in (0,2]$.  Assume  that:
\begin{tightenumerate}
\item
\begin{enumerate}[(a)]
\item  $\bP$-a.s.\  $F\in D([0,T];\cC_{loc}^{\alpha}(\bR^d;\bR^m)$ if $\alpha$ is fractional and $F\in D([0,T];\allowbreak C_{loc}^{\alpha}\allowbreak (\bR^d;\bR^m)$ if $\alpha=1,2$ ;
\item   for $d\bP dt$-almost-all $(\omega,t)\in \Omega\times [0,T]$,  $f_t(x)$ and $g_t(x)=(g^{i\varrho}_t(x))_{\varrho\ge 1}\in \ell_2(\bR^{d_2})$ are continuous in $x$ and 
$$
d\bP dt-\lim_{y\rightarrow x}\left[\int_{D}|h_t(y,z)-h_t(x,z)|
^{2}\pi (dz) +\int_{E}|h_t(y,z)-h_t(x,z)|\pi (dz)\right]= 0;
$$
\item for each $\rho\ge1$ and $i\in\{1,\ldots,d_1\}$ and for $d\bP d |\langle L^{c;i},w^{\varrho}\rangle|_t$-almost-all $(\omega,t)\in \Omega\times [0,T]$,  $g^{i\varrho}_t\in C^1_{loc}(\bR^d;\bR)$, if $\alpha= 2$, ;
\end{enumerate}
\item for each compact subset $K$ of $\bR^{d_1}$, $\bP$-a.s.\
$$
\int_{]0,T]}\sup_{x\in K} \left(|f_t(x)|+|g_t(x)|^2+\int_{D}|h_t(x,z)|^2\pi(dz)+ \int_{E}|h_t(x,z)|\pi(dz)\right)dt<\infty,
$$
$$
\sum_{\varrho\ge 1}\int_{]0,T]}\sup_{x\in K}| \nabla g^{i\varrho}_t(x)|d|\langle L^{c;i},w^{\varrho}\rangle|_t+
\sum_{t\le T}|\Delta F_t|_{\alpha\wedge 1;K}|\Delta L_t|^{\alpha\wedge 1}<\infty.
$$
\end{tightenumerate}
Then $\bP$-a.s for all $t\in [0,T]$,
\begin{align}\label{eq:Ito-Wentzellformula}
F_t(L_t)&=F_0(L_{0})+\int_{]0,t]}f_s(L_{s})ds+\int_{]0,t]}g_s^{\varrho}
(L_{s})dw^{\varrho}_{s}\notag\\
&\quad +\int_{]0,t]}\int_{Z}h_s(L_{s-},z)[\mathbf{1}_{D}(z)q(dr,dz) +\mathbf{1}_{E}(z)p(dr,dz)]\notag\\
&\quad +\int_{]0,t]}\partial_i F_{s-}(L_{s-})[\mathbf{1}_{[1,2]}(\alpha)dA^i_s+\mathbf{1}_{\{2\}}(\alpha)dL ^{c;i}_s]+\mathbf{1}_{\{2\}}(\alpha)\frac{1}{2}\int_{]0,t]}\partial_{ij}F_s(L_{s})d\langle
L^{c;i},L^{c;j}\rangle _{s}\notag\\
&\quad +\sum_{s\le t}\left(F_{s-}(L_{s})-F_{s-}(L_{s-})-\mathbf{1}_{[1,2]}(\alpha)\nabla F_{s-}(L_{s-})\Delta L_s\right)\notag\\
&\quad +\mathbf{1}_{\{2\}}(\alpha)\int_{]0,t]}\partial_ig^{\varrho}_{s}(L_{s})d\langle w^{\varrho},L^{c;i}\rangle _{s}+\sum_{s\leq t}\left(\Delta F_s(L_{s})-\Delta F_s(L_{s-})\right).
\end{align}
\end{proposition}
\begin{proof}
Since both sides have identical jumps and we can always interlace a finite set of jumps,  we may assume that  $|\Delta L_{t}|\le 1$ for all $t\in [0,T]$; that is, 
it is enough to prove the statement for $\tilde{L}_{t}=L_{t}-\sum_{s\leq
t}\mathbf{1}_{[1,\infty)}(|\Delta L_s|)\Delta L_{s}$, $t\in [0,T]$.  It suffices to assume that for some $K$ and all $\omega$, $|L_0|\le K$. For each  $R>K$, let $$
\tau_R=\inf \left( t\in [0,T]:|A|_t+|\langle L^c\rangle| _t+\sum_{s\le t}|\Delta L_s|^{\alpha}+|L_{t}|>
R\right) \wedge T$$  and  note that  $\bP$-a.s.\ $\tau_R\uparrow T$ as $R$ tends to infinity. 
If instead of $L, f, g, h,$ and $F$, we take $L_{\cdot\wedge \tau_R}$, $f\mathbf{1}_{(0,\tau_R]}$, $g^{\varrho}\mathbf{1}_{(0,\tau_R]}$,  $h\mathbf{1}_{(0,\tau_R]}$, $F\mathbf{1}_{(0,\tau_R]}$, then the assumptions  of the proposition hold for this new set of processes. Moreover, if we can prove \eqref{eq:Ito-Wentzellformula} for this new set of processes, then by taking the limit as $R$ tends to infinity, we obtain \eqref{eq:Ito-Wentzellformula}. Therefore, we may assume that for some $R>0$, $\bP$-a.s.\ for all $t\in [0,T]$,
\begin{equation}\label{ineq:boundsfubini}
|A|_t+|\langle L^c\rangle|_t+\sum_{s\le t}|\Delta L_s|^{\alpha}+|L_{t}|\le R.
 \end{equation}

Let $\phi\in C_{c}^{\infty }( \bR^{d_1},\mathbf{R})$ have support in the unit ball in $\bR^{d_1}$ and  satisfy  $\int_{%
\bR^{d_1}}\phi(x)dx=1,\phi(x)=\phi(-x),$ and $\phi(x)\geq 0$, for all $x\in\bR^{d_1}$. For
each $\varepsilon \in(0,1)$, let $\phi_{\varepsilon }(x)=\varepsilon ^{-d}\phi\left(
x/\varepsilon \right) ,x\in \bR^{d_1}$. By It\^{o}'s formula, for each $x\in\bR^{d_1}$,  $\bP$-a.s.\ for all $t\in [0,T]$, 
\begin{align}\label{eq:productruleIto-Wentzell}
F_t(x) \phi_{\varepsilon }(x-L_{t}) &= F_0(x)\phi_{\varepsilon }(x-L_0)-\int_{]0,t]}F_{s-}(x) \partial _{i}\phi_{\varepsilon
}(x-L_{s-})dL_{s}^{i}\notag\\
&\quad  + \mathbf{1}_{\{2\}}(\alpha)\frac{1}{2}\int_{]0,t]}F_s(x)\partial _{ij}\phi_{\varepsilon
}(x-L_{s})d\langle
L^{c;i},L^{c;j}\rangle _{s}+\int_{]0,t]}\phi_{\varepsilon }\left( x-L_{s}\right)
f_s(x)ds\notag\\
&\quad + \mathbf{1}_{\{2\}}(\alpha)\int_{]0,t]}g^{\varrho}_{s}(x)\partial _{i}\phi_{\varepsilon
}(x-L_{s})d\langle
w^{\varrho},L^{c;i}\rangle _{s}+\int_{]0,t]}\phi_{\varepsilon }\left( x-L_{s}\right) g^{\varrho}_s(x)dw^{\varrho}_{s} \notag\\
&\quad +\int_{]0,t]}\int_{Z}\phi_{\varepsilon }\left( x-L_{s-}\right)
h_s(x,z)[\mathbf{1}_{D}(z)q(dr,dz)+\mathbf{1}_{E}(z)p(dr,dz)]\\
&\quad +\sum_{s\leq t}\Delta F_s(x) \left(\phi_{\varepsilon
}(x-L_{s})-\phi_{\varepsilon }\left( x-L_{s-}\right) \right) \notag\\
&\quad +\sum_{s\leq t}F_{s-}(x)\left( \phi_{\varepsilon }(x-L_{s})-\phi_{\varepsilon }(x-L_{s-})+\partial _{i}\phi_{\varepsilon }\left(x-
L_{s-}\right) \Delta L_{s}\right).
\end{align}
Appealing to assumption (2) and \eqref{ineq:boundsfubini} (i.e. for the integrals against $F$),  we  integrate both sides of the above in $x$, apply Corollary \ref{cor:StochasticFubini} (see, also, Remark \ref{rem:stochasticfubini}) and the deterministic Fubini theorem, and then integrate by parts  to get that   $\bP$-a.s.\ for all $t\in [0,T]$, 
\begin{align}\label{eq:Ito-Wentzellapprox}
F^{(\varepsilon) }_{t }(L_{t})
&=F^{(\varepsilon)}_0(L_0)+\int_{]0,t]}\nabla F^{(\varepsilon)
}_{s-}(L_{s-})[\mathbf{1}_{[1,2]}(\alpha)dA^i_s+\mathbf{1}_{\{2\}}(\alpha)dL ^{c;i}_s]+\int_{]0,t]}f^{(\varepsilon) }_s(L_{s})dr \notag\\
&\quad +\int_{]0,t]}g^{(\varepsilon) }_s(L_{s})dw^{\varrho}_{s}+\int_{]0,t]}\int_{Z}h^{(\varepsilon) }_s(L_{s-},z)[\mathbf{1}_{D}(z)q(dr,dz)+\mathbf{1}_{E}(z)p(dr,dz)]\notag\\
&\quad +\mathbf{1}_{\{2\}}(\alpha)\frac{1}{2}\int_{]0,t]}\partial_{ij}F^{(\varepsilon)
}_s(L_{s})d\langle L^{c;i},L^{c;j}\rangle _{s} + \mathbf{1}_{\{2\}}(\alpha)\int_{]0,t]}\partial _{i}g
^{(\varepsilon);\varrho }_s (L_{s})d\langle w^{\varrho},L^{c;i}\rangle_s
\notag\\
&\quad +\sum_{s\leq t}\left(\Delta F^{(\varepsilon) }_s(L_{s})-\Delta F^{(\varepsilon)
}_s(L_{s-})\right)\notag\\
&\quad +\sum_{s\leq t}\left(F^{(\varepsilon) }_{s-}(L_{s})-F^{(\varepsilon)
}_{s-}(L_{s-})-\mathbf{1}_{[1,2]}(\alpha)\nabla F^{(\varepsilon) }_{s-}(L_{s-})\Delta L_s\right) \\
\end{align}
where for each $\omega,t,x,$ and $z$,
\begin{equation*}
F^{(\varepsilon )}_t(x):=\phi_{\varepsilon }\ast F_t(x),\;\;f^{(\varepsilon)
}_t=\phi_{\varepsilon }\ast f_t(x),\;\;g ^{(\varepsilon);\varrho }_t(x)=\phi_{\varepsilon
}\ast g^{\varrho}_t (x),\;\;h^{(\varepsilon) }_t(x,z)=\phi_{\varepsilon }\ast h_t(x,z),
\end{equation*}%
and $\ast $ denotes the convolution operator on $\bR^{d_1}$.  Let $B_{R+1}=\{x\in\bR^{d_1}: |x|\le R+1\}$.   Owing to assumption (1)(a) and standard properties of mollifiers,  for each multi-index $\gamma$ with $|\gamma|\le \alpha$, $\bP$-a.s.\ for all $t$,
$$
|\partial^{\gamma}F^{(\varepsilon )}_t(L_t)|\le \sup_{t\le T}\sup_{x\in B_{R+1}}| \partial^{\gamma}F_t(x)|<\infty
$$
and for each  $x$,
\begin{equation}\label{eq:convergenceFepsilon}
d\bP dt-\lim_{\varepsilon\downarrow 0}| \partial^{\gamma}F^{(\varepsilon )}_t(x)-\partial^{\gamma}F^{(\varepsilon )}_t(x)|=0.
\end{equation}
Similarly, by assumption 1(b), $d\bP dt$-almost-all $(\omega,t)\in \Omega\times [0,T]$,
\begin{gather*}
|f^{(\varepsilon)}_t(L_t)|\le \sup_{x\in B_{R+1}}|f_t(x)|<\infty, \quad  |g^{(\varepsilon)}_t(L_t)|\le \sup_{x\in B_{R+1}}|g_t(x)|<\infty,  \\
\int_{D}|h^{\varepsilon}_t(L_t,z)|^2\pi(dz)\le  \sup_{x\in B_{R+1}}\int_{D}|h_t(x,z)|^2\pi(dz),\\
\int_{E}|h^{\varepsilon}_t(L_t,z)|\pi(dz)\le \sup_{x\in B_{R+1}}\int_{E}|h_t(x,z)|\pi(dz)
\end{gather*}
and for each $x$,
$$
d\bP dt-\lim_{\varepsilon\downarrow 0} |f^{(\varepsilon)}_t(x) -f_t(x)|=0,  \quad d\bP dt-\lim_{\varepsilon\rightarrow 0}| g^{(\varepsilon)}_t(x) -g_t(x)|=0
$$
and
\begin{equation}\label{eq:hconvergstochasticfubini}
d\bP dt-\lim_{\varepsilon\downarrow 0}\int_{Z}[\mathbf{1}_{D}(z)|h^{(\varepsilon )}_t(x,z)-h_t(x,z)|^2+\mathbf{1}_{E}(z)|h^{(\varepsilon )}_t(x,z)-h_t(x,z)|]\pi(dz)=0,
\end{equation}
where in the last-line we have also used Minkowski's integral inequality and a standard mollifying convergence argument.  Using assumption 1(d),  for each $\rho\ge1$ and $i\in\{1,\ldots,d_1\}$ and for $d\bP d |\langle L^{c;i},w^{\varrho}\rangle|_t$-almost-all $(\omega,t)\in \Omega\times [0,T]$
$$
| \nabla g^{(\varepsilon);i\varrho}_t(L_t)|\le\sup_{x\in B_{R+1}}| \nabla g^{i\varrho}_t(x)|
$$
and for each $x$,
 $$
 d\bP d |\langle L^{c;i},w^{\varrho}\rangle|_t-\lim_{\varepsilon\rightarrow 0}| \nabla g^{(\varepsilon);i\varrho}_t(x) -\nabla g^{i\varrho}_t(x)|=0, \;\;\textrm{if } \alpha= 2.  
 $$
  Owing to assumption  1(a) and \eqref{ineq:boundsfubini}, $\bP$-a.s.\, 
\begin{gather*}
\sum_{s\le t}|F^{(\varepsilon) }_{s-}(L_{s})-F^{(\varepsilon)
}_{s-}(L_{s-})-\mathbf{1}_{[1,2]}(\alpha)\nabla F^{(\varepsilon) }_{s-}(L_{s-})\Delta L_s|\\
\le \sup_{t\le T}|F_{t}|_{\alpha;B_{R+1}}
\sum_{s\le t}|\Delta L_s|^{\alpha }\le R\sup_{t\le T}|F_{t}|_{\alpha;B_{R+1}}.
\end{gather*}
Since $\bP$-a.s.\ $F\in D([0,T];\cC^{\alpha}(\bR^d;\bR^m)$, it follows that for each $x$, $\bP$-a.s.\ for all $t$,
$$
\lim_{\varepsilon\downarrow 0}|\Delta F_t^{\varepsilon}(x)-\Delta F_t(x)|=0.
$$
By assumption (2), $\bP$-a.s for all $t$, we have
$$
\sum_{s\leq t}\left(\Delta F^{(\varepsilon) }_s(L_{s-}+\Delta L_{s})-\Delta F^{(\varepsilon)
}_s(L_{s-})\right)\le  \sum_{s\leq t}|\Delta F_{t}|_{\alpha\wedge 1;B_{R+1}}|\Delta L_s|^{\alpha \wedge 1}.
$$
Combining the above and using assumptions (1)(a) and (2) and the bounds given in \eqref{ineq:boundsfubini}  and the deterministic and stochastic dominated convergence theorem, we obtain convergence of all the terms in  \eqref{eq:Ito-Wentzellapprox}, which complete the proof.
\end{proof}

\bibliographystyle{alpha}
\bibliography{../../bibliography}

\begin{thebibliography}{BvNVW08}

\bibitem[BvNVW08]{BrNeVeWe08}
Z.~Brze{{\'z}}niak, J.~M. A.~M. van Neerven, M.~C. Veraar, and L.~Weis.
\newblock It\^o's formula in {UMD} {B}anach spaces and regularity of solutions
  of the {Z}akai equation.
\newblock {\em J. Differential Equations}, 245(1):30--58, 2008.

\bibitem[CK10]{ChKi10}
Zhen-Qing Chen and Kyeong-Hun Kim.
\newblock An l p-theory of non-divergence form spdes driven by l{\'e}vy
  processes.
\newblock {\em arXiv preprint arXiv:1007.3295}, 2010.

\bibitem[DHI13]{DeHoIm13}
Arnaud Debussche, Michael H{{\"o}}gele, and Peter Imkeller.
\newblock {\em The dynamics of nonlinear reaction-diffusion equations with
  small {L}{\'e}vy noise}, volume 2085 of {\em Lecture Notes in Mathematics}.
\newblock Springer, Cham, 2013.

\bibitem[DPMT07]{DaMeTu07}
Giuseppe Da~Prato, Jose-Luis Menaldi, and Luciano Tubaro.
\newblock Some results of backward {I}t\^o formula.
\newblock {\em Stoch. Anal. Appl.}, 25(3):679--703, 2007.

\bibitem[DPZ92]{DaZa92}
Giuseppe Da~Prato and Jerzy Zabczyk.
\newblock {\em Stochastic equations in infinite dimensions}, volume~44 of {\em
  Encyclopedia of Mathematics and its Applications}.
\newblock Cambridge University Press, Cambridge, 1992.

\bibitem[GM11]{GrMi11}
B.~Grigelionis and R.~Mikulevicius.
\newblock Nonlinear filtering equations for stochastic processes with jumps.
\newblock In {\em The {O}xford handbook of nonlinear filtering}, pages 95--128.
  Oxford Univ. Press, Oxford, 2011.

\bibitem[Gri76]{Gr76}
B.~Grigelionis.
\newblock Reduced stochastic equations of nonlinear filtering of random
  processes.
\newblock {\em Litovsk. Mat. Sb.}, 16(3):51--63, 232, 1976.

\bibitem[Gy{\"o}82]{Gy82}
I.~Gy{\"o}ngy.
\newblock On stochastic equations with respect to semimartingales. {III}.
\newblock {\em Stochastics}, 7(4):231--254, 1982.

\bibitem[Hau05]{Ha05}
Erika Hausenblas.
\newblock Existence, uniqueness and regularity of parabolic {SPDE}s driven by
  {P}oisson random measure.
\newblock {\em Electron. J. Probab.}, 10:1496--1546, 2005.

\bibitem[H{\O}UZ10]{HoOkUbZh10}
Helge Holden, Bernt {\O}ksendal, Jan Ub{\o}e, and Tusheng Zhang.
\newblock {\em Stochastic partial differential equations}.
\newblock Universitext. Springer, New York, second edition, 2010.
\newblock A modeling, white noise functional approach.

\bibitem[Jac79]{Ja79}
Jean Jacod.
\newblock {\em Calcul stochastique et probl{\`e}mes de martingales}, volume 714
  of {\em Lecture Notes in Mathematics}.
\newblock Springer, Berlin, 1979.

\bibitem[JS03]{JaSh03}
Jean Jacod and Albert~N. Shiryaev.
\newblock {\em Limit theorems for stochastic processes}, volume 288 of {\em
  Grundlehren der Mathematischen Wissenschaften [Fundamental Principles of
  Mathematical Sciences]}.
\newblock Springer-Verlag, Berlin, second edition, 2003.

\bibitem[KR77]{KrRo77}
N.~V. Krylov and B.~L. Rozovski{\u\i}.
\newblock The {C}auchy problem for linear stochastic partial differential
  equations.
\newblock {\em Izv. Akad. Nauk SSSR Ser. Mat.}, 41(6):1329--1347, 1448, 1977.

\bibitem[KR81]{KrRo81}
N.~V. Krylov and B.~L. Rozovski{\u\i}.
\newblock On the first integrals and {L}iouville equations for diffusion
  processes.
\newblock In {\em Stochastic differential systems ({V}isegr{\'a}d, 1980)},
  volume~36 of {\em Lecture Notes in Control and Information Sci.}, pages
  117--125. Springer, Berlin-New York, 1981.

\bibitem[Kry99]{Kr99}
N.~V. Krylov.
\newblock An analytic approach to {SPDE}s.
\newblock In {\em Stochastic partial differential equations: six perspectives},
  volume~64 of {\em Math. Surveys Monogr.}, pages 185--242. Amer. Math. Soc.,
  Providence, RI, 1999.

\bibitem[Kry11]{Kr11}
N.~V. Krylov.
\newblock On the {I}t\^o-{W}entzell formula for distribution-valued processes
  and related topics.
\newblock {\em Probab. Theory Related Fields}, 150(1-2):295--319, 2011.

\bibitem[Kun81]{Ku81}
Hiroshi Kunita.
\newblock On the decomposition of solutions of stochastic differential
  equations.
\newblock In {\em Stochastic integrals ({P}roc. {S}ympos., {U}niv. {D}urham,
  {D}urham, 1980)}, volume 851 of {\em Lecture Notes in Math.}, pages 213--255.
  Springer, Berlin-New York, 1981.

\bibitem[Kun86]{Ku86}
H.~Kunita.
\newblock {\em Lectures on stochastic flows and applications}, volume~78 of
  {\em Tata Institute of Fundamental Research Lectures on Mathematics and
  Physics}.
\newblock Published for the Tata Institute of Fundamental Research, Bombay; by
  Springer-Verlag, Berlin, 1986.

\bibitem[Kun04]{Ku04}
Hiroshi Kunita.
\newblock Stochastic differential equations based on {L}{\'e}vy processes and
  stochastic flows of diffeomorphisms.
\newblock In {\em Real and stochastic analysis}, Trends Math., pages 305--373.
  Birkh{\"a}user Boston, Boston, MA, 2004.

\bibitem[LM14a]{LeMi14a}
James-Michael Leahy and Remigijus Mikulevicius.
\newblock On degenerate linear stochastic evolutions equations driven by jump
  processes.
\newblock {\em arXiv preprint arXiv:1406.4541}, 2014.

\bibitem[LM14b]{LeMi14b}
James-Michael Leahy and Remigijus Mikulevicius.
\newblock On some properties of space inverses of stochastic flows.
\newblock {\em arXiv preprint arXiv:1411.6277}, 2014.

\bibitem[LS89]{LiSh89}
R.~Sh. Liptser and A.~N. Shiryayev.
\newblock {\em Theory of martingales}, volume~49 of {\em Mathematics and its
  Applications (Soviet Series)}.
\newblock Kluwer Academic Publishers Group, Dordrecht, 1989.
\newblock Translated from the Russian by K. Dzjaparidze [Kacha Dzhaparidze].

\bibitem[MB07]{Me07}
Thilo Meyer-Brandis.
\newblock Stochastic {F}eynman-{K}ac equations associated to {L}{\'e}vy-{I}t\^o
  diffusions.
\newblock {\em Stoch. Anal. Appl.}, 25(5):913--932, 2007.

\bibitem[Mey76]{Me76}
P.~A. Meyer.
\newblock La th{\'e}orie de la pr{\'e}diction de {F}. {K}night.
\newblock In {\em S{\'e}minaire de {P}robabilit{\'e}s, {X} ({P}remi{\`e}re
  partie, {U}niv. {S}trasbourg, {S}trasbourg, ann{\'e}e universitaire
  1974/1975)}, pages 86--103. Lecture Notes in Math., Vol. 511. Springer,
  Berlin, 1976.

\bibitem[Mik83]{Mi83}
R.~Mikulevi{\v{c}}ius.
\newblock Properties of solutions of stochastic differential equations.
\newblock {\em Litovsk. Mat. Sb.}, 23(4):18--31, 1983.

\bibitem[Mik00]{Mi00}
R.~Mikulevicius.
\newblock On the {C}auchy problem for parabolic {SPDE}s in {H}{\"o}lder
  classes.
\newblock {\em Ann. Probab.}, 28(1):74--103, 2000.

\bibitem[MP09]{MiPr09a}
R.~Mikulevicius and H.~Pragarauskas.
\newblock On {H}{\"o}lder solutions of the integro-differential {Z}akai
  equation.
\newblock {\em Stochastic Process. Appl.}, 119(10):3319--3355, 2009.

\bibitem[MP76]{MePi75}
Michel M{{\'e}}tivier and Giovanni Pistone.
\newblock Une formule d'isom{\'e}trie pour l'int{\'e}grale stochastique
  hilbertienne et {\'e}quations d'{\'e}volution lin{\'e}aires stochastiques.
\newblock {\em Z. Wahrscheinlichkeitstheorie und Verw. Gebiete}, 33(1):1--18,
  1975/76.

\bibitem[Par72]{Pa72}
{{\'E}}tienne Pardoux.
\newblock Sur des {\'e}quations aux d{\'e}riv{\'e}es partielles stochastiques
  monotones.
\newblock {\em C. R. Acad. Sci. Paris S{\'e}r. A-B}, 275:A101--A103, 1972.

\bibitem[Par75]{Pa75}
{{\'E}}tienne Pardoux.
\newblock \'{E}quations aux d{\'e}riv{\'e}es partielles stochastiques de type
  monotone.
\newblock In {\em S{\'e}minaire sur les \'{E}quations aux {D}{\'e}riv{\'e}es
  {P}artielles (1974--1975), {III}, {E}xp. {N}o. 2}, page~10. Coll{\`e}ge de
  France, Paris, 1975.

\bibitem[Pri12]{Pr12}
Enrico Priola.
\newblock Pathwise uniqueness for singular {SDE}s driven by stable processes.
\newblock {\em Osaka J. Math.}, 49(2):421--447, 2012.

\bibitem[Pri14]{Pr14}
Enrico Priola.
\newblock Stochastic flow for {SDE}s with jumps and irregular drift term.
\newblock {\em arXiv preprint arXiv:1405.2575}, 2014.

\bibitem[Pro05]{Pr05}
Philip~E. Protter.
\newblock {\em Stochastic integration and differential equations}, volume~21 of
  {\em Stochastic Modelling and Applied Probability}.
\newblock Springer-Verlag, Berlin, 2005.
\newblock Second edition. Version 2.1, Corrected third printing.

\bibitem[PZ07]{PeZa07}
S.~Peszat and J.~Zabczyk.
\newblock {\em Stochastic partial differential equations with {L}{\'e}vy
  noise}, volume 113 of {\em Encyclopedia of Mathematics and its Applications}.
\newblock Cambridge University Press, Cambridge, 2007.
\newblock An evolution equation approach.

\bibitem[Roz90]{Ro90}
B.~L. Rozovski{\u\i}.
\newblock {\em Stochastic evolution systems}, volume~35 of {\em Mathematics and
  its Applications (Soviet Series)}.
\newblock Kluwer Academic Publishers Group, Dordrecht, 1990.
\newblock Linear theory and applications to nonlinear filtering, Translated
  from the Russian by A. Yarkho.

\bibitem[RZ07]{RoZh07}
Michael R{{\"o}}ckner and Tusheng Zhang.
\newblock Stochastic evolution equations of jump type: existence, uniqueness
  and large deviation principles.
\newblock {\em Potential Anal.}, 26(3):255--279, 2007.

\bibitem[Tin77]{Ti77b}
E.~Tinfavi{\v c}ius.
\newblock Linearized stochastic equations of nonlinear filtering of random
  processes.
\newblock {\em Lithuanian Mathematical Journal}, 17(3):321--334, 1977-07-01
  1977.

\bibitem[Ver12]{Ve12}
Mark Veraar.
\newblock The stochastic {F}ubini theorem revisited.
\newblock {\em Stochastics}, 84(4):543--551, 2012.

\bibitem[Wal86]{Wa86}
John~B. Walsh.
\newblock An introduction to stochastic partial differential equations.
\newblock In {\em \'{E}cole d'{\'e}t{\'e} de probabilit{\'e}s de
  {S}aint-{F}lour, {XIV}---1984}, volume 1180 of {\em Lecture Notes in Math.},
  pages 265--439. Springer, Berlin, 1986.

\bibitem[Zha13]{Zh13b}
Xicheng Zhang.
\newblock Degenerate irregular {SDE}s with jumps and application to
  integro-differential equations of {F}okker-{P}lanck type.
\newblock {\em Electron. J. Probab.}, 18:no. 55, 25, 2013.

\bibitem[Zho13]{Zh13f}
Guoli Zhou.
\newblock Global well-posedness of a class of stochastic equations with jumps.
\newblock {\em Adv. Difference Equ.}, pages 2013:175, 23, 2013.

\end{thebibliography}

\end{document}